\documentclass[11pt]{article}

\usepackage{amsmath}
\usepackage{amsthm}
\usepackage{amsfonts}
\usepackage{amssymb}
\usepackage{amscd}

\newtheorem{thm}{Theorem}[section]
\newtheorem{cor}[thm]{Corollary}

\newtheorem{lem}[thm]{Lemma}
\newtheorem{prop}[thm]{Proposition}
\newtheorem{rem}[thm]{Remark}
\newtheorem{exam}[thm]{Example}
\newcommand{\rhou}{\textrm{\raisebox{0.6mm}{$\rho $}}}
\newcommand{\chiu}{\textrm{\raisebox{0.6mm}{$\chi $}}}

\numberwithin{equation}{section}

\begin{document}

\title{\bf \Large Polynomial invariants for a semisimple and cosemisimple Hopf algebra of finite dimension 
\centerline{\small \it Dedicated to Professor Noriaki Kawanaka on the occasion of his 60th birthday}
} 
\author{Michihisa Wakui \thanks{E-mail address: wakui@ipcku.kansai-u.ac.jp} \\ \centerline{\small Department of Mathematics, Faculty of Engineering Science, Kansai University,} \\ 
\centerline{\small Suita-shi, Osaka 564-8680, Japan}}
\date{preliminary version,\ June 28, 2009}

\maketitle 
\begin{abstract}
We introduce new polynomial invariants of a finite-dimensional semisimple and cosemisimple Hopf algebra $A$ over a field $\boldsymbol{k}$ by using the braiding structures of $A$. 
We investigate basic properties of the polynomial invariants including stability under extension of the base field. 
Furthermore, we show that our polynomial invariants are indeed tensor invariants of the representation category of $A$, and recognize the difference of the representation category and the representation ring of $A$. 
Actually, by computing and comparing polynomial invariants, we find new examples of pairs of Hopf algebras whose representation rings are isomorphic, but representation categories are distinct. 
\end{abstract}

{\small {\bf Mathematics Subject Classifications (2000):} 16W30, 18D10, 19A49}

\baselineskip 16pt

\section{Introduction}
In representation theory of Hopf algebras over a field $\boldsymbol{k}$ it is a fundamental problem to know conditions for that the representation categories of two given Hopf algebras are equivalent as (abstract) $\boldsymbol{k}$-linear monoidal categories. 
A complete answer for this is given by Schauenburg~\cite{Sc2, Sc4}. 
He introduced a notion of bi-Galois extensions, 
and showed that the monoidal equivalences of comodule categories over Hopf algebras are classified  by bi-Galois extensions of the base field $\boldsymbol{k}$. 
If the Hopf algebra $A$ is of finite dimension, then this is equivalent to that $A$ and  $B$ are cocycle deformations of each other, which are introduced by Doi~\cite{Doi}. 
Many researchers have been successful to determine the bi-Galois objects and the  cocycle deformations for various special families of Hopf algebras, e.g. \cite{Mas0, DT, Sc3, Mas2}. 
However, it is very difficult to do so in general. 
\par 
In this paper we introduce a new family of invariants of a semisimple and cosemisimple Hopf algebra of finite dimension by using the braiding structures of it, 
and show that our invariants are useful for examining whether the representation categories of two such Hopf algebras are monoidal equivalent or not. 
\par 
The basic idea of our method is to utilize quantum invariants of low-dimensional manifolds, that are topological invariants defined by using quantum groups, namely, Hopf algebras with braiding structures. 
In contrast to most of current investigations on quantum invariants in which 
topological problems of low-dimensional manifolds are studied under a fixed Hopf algebra, in this  research, we fix a framed knot or link, and study on representation categories of Hopf algebras. 
In particular, in this paper, by use of quantum invariants of the unknot with $(+1)$-framing we introduce polynomials $P_A^{(d)}(x)\ (d=1,2,\ldots )$ as invariants of  a finite-dimensional semisimple and cosemisimple Hopf algebra $A$ over $\boldsymbol{k}$. 
For each positive integer $d$ the polynomial $P_A^{(d)}(x)$ is defined by 
$$P_A^{(d)}(x)=\prod\limits_{i=1}^t\prod\limits_{R\textrm{\,:\,braidings of $A$}}\Bigl( x-\frac{\underline{\dim }_R M_i}{\dim M_i}\Bigr)\ \ \in \ \ \boldsymbol{k}[x],$$
where $\{ M_1,\ldots ,M_t\} $ is a full set of non-isomorphic absolutely simple left $A$-modules of dimension $d$ (so, $\dim M_i=d$ for all $i$), and $\underline{\dim }_R M_i\in \boldsymbol{k}$ is the quantum invariant of unknot with $(+1)$-framing and colored by $M_i$. 
In algebraic language, $\underline{\dim }_R M_i$ is the category-theoretic rank of $M_i$ in the left  rigid braided monoidal category $({}_A\mathbb{M}^{\textrm{f.d.}}, c_R)$ \cite{Maj}, 
where ${}_A\mathbb{M}^{\textrm{f.d.}}$ is the monoidal category of finite-dimensional left $A$-modules and $A$-linear maps, and $c_R$ is the braiding of ${}_A\mathbb{M}^{\textrm{f.d.}}$ determined by $R$. 
\par 
Provided that the polynomial $P_A^{(d)}(x)$ is not a constant, 
 all roots of $P_A^{(d)}(x)$ are $n$-th roots of unity for some positive integer $n$. 
Furthermore the polynomial has nice properties such as integer property and  stability under extension of the base field as follows.
All coefficients of the polynomial are integers if $\boldsymbol{k}$ is a finite Galois extension of the rational number field $\mathbb{Q}$, and $A$ is a scalar extension of some finite-dimensional semisimple Hopf algebra over $\mathbb{Q}$. 
The polynomial is stabilized by taking some suitable extension of the base field, more precisely, there is a finite separable field extension $L/\boldsymbol{k}$ so that  $P_{A^E}^{(d)}(x)=P_{A^L}^{(d)}(x)$ for any field extension $E/L$. 
\par 
It is more interesting to note that our polynomial invariants are indeed invariants of representation categories of Hopf algebras, and recognize the difference of  representation categories and representation rings of that. 
In general, if representation categories of two finite-dimensional semisimple Hopf algebras are equivalent as monoidal categories, then their representation rings are isomorphic. 
However, the converse is not true. 
For example, by Tambara and Yamagami~\cite{TY}, and also Masuoka~\cite{Mas2}, it has proved that if the characteristic of $\boldsymbol{k}$ is $0$ or $p>2$, then 
three non-commutative and semisimple Hopf algebras $\boldsymbol{k}[D_8], \boldsymbol{k}[Q_8], K_8$ of dimension $8$ have the same representation ring, but their representation categories are not mutually equivalent, where $D_8$ is the dihedral group of order $8$, $Q_8$ is the quaternion group, and $K_8$ is the Kac-Paljutkin algebra \cite{KP, Mas1}. 
This result is confirmed by our polynomial invariants, again. 
Moreover, by computing and comparing polynomial invariants we find new examples of pairs of Hopf algebras, whose representation rings are the same, but representation categories are distinct. 
\par 
This paper consists of six sections in total, and they are divided into two parts except this introduction:  
from Section 2 to Section 4 the definition and general properties of the polynomial invariants are discussed, and from Section 5 to the final section several concrete examples are computed, and applications are described. 
Detailed contents are as follows. 
In Section~2 we introduce the definition of our polynomial invariants of a semisimple and cosemisimple Hopf algebra of finite dimension. It is proved that the polynomial invariants are indeed invariants of the representation category of such a Hopf algebra. 
In Section~3 some basic properties of polynomial invariants are studied. 
It is shown that the polynomial invariants have nice properties such as integer property and  stability under extension of the base field. 
In Section~4 by dualizing the method of construction of our polynomial invariants, we state a formula to compute them in terms of coalgebraic and comodule-categorical language. 
In Section~5 we demonstrate computations of polynomial invariants for several Hopf algebras including the Hopf algebras $A_{Nn}^{+\lambda }$ ($N$ is odd, and $\lambda =\pm 1$), that are introduced by Satoshi Suzuki \cite{Suz}, and 
by comparing them we re-prove the Tambara, Yamagami and Masuoka's result as previously mentioned, and also find some pairs of Hopf algebras, whose representation rings are isomorphic, but representation categories are distinct. 
In the final section as an appendix we determine the structures of representation rings of the Hopf algebras $A_{Nn}^{+\lambda }$, and determine when they are self-dual, which are used in Section 5. 
\par 
Throughout this paper, we use the notation $\otimes $ instead of $\otimes _{\boldsymbol{k}}$, and denote by $\textrm{ch}(\boldsymbol{k})$ the characteristic of the field $\boldsymbol{k}$. 
For a Hopf algebra $A$, denoted by $\Delta $, $\varepsilon $ and $S$ are 
the coproduct, the counit, and the antipode of $A$,  respectively, and 
$G(A)$ is the group consisting of the group-like elements in $A$, and $A^{\textrm{cop}}$ is the resulting Hopf algebra from $A$ by replacing $\Delta$ by the  opposite coproduct $\Delta ^{\textrm{cop}}$. 
We write ${}_A\mathbb{M}$ for the $\boldsymbol{k}$-linear monoidal category whose objects are left $A$-modules and morphisms are left $A$-linear maps, and 
write ${}^A\mathbb{M}$ for the $\boldsymbol{k}$-linear monoidal category whose objects are left $A$-comodules and morphisms are left $A$-colinear maps. 
For general references on Hopf algebras we refer to Abe's book \cite{A}, Montgomery's book \cite{Mon} and Sweedler's book \cite{Sw}. 
For general references on monoidal categories we refer to MacLane's book \cite{Mac} and Joyal and Street's paper \cite{JS}. 

\section{Definition of polynomial invariants}
\par 
In this section we introduce a new family of invariants of a semisimple and cosemisimple Hopf algebra of finite dimension over an arbitrary field. 
They are given by polynomials derived from the quasitriangular  structures of the Hopf algebra. 
By the method of construction of the polynomials they also become invariants under monoidal equivalence of representation categories of Hopf algebras. 
\par 
Let us recall the definition of a quasitriangular Hopf algebra \cite{Dri}. 
Let $A$ be a Hopf algebra over a field $\boldsymbol{k}$, and $R \in A\otimes A$ an invertible element. 
The pair $(A, R)$ is said to be a \textit{quasitriangular Hopf algebra}, and $R$ is said to be a \textit{universal $R$-matrix} of $A$, if the following three conditions are satisfied:
\begin{enumerate}
\item[$\bullet$] $\Delta ^{\textrm{cop}}(a)=R\cdot \Delta (a)\cdot R^{-1} \ \ \mbox{ for all } a\in A$, 
\item[$\bullet$] $(\Delta \otimes \textrm{id})(R)=R_{13}R_{23}$, 
\item[$\bullet$] $(\textrm{id}\otimes \Delta )(R)=R_{13}R_{12}$. 
\end{enumerate}
Here $\Delta ^{\textrm{cop}}=T\circ \Delta $, $T: A\otimes A\longrightarrow A\otimes A, \ T(a\otimes b)=b\otimes a$, and $R_{ij}\in A\otimes A\otimes A$ is given by $R_{12}=R\otimes 1, \ R_{23}=1\otimes R,\ R_{13}=(T\otimes \textrm{id})(R_{23})=(\textrm{id}\otimes T)(R_{12})$. 
\par 
If $R=\sum _{i}\alpha _i\otimes \beta _i$ is a universal $R$-matrix of $A$, then   the element $u=\sum _{i}S(\beta _i)\alpha _i$ of $A$ is invertible, and it 
has the following properties: 
\begin{enumerate}
\item[(i)] $S^2(a)=uau^{-1}$ for all $a\in A$, 
\item[(ii)] $S(u)=\sum_{i}\alpha _iS(\beta_i)$. 
\end{enumerate}

The above element $u$ is called the \textit{Drinfel'd element} associated to $R$. 
If $A$ is semisimple and cosemisimple of finite dimension, then the Drinfel'd element $u$ belongs to the center of $A$ by the property (i) and $S^2=\textrm{id}_A$ \cite[Corollary 3.2(i)]{EG2}. 
\par 
Let $(A, R)$ be a quasitriangular Hopf algebra over a field $\boldsymbol{k}$ and $u$ the Drinfel'd element associated to $R$. 
For a finite-dimensional left $A$-module $M$, we denote by $\underline{\dim}_R\kern0.1em M$ the trace of the left action of $u$ on $M$, and call it the   \textit{$R$-dimension} of $M$. 
The $R$-dimension $\underline{\dim}_R\kern0.1em M$ is a special case of the braided dimension of $M$ in the left rigid braided monoidal category $({}_A\mathbb{M}^{\text{f.d.}}, c_R)$ (see Section 4 for the definition of braided dimensions). 

\par  
To define polynomial invariants, we  use the following result established by Etingof and Gelaki \cite{EG2}. 

\par \smallskip 
\begin{thm}[{\bf Etingof and Gelaki}]\label{2.1}
Let $A$ be a cosemisimple Hopf algebra of finite dimension over a  field $\boldsymbol{k}$. Then 
\par 
(1) the set of universal $R$-matrices $\underline{\rm{Braid}}(A)$ is finite, 
\par 
(2) provided that $A$ is semisimple, $(\dim M)1_{\boldsymbol{k}}\not= 0$ for any absolutely simple left $A$-module $M$.  
\end{thm}

\par \smallskip 
\begin{rem}\rm 
The proof of Part (1) was given in \cite[Corollary 1.5]{EG2} (see also Remark~\ref{3.12} in Section 3). 
In an extra case such as characteristic $0$ or positive characteristic with some additional assumptions, it was proved by Radford~\cite[Theorem 1]{Ra2}. 
The proof of Part (2) was given in \cite[Corollary 3.2(ii)]{EG2}. 
The Etingof and Gelaki's proof is based on a Larson's result \cite[Theorem 2.8]{L}, that is the same as Part (2) with the assumption $S^2=\textrm{id}_A$.  
\end{rem}

\par \smallskip 
Let $A$ be a semisimple and cosemisimple Hopf algebra of finite dimension over a  field $\boldsymbol{k}$. 
For a finite-dimensional left $A$-module $M$ with 
$(\dim M)1_{\boldsymbol{k}}\not= 0$, we have a polynomial 
$$P_{A,M}(x):=\prod\limits_{R\in \underline{\textrm{Braid}}(A)}\Bigl( x-\frac{\underline{\dim }_R M}{\dim M}\Bigr)\ \ \in \ \ \boldsymbol{k}[x].$$

Furthermore, for each positive integer $d$ a polynomial $P_A^{(d)}(x)$ is defined by 
$$P_A^{(d)}(x):=\prod\limits_{i=1}^tP_{A,M_i}(x)\ \ \in \ \ \boldsymbol{k}[x],$$
where $\{ M_1,\ldots ,M_t\} $ is a full set of non-isomorphic absolutely simple left $A$-modules of dimension $d$.  
Here, if there is no absolutely simple left $A$-module of dimension $d$, then we set  
$P_A^{(d)}(x)=1$. 

\par \smallskip 
\begin{exam}\label{2.3} \rm 
Let $G=C_m$ be the cyclic group of order $m$ generated by $g$, and let $\boldsymbol{k}$ be a field whose characteristic does not divide $m$. 
Suppose that $\boldsymbol{k}$ contains a primitive $m$-th root of unity $\omega $.  
Then any universal $R$-matrix of the group Hopf algebra $\boldsymbol{k}[C_m]$ is given by 
\begin{equation}\label{eq2.3}
R_d=\sum\limits_{i,j=0}^{m-1} \omega ^{dij}E_i\otimes E_j\qquad (d=0,1,\ldots , m-1),
\end{equation}
where 
$E_i=\frac{1}{m}\sum_{j=0}^{m-1} \omega ^{-ij} g^j$ (see \cite{Ra3} for example). 
Let $M_j=\boldsymbol{k}$ be the (absolutely) simple left $\boldsymbol{k}[C_m]$-module equipped with the action $\chi_j(g^p)=\omega ^{jp}\ (p=0,1,\cdots ,m-1)$. 
For each $d$ and $i$, then $\underline{\dim}_{R_d}M_j=\omega ^{-dj^2}$ since 
the Drinfel'd element $u_d$ of $R_d$ is given by 
$u_d=\sum_{i=0}^{m-1} \omega ^{-di^2}E_i$. 
Thus we have 
$$P_{\boldsymbol{k}[C_m]}^{(1)}(x)=\prod_{d,j=0}^{m-1}(x-\omega ^{-dj^2})=\prod_{j=0}^{m-1}(x^{\frac{m}{\textrm{gcd}\kern0.1em (j^2,m)}}-1)^{\textrm{gcd}\kern0.1em (j^2,m)}.$$
\end{exam}

\par \medskip 
Two Hopf algebras $A$ and $B$ over $\boldsymbol{k}$ are said to be \textit{monoidally Morita equivalent} 
if monoidal categories ${}_A\mathbb{M}$ and ${}_B\mathbb{M}$ are equivalent as $\boldsymbol{k}$-linear monoidal categories. 

\par \smallskip 
\begin{lem}\label{2.2}
Let $A$ and $B$ be Hopf algebras of finite dimension over $\boldsymbol{k}$.  
If a $\boldsymbol{k}$-linear monoidal functor 
$F:{}_A\mathbb{M}\longrightarrow {}_B\mathbb{M}$ is an equivalence between  monoidal categories, 
then $\dim M=\dim F(M)$ for any finite-dimensional left $A$-module $M$. 
\end{lem}
\begin{proof}
We set $H:=A^{\ast\textrm{cop}},\ L:=B^{\ast\textrm{cop}}$.  
Then the monoidal categories ${}_A\mathbb{M}$ and ${}_B\mathbb{M}$ can be identified with the monoidal categories ${}^{H}\mathbb{M},\ {}^{L}\mathbb{M}$, respectively, and $F$ can be regarded  as a $\boldsymbol{k}$-linear monoidal functor from ${}^{H}\mathbb{M}$ to ${}^{L}\mathbb{M}$. 
By a Schauenburg's result \cite[Corollary 5.7]{Sc2}, there is an $(L,H)$-bi-Galois object $X$ and a monoidal natural equivalence $\xi : F\Longrightarrow F_0$, where $F_0$ is the $\boldsymbol{k}$-linear monoidal functor defined by 
$$F_0: {}^{H}\mathbb{M}\longrightarrow {}^{L}\mathbb{M},\quad M\longmapsto X\square _H M,$$
and $\square _H$ denotes the cotensor product over $H$. 
Thus for a finite-dimensional left $A$-module $M$, 
we have a $\boldsymbol{k}$-linear isomorphism 
\begin{align*}
X\otimes F(M)\cong X\otimes (X\square _H M)&\cong (X\otimes X)\square _HM \\  
&\cong (X\otimes H)\square _H M
\cong X\otimes (H\square _H M)\cong X\otimes M.
\end{align*}
This implies that $\dim F(M)=\dim M$. 
\end{proof}

\par \smallskip 
\begin{lem}\label{2.4}
Let $A$ and $B$ be Hopf algebras of finite dimension over $\boldsymbol{k}$.  
If a $\boldsymbol{k}$-linear monoidal functor 
$(F,\phi ,\omega ):{}_A\mathbb{M}\longrightarrow {}_B\mathbb{M}$ is an equivalence between monoidal categories, 
then there is a bijection $\Phi : \underline{\rm{Braid}}(A)\longrightarrow \underline{\rm{Braid}}(B)$ such that for a finite-dimensional left $A$-module $M$ and a universal $R$-matrix $R\in \underline{\rm{Braid}}(A)$, 
$$\underline{\dim}_R M=\underline{\dim}_{\Phi (R)}F(M).$$
\end{lem}
\begin{proof}
Let $(G,\phi',\omega ' ):{}_B\mathbb{M}\longrightarrow {}_A\mathbb{M}$ be a quasi-inverse of $(F,\phi ,\omega )$. 
Then there are $\boldsymbol{k}$-linear monoidal natural transformations 
$\varphi :(F,\phi ,\omega )\circ (G,\phi',\omega ' )\Longrightarrow 1_{{}_B\mathbb{M}}$ and $\psi :(G,\phi ', \omega ')\circ (F,\phi ,\omega )\Longrightarrow 1_{{}_A\mathbb{M}}$, where $1_{\mathcal{\nu }}$ stands for the identity functor on $\mathcal{\nu }={}_A\mathbb{M}, {}_B\mathbb{M}$. 
\par 
A universal $R$-matrix $R=\sum_{i}\alpha _i\otimes \beta _i$ of $A$ defines a braiding $c=\{ c_{M,N}:M\otimes N\longrightarrow N\otimes M\} _{M,N\in {}_A\mathbb{M}}$ consisting of $A$-linear isomorphisms 
$$c_{M,N}(m\otimes n)=\sum\limits_{i}\beta _in\otimes \alpha _im\quad (m\in M,\ n\in N).$$ 
The braiding $c$ gives rise to a braiding $c^{\prime}$ of ${}_B\mathbb{M}$,  which consists of $B$-linear isomorphisms 
$c_{P,Q}^{\prime}:P\otimes Q\longrightarrow Q\otimes P\ (P,Q\in {}_B\mathbb{M})$ such that the following diagram commutes. 
\begin{equation*}
\begin{CD}
P\otimes Q @>\qquad c_{P,Q}'\qquad >> Q\otimes P \\ 
@A\varphi (P)\otimes \varphi (Q)AA  @AA\varphi (Q)\otimes \varphi (P)A \\ 
FG(P)\otimes FG(Q) @. FG(Q)\otimes FG(P) \\ 
@V\phi _{G(P), G(Q)}VV  @VV\phi _{G(Q), G(P)}V \\ 
F(G(P)\otimes G(Q)) @>F(c_{G(P),G(Q)})>> F(G(Q)\otimes G(P))
\end{CD}
\end{equation*}

\par \medskip 
Then $\Phi (R):=(T\circ c^{\prime}_{B,B})(1)$ is a universal $R$ -matrix of $B$, 
where $T:B\otimes B\longrightarrow B\otimes B$ is defined by $T(a\otimes b)=b\otimes a\ (a,b\in B)$. 
It is easy to see from definition that the map $\Phi :\underline{\textrm{Braid}}(A)\longrightarrow \underline{\textrm{Braid}}(B)$ defined as above is bijective. 
\par 
Let $M$ be a finite-dimensional left $A$-module, and $e_M$ and $n_M$ the evaluation and coevaluation morphisms defined by 
\begin{align*}
e_M:M^{\ast}\otimes M\longrightarrow \boldsymbol{k},& \quad 
e_M(f\otimes m)=f(m)\quad (f\in M^{\ast}, m\in M),\\ 
n_M: \boldsymbol{k}\longrightarrow M\otimes M^{\ast},& \quad 
n_M(1)=\sum_i e_i\otimes e_i^{\ast} \quad (\textrm{the canonical element}). 
\end{align*}

Then $e_M\circ c_{M,M^{\ast}}\circ n_M=(\underline{\dim}_RM)\textrm{id}_{\boldsymbol{k}}$. 
So, we may identify $\underline{\dim}_RM=e_M\circ c_{M,M^{\ast}}\circ n_M$. 
We set $e^{\prime}_{F(M)}:=\omega ^{-1}\circ F(e_M)\circ \phi _{M^{\ast}, M}$ and $n^{\prime}_{F(M)}:=\phi ^{-1}_{M,M^{\ast}}\circ F(n_M)\circ \omega $. 
Then $(F(M^{\ast}), e^{\prime}_{F(M)}, n^{\prime}_{F(M)})$ is a left dual for $F(M)$. 
Since the $\boldsymbol{k}$-linear monoidal functor 
$(F,\phi ,\omega )$ becomes a braided monoidal functor from $({}_A\mathbb{M}, c)$ to $({}_B\mathbb{M}, c')$, it follows that 
\begin{align*}
\underline{\dim}_{\Phi (R)}F(M)
&=e^{\prime}_{F(M)}\circ c^{\prime}_{F(M), F(M^{\ast})}\circ n^{\prime}_{F(M)}\\ 
&=\omega ^{-1}\circ F(e_M\circ c_{M,M^{\ast}}\circ n_M)\circ \omega \\ 
&=(\underline{\dim}_R M)\kern0.2em \omega ^{-1}\circ F(\textrm{id}_{\boldsymbol{k}})\circ \omega \\ 
&=\underline{\dim}_R M . 
\end{align*}
\end{proof}

\par 
\begin{thm} \label{2.5} 
Let $A$ and $B$ be semisimple and cosemisimple Hopf algebras of finite dimension over $\boldsymbol{k}$. 
If $A$ and $B$ are monoidally Morita equivalent, then $P_A^{(d)}(x)=P_B^{(d)}(x)$ for any positive integer $d$. 
\end{thm}
\begin{proof}
Let $F:{}_A\mathbb{M}\longrightarrow {}_B\mathbb{M}$ be a 
$\boldsymbol{k}$-linear monoidal functor which gives an equivalence of monoidal categories, and let us consider the bijection $\Phi : \underline{\textrm{Braid}}(A)\longrightarrow \underline{\textrm{Braid}}(B)$ given as in the proof of Lemma \ref{2.2}. 
\par 
Let $M$ be an absolutely simple left $A$-module. 
Then $F(M)$ is also an absolutely simple left $B$-module, and by Lemma \ref{2.2} and Lemma~\ref{2.4} we have 
\begin{equation}\label{eq2.4}
P_{A,M}(x)=P_{B,F(M)}(x). 
\end{equation}

Let $\{ M_1,\ldots ,M_t\} $ be a full set of non-isomorphic absolutely simple left $A$-modules of dimension $d$.  
Then $\{ F(M_1),\ldots , F(M_t)\} $ is also a full set of non-isomorphic absolutely simple left $B$-modules of dimension $d$. 
Applying the equation (\ref{eq2.4}) to $M=M_i\ (i=1,\ldots ,t)$, and taking the product of them, we have $P_{A}^{(d)}(x)=P_{B}^{(d)}(x)$. 
\end{proof}

\par \smallskip 
\begin{rem}\rm 
Our polynomial invariants are useful only if a semisimple and cosemisimple Hopf algebra has a quasitriangular structure. 
However, by considering the polynomial invariants of the Drinfeld double of it  
we have monoidal invariants of the original (arbitrary) semisimple Hopf algebra of finite dimension. 
\end{rem}

\par \medskip 
\section{Properties of polynomial invariants}
\par 
In this section we investigate basic properties of the polynomial invariants  $P_A^{(d)}(x)\ (d=1,2,\cdots )$ defined in Section 2. 
This section consists of two subsections. 
In the subsection~3.1 we prove that all coefficients of $P_A^{(d)}(x)$ are integers if $\boldsymbol{k}$ is a finite Galois extension of $\mathbb{Q}$, and $A$ is a scalar extension of some finite-dimensional semisimple Hopf algebra over $\mathbb{Q}$. In the subsection~3.2  we prove that $P_A^{(d)}(x)$ is stabilized by taking some suitable extension of the base field, more precisely, there is a finite separable field extension $L/\boldsymbol{k}$ so that  $P_{A^E}^{(d)}(x)=P_{A^L}^{(d)}(x)$ for any field extension $E/L$. 

\par \smallskip 
\subsection{Integer property of polynomial invariants} 
First of all, we show that the coefficients of $P_A^{(d)}(x)$ lie in the integral closure of the prime ring of the base field of $A$. 
To do this we need the following lemma. 

\par \smallskip 
\begin{lem}\label{3.1}
Let $(A, R)$ be a quasitriangular Hopf algebra over $\boldsymbol{k}$ and $u$ the Drinfel'd element associated to $R$. 
If $A$ is semisimple and cosemisimple, then $u^{(\dim A)^3}=1$. 
\end{lem}
\begin{proof}
Let us consider the following subHopf algebras of $A$ : 
$$
B=\{ \ (\alpha \otimes \textrm{id})(R)\ \vert \ \alpha \in A^{\ast} \ \} ,\quad 
H=\{ \ (\textrm{id}\otimes \alpha )(R)\ \vert \ \alpha \in A^{\ast} \ \}. 
$$
By \cite[Proposition 2]{Ra4}, the Hopf algebra $B$ is isomorphic to the Hopf algebra $H^{\ast \textrm{cop}}$. 
Let $(D(H), \mathcal{R})$ be the Drinfel'd double of $H$. 
By \cite[Theorem 2]{Ra4}, there is a  homomorphism $F: (D(H),\mathcal{R})\longrightarrow (A,R)$ of quasitriangular Hopf algebras. 
It follows that the Drinfel'd element $\tilde{u}$ of $(D(H), \mathcal{R})$ satisfies $F(\tilde{u})=u$. 
Since $A$ is semisimple, subHopf algebras $H$ and  $H^{\ast \textrm{cop}}\cong B$ are also semisimple \cite[Corollary 2.5]{LR3}. 
Thus $H$ is semisimple and cosemisimple. 
So,  we have $\tilde{u}^{(\dim H)^3}=1$ by \cite[Theorem 2.5 \& Theorem 4.3]{EG1}, and $u^{(\dim H)^3}=1$. 
Since $\dim A$ is divided by $\dim H$ \cite[Proposition 2]{Ra4}, 
the equation $u^{(\dim A)^3}=1$ is obtained. 
\end{proof}

\par \smallskip 
For a field $K$, let $Z_K$ denote the integral closure of the prime ring of $K$, 
that is, if the characteristic of $K$ is $0$, then $Z_K$ is the ring of algebraic integers in $K$, and if the characteristic of $K$ is $p>0$, then $Z_K$ is the algebraic closure of the prime field  $\mathbb{F}_p$ in $K$. 

\par \smallskip 
\begin{lem}\label{3.2}
Let $(H,R)$ be a semisimple and cosemisimple quasitriangular Hopf algebra over a field $K$. 
If $M$ is an absolutely simple left $H$-module, then 
(by Theorem \ref{2.1} $(\dim M)1_K\not= 0$ and, )
$$\biggl( \frac{\underline{\dim}_R M}{\dim M}\biggr) ^{(\dim H)^3}=1.$$
In particular, 
$\frac{\underline{\dim}_R M}{\dim M}\in Z_K$. 
\end{lem} 
\begin{proof}
The Drinfel'd element $u$ of $(H, R)$ belongs to the center of $H$ since $H$ is semisimple and cosemisimple. 
Thus the left action $\underline{u}_M:M\longrightarrow M$ of $u$ is a left $H$-endomorphism. 
Since $M$ is absolutely simple, 
$\underline{u}_M$ is a scalar multiple of the identity morphism, so it can be written as 
$\underline{u}_M=\omega _M\kern0.1em \textrm{id}_M$ for some $\omega _M\in K$. 
Then by Lemma \ref{3.1} we have 
$$\dim M
=\textrm{Tr}(\underline{u}_M^{(\dim H)^3})
=\omega _M^{(\dim H)^3}\textrm{Tr}(\textrm{id}_M)
=\omega _M^{(\dim H)^3}\dim M.$$
Thus $\omega _M^{(\dim H)^3}=1$. 
This implies that 
$\omega _M=\frac{\underline{\dim}_R M}{\dim M}$ belongs to $Z_K$. 
\end{proof}

\par \smallskip 
Form Lemma \ref{3.2}, we have the following immediately. 

\par \smallskip 
\begin{prop}\label{3.3}
Let $H$ be a semisimple and cosemisimple Hopf algebra of finite dimension over a  field $K$. 
Then for any absolutely simple left $H$-module $M$,  the coefficients of the polynomial 
$P_{H,M}(x)$ are in $Z_K$. 
Therefore, $P_H^{(d)}(x)\in Z_K[x]$ for any positive integer $d$.  
\end{prop} 

\par \smallskip 
Now, we will examine relationship between polynomial invariants and Galois extensions of fields. 
Let $K/\boldsymbol{k}$ be a field extension, and let $\textrm{Aut}(K/\boldsymbol{k})$ denote the automorphism group of $K/\boldsymbol{k}$. 
For a $K$-linear space $M$ and $\sigma \in \textrm{Aut}(K/\boldsymbol{k})$, 
a $K$-linear space ${}^{\sigma }\kern-0.2em M$ is defined as follows : 

\begin{enumerate}
\item[(i)] ${}^{\sigma }\kern-0.2em M=M$ as additive groups,  
\item[(ii)] the action  $\star$ of $K$ on ${}^{\sigma }\kern-0.2em M$ is given by 
\begin{equation}\label{eq3.1}
c\star m:=\sigma (c)\cdot m\qquad (c\in K,\ m \in M),
\end{equation}
where $\cdot $ in the right-hand side stands for the original action of $K$ on $M$.
\end{enumerate}

For a $K$-linear map $f: M\longrightarrow N$ and $\sigma \in \textrm{Aut}(K/\boldsymbol{k})$, we have 
$$f(c\star m)=f(\sigma (c)\cdot m)=\sigma (c)\cdot f(m)=c\star f(m)\qquad (c\in K,\ m\in M),$$
thus $f$ can be regarded as a $K$-linear map from ${}^{\sigma }\kern-0.2em M$ to ${}^{\sigma }\kern-0.2em N$. 
We denote by ${}^{\sigma }\kern-0.2em f$ the $K$-linear map $f: {}^{\sigma }\kern-0.2em M\longrightarrow {}^{\sigma }\kern-0.2em N$. 
\par 
The monoidal category ${}_K\mathbb{M}$ of $K$-linear spaces and $K$-linear maps has a canonical braiding, which is given by usual twist maps. 
For an automorphism $\sigma \in \textrm{Aut}(K/\boldsymbol{k})$, the functor 
\begin{equation}
{}^{\sigma }\kern-0.2em F: {}_K\mathbb{M}\longrightarrow {}_K\mathbb{M},\qquad 
M\longmapsto {}^{\sigma }\kern-0.2em M,\quad f\longmapsto  {}^{\sigma }\kern-0.2em f
\end{equation}

\noindent 
gives a $K$-linear braided monoidal functor. 
Since 
${}^{\sigma }\kern-0.2em F\circ {}^{\tau }\kern-0.2em F={}^{\sigma \tau }\kern-0.2em F$ 
for all $\sigma , \tau \in \textrm{Aut}(K/\boldsymbol{k})$, 
the functor ${}^{\sigma }\kern-0.2em F: {}_K\mathbb{M}\longrightarrow {}_K\mathbb{M}$ gives an isomorphism of  $K$-linear braided monoidal categories. 
In general, a $K$-linear braided monoidal functor
$F: {}_K\mathbb{M}\longrightarrow {}_K\mathbb{M}$ 
maps a Hopf algebra to a Hopf algebra. 
So, for a Hopf algebra $H$ over $K$ and an automorphism $\sigma \in \textrm{Aut}(K/\boldsymbol{k})$, ${}^{\sigma }\kern-0.2em H$ also a Hopf algebra over $K$. The Hopf algebra structure of ${}^{\sigma }\kern-0.2em H$ is the same as that of $H$ with the exception that 
the action of $K$ on ${}^{\sigma }\kern-0.2em H$ is given by (\ref{eq3.1}), and 
the counit $\varepsilon _{{}^{\sigma }\kern-0.2em H}$ of ${}^{\sigma }\kern-0.2em H$ is given by $\varepsilon _{{}^{\sigma }\kern-0.2em H}=\sigma ^{-1}\circ \varepsilon _H$.  

\par \bigskip 
\begin{lem}\label{3.4}
Let $K/\boldsymbol{k}$ be a field extension, and $H$ a Hopf algebra over $K$. 
Let $R\in H\otimes _KH$ be a universal $R$-matrix of $H$, and $\sigma \in \textrm{Aut}(K/\boldsymbol{k})$. Then 
\par 
(1) $R$ is also a universal $R$-matrix of ${}^{\sigma }\kern-0.2em H$. 
We write ${}^{\sigma }\kern-0.2em R$ for this universal $R$-matrix of ${}^{\sigma }\kern-0.2em H$ . 
\par 
(2) the Drinfel'd element of the quasitriangular Hopf algebra $({}^{\sigma }\kern-0.2em H, {}^{\sigma }\kern-0.2em R)$ coincides with the Drinfel'd element of $(H,R)$. 
\par 
(3) for a finite-dimensional left $H$-module $M$,  we have $\underline{\dim}_{{}^{\sigma }\kern-0.2em R}\kern0.1em {}^{\sigma }\kern-0.2em M=\sigma ^{-1}(\underline{\dim}_R\kern0.1em M)$. 
\end{lem}
\begin{proof}
Parts (1) and (2) follow from definition. 
We show Part (3). 
Let $u\in H$ be the Drinfel'd element of  $(H,R)$. 
By Part (2), $u$ is also the Drinfel'd element of  $({}^{\sigma }\kern-0.2em H, {}^{\sigma }\kern-0.2em R)$. 
Let $\{ e_i\} _{i=1}^n$ be a basis of $M$ over $K$, and 
write 
$$u\cdot  e_i=\sum\limits_{j=1}^n a_{ji} \kern0.1em e_j\qquad (a_{ji}\in K).$$
Then 
$u\cdot  e_i=\sum_{j=1}^n \sigma^{-1}(a_{ji}) \star  e_j$, and whence 
$$\underline{\dim}_{{}^{\sigma }\kern-0.2em R}\kern0.1em {}^{\sigma }\kern-0.2em M
=\sum\limits_{i=1}^n \sigma^{-1}(a_{ii})=\sigma^{-1}\Bigl( \sum\limits_{i=1}^n a_{ii} \Bigr) 
=\sigma ^{-1}(\underline{\dim}_R\kern0.1em M).$$ 
\end{proof}

For a field automorphism $\sigma :K\longrightarrow K$ and a polynomial $P(x)=c_0+c_1x+\cdots +c_mx^m\ (c_i\in K,\ i=1,\ldots ,m)$, we define $\sigma \cdot P(x)\in K[x]$ by
$$\sigma \cdot P(x):=\sigma (c_0)+\sigma (c_1)x+\cdots +\sigma (c_m)x^m.$$

\par 
\begin{lem}\label{3.5}
Let $K/\boldsymbol{k}$ be a field extension, and $H$ a semisimple and cosemisimple Hopf algebra over $K$ of finite dimension. 
If $M$ is a finite-dimensional left $H$-module such that $(\dim M)1_K\not= 0$, then for an automorphism $\sigma \in \textrm{Aut}(K/\boldsymbol{k})$ we have 
$$\sigma ^{-1}\cdot P_{H,M}(x)=P_{{}^{\sigma }\kern-0.2em H,{}^{\sigma }\kern-0.2em M}(x).$$ 
\end{lem} 
\begin{proof}
Setting $N={}^{\sigma }\kern-0.2em M$, by Lemma \ref{3.4}(3) we have 
\begin{align*}
\sigma ^{-1}\cdot P_{H,M}(x)
&=\prod\limits_{R\in \underline{\textrm{Braid}}(H)}\Bigl( x-\sigma ^{-1}\Bigl( \frac{\underline{\dim}_R M}{\dim M}\Bigr)\Bigr)  \\ 
&=\prod\limits_{R\in \underline{\textrm{Braid}}(H)}\Bigl( x-\frac{\sigma ^{-1}(\underline{\dim}_R M)}{\dim M}\Bigr)  \\ 
&=\prod\limits_{R\in \underline{\textrm{Braid}}(H)}\Bigl( x-\frac{\underline{\dim}_{\kern0.1em {}^{\sigma }\kern-0.1em R} N}{\dim N}\Bigr)  \\ 
&\underset{(\ast )}{=} P_{{}^{\sigma }\kern-0.2em H,N}(x). 
\end{align*}
Here, the last equation ($\ast$) follows from what the map $\underline{\textrm{Braid}}(H)\longrightarrow \underline{\textrm{Braid}}({}^{\sigma }\kern-0.2em H),\  R\longmapsto {}^{\sigma }\kern-0.2em R$ is bijective by Lemma \ref{3.4}(1). 
\end{proof}

Let $A$ be a Hopf algebra over a field $\boldsymbol{k}$, and $K$ an extension field  of $\boldsymbol{k}$. 
Then $A^K=A\otimes K$ becomes a Hopf algebra over $K$, and the automorphism group $\textrm{Aut}(K/\boldsymbol{k})$ acts on $A^K$ as follows:  
\begin{equation}\label{eq:action}
\sigma \cdot (a\otimes c)=a\otimes \sigma (c)\qquad (\sigma \in \textrm{Aut}(K/\boldsymbol{k}),\ a\in A,\ c\in K).
\end{equation}
We set $H=A^K$. 
Then for each $\sigma \in \textrm{Aut}(K/\boldsymbol{k})$ the left action on $H$ given by (\ref{eq:action}) defines a Hopf algebra isomorphism from $H$ to ${}^{\sigma } \kern-0.2em H$, denoted by $\tilde{\sigma }: H\longrightarrow {}^{\sigma } \kern-0.2em H$. 
Furthermore, we see that if $R$ is a universal $R$-matrix of $H$, then $\tilde{\sigma }$ becomes a homomorphism of quasitriangular Hopf algebras from $(H, R)$ to $({}^{\sigma }\kern-0.2em H, {}^{\sigma }\kern-0.2em R)$. 
Thus  $\textrm{Aut}(K/\boldsymbol{k})$ acts on $\underline{\textrm{Braid}}(H)$ from the right by 
$$R\in \underline{\textrm{Braid}}(H)\longmapsto (\tilde{\sigma }^{-1}\otimes _K \tilde{\sigma}^{-1})({}^{\sigma }\kern-0.2em R)\in \underline{\textrm{Braid}}(H).$$

\par \medskip 
\begin{thm}\label{3.6}
Let $K/\boldsymbol{k}$ be a finite Galois extension of fields, and 
$A$ a semisimple and cosemisimple Hopf algebra over $\boldsymbol{k}$ of finite dimension. 
Then $P_{A^K}^{(d)}(x)\in (\boldsymbol{k}\cap Z_K)[x]$ for each positive integer $d$. 
\end{thm} 
\begin{proof}
First of all, let us check to see that the Hopf algebra $H=A^K$  is semisimple and cosemisimple. 
Since the Hopf algebra $A$ is semisimple, it is separable (see \cite[Corollary 2.2.2]{Mon}). 
Thus $H$ is a semisimple Hopf algebra over $K$ of finite dimension. 
Applying the same argument to the dual Hopf algebra $A^{\ast}$, we see that 
$H$ is a cosemisimple Hopf algebra over $K$. 
Eventually, we see that $H$ is semisimple and cosemisimple. 
\par 
Let $\textrm{Irr}^{(d)}_0(H)$ denote the set of isomorphism classes $[M]$ of absolutely simple left $H$-modules $M$ of dimension $d$, and set $t=\sharp \textrm{Irr}_0^{(d)}(H)$. 
If $t=0$, then $P_{H}^{(d)}(x)=1$, and hence $P_{H}^{(d)}(x)\in (\boldsymbol{k}\cap Z_K)[x]$. 
\par 
Hereinafter, we consider the case of $t>0$. 
For an automorphism $\sigma \in \textrm{Gal}(K/\boldsymbol{k})$, the map 
$$\textrm{Irr}^{(d)}_0(H)\longrightarrow \textrm{Irr}^{(d)}_0({}^{\sigma }\kern-0.2em H),\qquad [M] \longmapsto [{}^{\sigma }\kern-0.2em M]$$ 
is bijective. 
Here, by Lemma~\ref{3.5} we have 
$\sigma ^{-1}\cdot P_H^{(d)}(x)=P_{{}^{\sigma }\kern-0.2em H}^{(d)}(x)=P_H^{(d)}(x)$, and we see that $P_H^{(d)}(x)\in \boldsymbol{k}[x]$. 
On the other hand, since $P_H^{(d)}(x)\in Z_K[x]$ by Proposition~\ref{3.3}, it follows that $P_H^{(d)}(x)\in (\boldsymbol{k}\cap Z_K)[x]$. 
\end{proof}

\par \bigskip 
As applications of the above theorem we have two corollaries. 

\par 
\begin{cor}\label{3.7}
Let $K$ be a finite Galois extension field of $\mathbb{Q}$, and $A$ a semisimple Hopf algebra over $\mathbb{Q}$ of finite dimension. 
Then $P_{A^K}^{(d)}(x)\in \mathbb{Z}[x]$ for any positive integer $d$, where $\mathbb{Z}$ denotes the rational integral ring. 
\end{cor} 
\begin{proof} 
By \cite{LR1} a finite-dimensional semisimple Hopf algebra over a field of characteristic $0$ is cosemisimple. 
Thus the semisimple Hopf algebra $A^K$ is cosemisimple. 
Since $\mathbb{Q}\cap Z_K=\mathbb{Z}$, by applying Theorem~\ref{3.6} we have $P_{A^K}^{(d)}(x)\in \mathbb{Z}[x]$. 
\end{proof}

\par \bigskip 
\begin{cor}\label{3.8}
Let $\varGamma $ be a finite group, and 
$K$ a finite Galois extension field of $\mathbb{Q}$. Then 
$P_{K[\varGamma ]}^{(d)}(x)\in \mathbb{Z}[x]$ for any positive integer $d$. 
\end{cor} 
\begin{proof} 
The group Hopf algebra $K[\varGamma ]$ is isomorphic to the scalar extension of $\mathbb{Q}[\varGamma ]$ by $K$. 
Since a group algebra over a field of characteristic $0$ is semisimple, by Corollary~\ref{3.7} we have 
$P_{K[\varGamma ]}^{(d)}(x)\in \mathbb{Z}[x]$ for any positive integer $d$. 
\end{proof}

\par \medskip 
\subsection{Stability of polynomial invariants} 
Given a field extension $K/\boldsymbol{k}$, any universal $R$-matrix $R$ of $A$ can be regarded as a universal $R$-matrix of the scalar extension $A^K$ in a natural way (see below for details). 
So, for an increasing sequence of finite extensions of fields 
$\boldsymbol{k} \subsetneqq K_1\subsetneqq K_2\subsetneqq \cdots $ 
we have an increasing sequence of sets of universal $R$-matrices $\underline{\textrm{Braid}}(A)\subset \underline{\textrm{Braid}}(A^{K_1})\subset \underline{\textrm{Braid}}(A^{K_2})\subset \cdots $. 
In this subsection we show that there is $n$ such that $\underline{\textrm{Braid}}(A^{K_n})=\underline{\textrm{Braid}}(A^{K_{n+1}})=\cdots $. 
As an application, we prove that the polynomial invariants $P_A^{(d)}(x)\ (d=1,2,\dots )$ are stabilized by taking some suitable extension of the base field. 
\par 
Let $A$ be a Hopf algebra over a field $\boldsymbol{k}$, and $L$ a commutative algebra over $\boldsymbol{k}$. 
Then $A^L=A\otimes L$ becomes a Hopf algebra over $L$. 
Furthermore, if $R=\sum _i\alpha _i\otimes \beta _i$ is a universal $R$-matrix of $A$, then 
$$R^L=\sum _i(\alpha _i\otimes 1_{\boldsymbol{k}})\otimes _L (\beta _i\otimes 1_{\boldsymbol{k}}) \ \ \in \ \ A^L\otimes _LA^L$$
is a universal $R$-matrix of $A^L$. 
Via the injection $\underline{\textrm{Braid}}(A) \longrightarrow \underline{\textrm{Braid}}(A^L)$ which sends $R$ to $R^L$, 
we  frequently regard $\underline{\textrm{Braid}}(A)$ as a subset of $\underline{\textrm{Braid}}(A^L)$.  
\par 
Let $\underline{\textrm{alg}}_{\boldsymbol{k}}$ denote the $\boldsymbol{k}$-additive category whose objects are commutative algebras over $\boldsymbol{k}$ and morphisms are algebra maps between them. 
Let $A$ and $B$ be two Hopf algebras over $\boldsymbol{k}$. 
For a commutative algebra $L\in \underline{\textrm{alg}}_{\boldsymbol{k}}$,  
we set 
$$\textrm{Hopf}_L(A^L, B^L):=\{ \textrm{ the $L$-Hopf algebra maps  $A^L \longrightarrow B^L$ }\} ,$$
and for an algebra map $f: L_1\longrightarrow L_2$ between commutative algebras $L_1, L_2\in \underline{\textrm{alg}}_{\boldsymbol{k}}$ and $\varphi \in \textrm{Hopf}_{L_1}(A^{L_1}, B^{L_1})$ 
we define a map $f_{\ast}\varphi \in \textrm{Hopf}_{L_2}(A^{L_2}, B^{L_2})$ by the composition: 
\begin{align*}
& A\otimes L_2 \xrightarrow{\ \textrm{id}\otimes \eta \ } A\otimes (L_1\otimes L_2) 
\cong (A\otimes L_1)\otimes L_2 \xrightarrow{\ \varphi \otimes \textrm{id}\ } 
(B\otimes L_1)\otimes L_2 \\ 
& \quad \qquad \xrightarrow{\ \textrm{id}_B\otimes f  \ } 
(B\otimes L_2)\otimes L_2\cong B\otimes (L_2\otimes L_2) \xrightarrow{\ \textrm{id}\otimes \mu_{L_2}  \ } B\otimes L_2, 
\end{align*}
where $\mu_{L_2}$ is the multiplication of $L_2$, and $\eta : L_2\longrightarrow L_1\otimes L_2$ is the $\boldsymbol{k}$-algebra map  defined by $\eta (y)=1_{L_1}\otimes y\ (y\in L_2)$. 
This $\boldsymbol{k}$-linear map $f_{\ast}\varphi $ is directly defined by 
$$(f_{\ast}\varphi )(a\otimes y)=\sum\limits_i b_i\otimes f(x_i)y, 
\qquad \Bigl(\textrm{$\varphi (a\otimes 1_{L_1})=\sum\limits_i b_i \otimes x_i$}\Bigr)$$
for all $a\in A,\ y\in L_2$. 
Let $\underline{\textrm{Set}}$ denote the category whose objects are sets and morphisms are maps. 
Then we have a covariant functor
$\textbf{Hopf}(A,B): \underline{\textrm{alg}}_{\boldsymbol{k}}\longrightarrow \underline{\textrm{Set}}$ such as 
\begin{align*}
\textrm{for an object} & : \quad L \longmapsto \textrm{Hopf}_L(A^L, B^L),\\ 
\textrm{for a morphism} & : \quad f\longmapsto \Bigl( \textbf{Hopf}(A,B)(f): \varphi \longmapsto f_{\ast}\varphi \Bigr) . 
\end{align*}

Provided that $A$ and $B$ are of finite dimension over $\boldsymbol{k}$, 
the functor $\textbf{Hopf}(A,B)$ becomes an algebraic affine scheme over $\boldsymbol{k}$, that is, it is represented by a finitely generated commutative algebra $Z\in \underline{\textrm{alg}}_{\boldsymbol{k}}$ \cite[p.4--5 \& p.58]{Wat1}. 
Furthermore, if $A$ is semisimple, and $B$ is cosemisimple, then the representing object $Z$ is separable and of finite dimension, as the proof is given in the next proposition.  
This result is essentially proved by Etingof and Gelaki \cite[Corollary 1.3]{EG2}, and  in the case of $A=B$ and the affine group scheme $\textbf{Aut}(A)$ the same result is proved by Waterhouse \cite[Theorem 1]{Wat2} (see also \cite[Corollary 6]{Ra1} and \cite[Proposition 1]{Ra2} in case of  characteristic $0$ or positive characteristic with some additional assumptions). 

\par \smallskip 
\begin{prop}\label{3.9}
Let $A$ and $B$ be Hopf algebras over $\boldsymbol{k}$ of finite dimension. 
If $A$ is semisimple, and $B$ is cosemisimple, then the affine scheme 
$\textbf{Hopf}(A,B)$ is finite etale, that is, a representing object $Z\in \underline{\textrm{alg}}_{\boldsymbol{k}}$ of $\textbf{Hopf}(A,B)$ is separable and of finite dimension. 
\end{prop}
\begin{proof}
For two  algebras $A$ and $B$ over $\boldsymbol{k}$, 
let $\textrm{Alg}_{\boldsymbol{k}}(A, B)$ denote the set of the $\boldsymbol{k}$-algebra maps from $A$ to $B$. 
Let $Z\in \underline{\textrm{alg}}_{\boldsymbol{k}}$ be a representing object of the functor 
$\textbf{Hopf}(A,B):\underline{\textrm{alg}}_{\boldsymbol{k}}\longrightarrow \underline{\textrm{Set}}$. 
Since $Z$ is finitely generated commutative, it is a  Noetherian algebra over $\boldsymbol{k}$. 
Thus for an algebraically closed field $\bar{\boldsymbol{k}}$ which contains  $\boldsymbol{k}$, the $\bar{\boldsymbol{k}}$-algebra $\bar{Z}:=Z\otimes \bar{\boldsymbol{k}}$ is also finitely generated commutative Noetherian. 
Then by Hilbert's Nullstellensatz the nilradical $\sqrt{0}$ of $\bar{Z}$ coincides with the Jacobson radical $\textrm{rad}\kern0.2em \bar{Z}$, and the set of minimal prime ideals of $\bar{Z}$ is finite. 
We will show that any maximal ideal of $\bar{Z}$ is a minimal prime ideal. 
To do this, it is enough to prove that $\mathfrak{m}^2=\mathfrak{m}$ for any maximal ideal $\mathfrak{m}$ of $\bar{Z}$. 
For $\phi \in \textrm{Alg}_{\bar{\boldsymbol{k}}}(\bar{Z}, \bar{\boldsymbol{k}})$, 
let us consider the $\phi $-derivations 
$$\textrm{Der}_{\phi}(\bar{Z}, \bar{\boldsymbol{k}})=\{ \ \textrm{$\bar{\boldsymbol{k}}$-linear maps } \delta : \bar{Z}\longrightarrow \bar{\boldsymbol{k}} \ \vert \ 
\delta (\bar{z}\bar{z}^{\prime})=\phi (\bar{z})\bar{z}^{\prime}+\bar{z}\phi (z^{\prime})\ \ (\bar{z}, \bar{z}^{\prime}\in \bar{Z})\ \} .$$
Considering the maximal ideal $\mathfrak{m}_{\phi} :=\textrm{Ker}\kern0.2em \phi $ of $\bar{Z}$, we have a $\bar{\boldsymbol{k}}$-linear isomorphism 
$(\mathfrak{m}_{\phi }/\mathfrak{m}_{\phi }^2)^{\ast} \longrightarrow \textrm{Der}_{\phi }(\bar{Z}, \bar{\boldsymbol{k}})$. 
However, $\textrm{Der}_{\phi}(\bar{Z}, \bar{\boldsymbol{k}})=0$ by the proof of \cite[Corollary 1.3]{EG2}. 
So, we have $\mathfrak{m}_{\phi }=\mathfrak{m}_{\phi }^2$. 
It follows from Hilbert's Nullstellensatz that $\mathfrak{m}=\mathfrak{m}^2$ for any maximal ideal $\mathfrak{m}$ of $\bar{Z}$. 
Since $\mathfrak{m}\bar{Z}_{\mathfrak{m}}=0$, the localization $\bar{Z}_{\mathfrak{m}}$ becomes a field. 
This means that $\mathfrak{m}$ is a minimal prime ideal of $\bar{Z}$. 
\par 
From the above argument, it follows that the set of maximal ideals of $\bar{Z}$ is finite. 
So, let $\mathfrak{m}_1,\ldots , \mathfrak{m}_l$ be the maximal ideals of $\bar{Z}$. 
Then we have $\textrm{rad} \kern0.2em \bar{Z}=\mathfrak{m}_1\cap  \ldots \cap \mathfrak{m}_l=\mathfrak{m}_1\cdot \cdots \cdot \mathfrak{m}_l$. 
Since $\textrm{rad}\kern0.2em \bar{Z}=\sqrt{0}$ is nilpotent, 
 $(\textrm{rad}\kern0.2em  \bar{Z})^n=0$ for a sufficiently large integer $n$. 
From this fact and $\mathfrak{m}_j^2=\mathfrak{m}_j\ (j=1,\ldots ,l)$, we see that $\textrm{rad}\kern0.2em  \bar{Z}=(\textrm{rad}\kern0.2em  \bar{Z})^n=0$. 
Therefore, $\mathfrak{m}_1\cdot \cdots \cdot \mathfrak{m}_l=0$. 
This implies that the Noetherian algebra $\bar{Z}$ is Artinian. 
Thus $\bar{Z}$ is semisimple, that is, $Z$ is separable. 
\end{proof}

\par \smallskip 
For a Hopf algebra $A$ over $\boldsymbol{k}$, the $\boldsymbol{k}$-linear space  $\textrm{Hom}_{\boldsymbol{k}}(A^{\ast}, A)$ possesses a left $A$-module action $\rightharpoonup$ and a right $A$-module action $\leftharpoonup$ defined as follows. 
\begin{align*}
(a\rightharpoonup f)(p)&=\sum a_{(1)} \langle p_{(1)}, a_{(2)}\rangle f(p_{(2)}),\\ 
(f\leftharpoonup a)(p)&=\sum \kern0.1em  f(p_{(1)})\langle p_{(2)}, a_{(1)}\rangle a_{(2)}
\end{align*}
for $f\in \textrm{Hom}_{\boldsymbol{k}}(A^{\ast}, A),\ p\in A^{\ast},\ a\in A$. 
Here $\langle \ ,\ \rangle $ denotes the natural pairing of $A^{\ast}$ and $A$, and the sigma notation such as $\Delta (a)=\sum a_{(1)}\otimes a_{(2)}$ is used.

\par \smallskip 
\begin{lem}[{\cite{Ra3}}]\label{3.10}
Let $A$ be a Hopf algebra over $\boldsymbol{k}$, and $F: A\otimes A \longrightarrow \textrm{Hom}_{\boldsymbol{k}}(A^{\ast}, A^{\textrm{cop}})$ 
be the  injective $\boldsymbol{k}$-linear map defined by 
$$(F(a\otimes b))(p)=p(a)b\qquad (a,b\in A,\ p\in A^{\ast}).$$
For an element $R\in A\otimes A$ we set $f=F(R): A^{\ast}\longrightarrow A^{\textrm{cop}}$. Then 
\par  
(1) given $a\in A$, $R$ satisfies 
$\Delta ^{\textrm{cop}}(a)\cdot R=R\cdot \Delta (a)$  if and only if 
$a \rightharpoonup f=f\leftharpoonup a$, 
\par 
(2) $R$ satisfies $(\Delta \otimes \textrm{id})(R)=R_{13}R_{23}$ and $(\varepsilon \otimes \textrm{id})(R)=1$ if and only if $f$ is an algebra map, 
\par 
(3) in the case where $A$ is finite-dimensional, 
$R$ satisfies $(\textrm{id}\otimes \Delta )(R)=R_{13}R_{12}$ and $(\textrm{id}\otimes \varepsilon )(R)=1$ if and only if $f$ is a coalgebra map. 
\par 
Therefore, if $A$ is finite-dimensional, then the following are equivalent : 
\begin{enumerate}
\item[(i)] $R$ is a universal $R$-matrix of $A$, 
\item[(ii)] $f$ is a Hopf algebra map satisfying $a \rightharpoonup f=f\leftharpoonup a$ for all $a\in A$,
\end{enumerate} 
and the restriction of $F$ to $\underline{\textrm{Braid}}(A)$ defines an injection  
$j_A : \underline{\textrm{Braid}}(A)\longrightarrow \textrm{Hopf}_{\boldsymbol{k}}(A^{\ast}, A^{\textrm{cop}})$.   
\end{lem}

\par \medskip 
Combining Proposition~\ref{3.9} with Lemma~\ref{3.10}, we have the following theorem. 

\par \medskip 
\begin{thm}\label{3.11}
Let $A$ be a cosemisimple Hopf algebra over a field $\boldsymbol{k}$ of finite dimension. 
Then there is a separable finite extension field $L$ of $\boldsymbol{k}$ such that $\underline{\rm{Braid}}(A^L)=\underline{\rm{Braid}}(A^E)$ for any field extension $E/L$. 
\end{thm}
\begin{proof}
Since $A$ is cosemisimple, the dual Hopf algebra $A^{\ast}$ is semisimple. 
By Theorem~\ref{3.9}, a representing object $Z\in \underline{\textrm{alg}}_{\boldsymbol{k}}$ of the functor $\textbf{Hopf}(A^{\ast}, A^{\textrm{cop}}):\underline{\textrm{alg}}_{\boldsymbol{k}}\longrightarrow \underline{\textrm{Set}}$ is separable and of finite dimension. 
So, we take a separable finite extension field $L$ of $\boldsymbol{k}$ which is a splitting field of $Z$. 
We set $H=A^L$, and consider the injective map 
$j_H : \underline{\textrm{Braid}}(H)\longrightarrow \textrm{Hopf}_L(H^{\ast}, H^{\textrm{cop}})$ defined as in Lemma~\ref{3.10}. 
\par 
Let $E$ be an extension field of $L$. Then the following diagram commutes for the inclusion $i: L\longrightarrow E$. 

$$\setlength{\unitlength}{0.7mm}
\begin{picture}(160,97)(0,0)
\put(0,90){\makebox(20,10)[c]{$\underline{\textrm{Braid}}(H)$}}
\put(125,90){\makebox(20,10)[c]{$\textrm{Hopf}_L(H^{\ast}, H^{\textrm{cop}})$}}
\put(125,65){\makebox(20,10)[c]{$\textrm{Hopf}_L((A^{\ast})^L, (A^{\textrm{cop}})^L)$}}
\put(125,25){\makebox(20,10)[c]{$\textrm{Hopf}_E((A^{\ast})^E, (A^{\textrm{cop}})^E)$}}
\put(0,25){\makebox(20,10)[c]{$\underline{\textrm{Braid}}(H^E)$}}
\put(0,0){\makebox(20,10)[c]{$\underline{\textrm{Braid}}(A^E)$}}
\put(125,0){\makebox(20,10)[c]{$\textrm{Hopf}_E((A^E)^{\ast}, (A^E)^{\textrm{cop}})$}}
\put(30,95){\vector(1,0){75}}
\put(7,88){\vector(0,-1){50}}
\put(137,63){\vector(0,-1){25}}
\put(27,5){\vector(1,0){67}}
\put(5,16){$\wr\kern-0.3em \parallel$}
\put(135,16){$\wr\kern-0.3em \parallel$}
\put(135,81){$\wr\kern-0.3em \parallel$}
\put(55,98){$j_H$}
\put(55,9){$j_{A^E}$}
\put(140,50){$\textbf{Hopf}(A^{\ast}, A^{\textrm{cop}})(i)$}
\put(-7,16){\small $(\ast 1)$}
\put(141,81){\small $(\ast 2)$}
\put(141,16){\small $(\ast 3)$}
\end{picture}
$$
\noindent 
Here, the arrow $\underline{\textrm{Braid}}(H)\longrightarrow \underline{\textrm{Braid}}(H^E)$ of L.H.S. is expressing the map defined by $R\longmapsto R^E$, and 
the isomorphisms marked with $(\ast 1)$, $(\ast 2)$, $(\ast 3)$ are induced from 
the Hopf algebra isomorphisms defined below, respectively: 
\begin{align*}
H^E=(A^L)\otimes _LE \longrightarrow A^E,&\qquad (a\otimes x)\otimes _L y \longmapsto a\otimes xy\qquad (a \in A,\ x\in L,\ y\in E),\\ 
\iota _L: (A^{\ast})^L \longrightarrow (A^L)^{\ast}=H^{\ast},&\quad \iota _L(q\otimes x): a\otimes x^{\prime} \longmapsto q(a)xx^{\prime} \quad (q\in  A^{\ast},\ x,x^{\prime}\in L),\\ 
\iota _E: (A^{\ast})^E \longrightarrow (A^E)^{\ast},&\quad 
\iota _E(q\otimes y): a\otimes y^{\prime} \longmapsto q(a)yy^{\prime} \quad (q\in  A^{\ast},\ y,y^{\prime}\in E). 
\end{align*}

We will show that $\textbf{Hopf}(A^{\ast}, A^{\textrm{cop}})(i)$ is bijective. 
Since $Z$ is a representing object of $\textbf{Hopf}(A^{\ast}, A^{\textrm{cop}})$, there are  isomorphisms $\Phi_L$ and $\Phi_E$ such that 
the diagram

$$\setlength{\unitlength}{0.7mm}
\begin{picture}(160,50)(0,0)
\put(135,40){\makebox(20,10)[c]{$\textrm{Alg}_{\boldsymbol{k}}(Z, L)$}}
\put(30,40){\makebox(20,10)[c]{$\textrm{Hopf}_L((A^{\ast})^L, (A^{\textrm{cop}})^L)$}}
\put(30,0){\makebox(20,10)[c]{$\textrm{Hopf}_E((A^{\ast})^E, (A^{\textrm{cop}})^E)$}}
\put(135,0){\makebox(20,10)[c]{$\textrm{Alg}_{\boldsymbol{k}}(Z, E)$}}
\put(75,45){\vector(1,0){52}}
\put(44,40){\vector(0,-1){28}}
\put(146,40){\vector(0,-1){28}}
\put(75,5){\vector(1,0){50}}
\put(95,48){$\Phi_L$}
\put(95,8){$\Phi_E$}
\put(-5,24){$\textbf{Hopf}(A^{\ast}, A^{\textrm{cop}})(i)$}
\put(148,24){$\textrm{Alg}_{\boldsymbol{k}}(\textrm{id}_Z, i)$}
\end{picture}
$$

\par \noindent 
commutes. 
From this commutative diagram it is sufficient to show that 
$i_{\ast}:=\textrm{Alg}_{\boldsymbol{k}}(\textrm{id}_Z, i): \textrm{Alg}_{\boldsymbol{k}}\kern-0.1em (Z,L)\longrightarrow \textrm{Alg}_{\boldsymbol{k}}\kern-0.1em (Z,E)$ is bijective. 
It is clear that $i_{\ast}$ is injective. 
To prove that $i_{\ast}$ is surjective, let $f$ be an algebra map in $\textrm{Alg}_{\boldsymbol{k}}\kern-0.1em (Z,E)$. Then $f(Z)\subset L$, since $Z\otimes L=L^{\oplus n}$ for some $n$ as algebras over $L$, and $E$ is a field. 
So, every algebra map $f\in \textrm{Alg}_{\boldsymbol{k}}\kern-0.1em (Z,E)$ defines an algebra map $g: Z\longrightarrow L$ such as $i_{\ast}(g)=f$. 
This means that 
$\textrm{Alg}_{\boldsymbol{k}}(\textrm{id}_Z, i): \textrm{Alg}_{\boldsymbol{k}}\kern-0.1em (Z,L)\longrightarrow \textrm{Alg}_{\boldsymbol{k}}\kern-0.1em (Z,E)$ is surjective. 
\par 
Next, we will show that $\textbf{Hopf}(A^{\ast}, A^{\textrm{cop}})(i)$ preserves the actions $\rightharpoonup$ and $\leftharpoonup$. 
Let $\varphi $ be an element in $\textrm{Hopf}_L((A^{\ast})^L, (A^{\textrm{cop}})^L)$. Then by using what 
$i_{\ast}\varphi =(\textbf{Hopf}(A^{\ast}, A^{\textrm{cop}})(i))(\varphi )$ is 
given by 
$$(i_{\ast}\varphi )(p\otimes y)=\varphi (p\otimes 1_L)\cdot (1_A\otimes y)\qquad (p\in A^{\ast},\ y\in E),$$
we have 
\begin{align*}
&((a\otimes 1_L)\rightharpoonup \varphi )(p\otimes 1_L)=(\varphi \leftharpoonup (a\otimes 1_L) )(p\otimes 1_L) \\ 
&\qquad \qquad \qquad \Longleftrightarrow \ \ 
((a\otimes 1_E)\rightharpoonup  i_{\ast}\varphi )(p\otimes 1_E) 
=( i_{\ast}\varphi \leftharpoonup (a\otimes 1_E))(p\otimes 1_E)
\end{align*}
for all $a\in A,\ p\in A^{\ast}$. 
This implies that 
\begin{align*}
&(a\otimes x)\rightharpoonup \varphi  =\varphi \leftharpoonup (a\otimes x)  \ \ \textrm{for all $a\in A$ and $x\in L$} \\ 
&\quad \ \ \Longleftrightarrow \ \ 
(a\otimes y)\rightharpoonup  i_{\ast}\varphi = i_{\ast}\varphi \leftharpoonup  (a\otimes y)\ \ \textrm{for all $a\in A$ and $y\in E$}. 
\end{align*}

By Lemma~\ref{3.10}, then we see that the map $\underline{\textrm{Braid}}(H)\longrightarrow \underline{\textrm{Braid}}(H^E),\ R\longmapsto R^E$ is bijective. \end{proof}

\par \smallskip 
\begin{rem}\label{3.12}\rm 
Using Proposition~\ref{3.9} and Lemma~\ref{3.10}, one can give a proof of  Theorem~\ref{2.1}(i) as follows  (c.f. \cite[Corollary 1.5]{EG2}). 
Let $A$ be a cosemisimple Hopf algebra over $\boldsymbol{k}$ of finite dimension, and take $L$ and $H$ as in the proof of Theorem~\ref{3.11}. 
Then  by Proposition~\ref{3.9} the set $\textrm{Hopf}_L(H^{\ast}, H^{\textrm{cop}})$ is finite. 
It follows from the injectivity of $j_H$ that the set $\underline{\textrm{Braid}}(H)$ is finite, too. 
Since $\underline{\textrm{Braid}}(A)$ can be identified with a subset of $\underline{\textrm{Braid}}(H)$, the set $\underline{\textrm{Braid}}(A)$ is also finite. 
This ends the proof of Theorem~\ref{2.1}(i). 
\par 
We note that for any field extension $K/\boldsymbol{k}$ it is also true that $\underline{\rm{Braid}}(A^K)$ is finite. 
Because, the scalar extension $A^K$ is still cosemisimple.  
\end{rem} 

\par \smallskip 
\begin{cor}
Let $A$ be a semisimple and cosemisimple Hopf algebra over a field $\boldsymbol{k}$ of finite dimension. 
Then there is a separable finite extension field $L$ of $\boldsymbol{k}$ such that 
for any field extension $E$ of $L$ and for any positive integer $d$, 
$P_{A^E}^{(d)}(x)=P_{A^L}^{(d)}(x)$ in $E[x]$. 
\end{cor}
\begin{proof}
Since $A$ is separable by \cite[Corollary 2.2.2]{Mon}, there is a finite separable extension field $K$ of $\boldsymbol{k}$ such that 
it is a splitting field of $A$. 
Then $H:=A^K$ is a semisimple and cosemisimple Hopf algebra over $K$ of finite dimension. 
\par 
Let $L$ be a separable finite extension field of $K$ having the property in Theorem~\ref{3.11}. 
Then for any field extension $E/L$ and any absolutely simple left $H^L$-module $M$, the left $(H^{L})^E$-module $M^E$ is absolutely simple, and 
the equation
\begin{equation}\label{eq3.3}
P_{(H^{L})^E, M^E}(x)=P_{H^{L}, M}(x)
\end{equation}
holds in $E[x]$ since $\underline{\dim}_{R^E} M^E = \underline{\dim}_R M$ for all 
$R\in \underline{\textrm{Braid}}(H)$. 
Since $K$ is a splitting field of $H$, if 
$\{ M_1,\ldots , M_t\} $ is a full set of non-isomorphic absolutely simple left $H$-modules of dimension $d$, then $\{ M_1^E,\ldots , M_t^E\} $ is a full set of non-isomorphic absolutely simple left $H^E$-modules of dimension $d$. 
From this fact and the equation (\ref{eq3.3}), we have the equation 
$$
P_{(H^L)^E}^{(d)}(x)
=\prod\limits_{i=1}^t P_{(H^{L})^E, (M_i^L)^E}(x) \\ 
=\prod\limits_{i=1}^t P_{H^{L}, M_i^L}(x) \\ 
=P_{H^L}^{(d)}(x) \qquad \textrm{in} \ \ E[x]. 
$$
Since $(H^L)^E\cong H^E=(A^K)^E\cong A^E$ as Hopf algebras over $E$,  
we have $P_{(H^L)^E}^{(d)}(x)=P_{A^E}^{(d)}(x)$, and since $H^L=(A^K)^L\cong A^L$ as Hopf algebras over $L$,  
we have $P_{H^L}^{(d)}(x)=P_{A^L}^{(d)}(x)$. 
Thus the desired equation 
$P_{A^E}^{(d)}(x)=P_{A^L}^{(d)}(x)$ is obtained.  
\par 
Finally, we remark that the extension $L/\boldsymbol{k}$ is separable since two extensions $L/K$ and $K/\boldsymbol{k}$ are separable. 
\end{proof}

\par \medskip 
\section{Dual formulas for polynomial invariants}
\par 
In this section we give a formula to compute the polynomial invariants for a self-dual Hopf algebra of finite dimension in terms of the braidings of the dual Hopf algebra. 
\par 
Let us recall the definition of a braiding of a Hopf algebra \cite{Doi}. 
Let $A$ be a Hopf algebra over a field $\boldsymbol{k}$, and let $\sigma :A\otimes A\longrightarrow \boldsymbol{k}$ be a $\boldsymbol{k}$-linear map that is  invertible with respect to the convolution product. 
The pair $(A,\sigma  )$ is said to be a \textit{braided Hopf algebra}, and 
$\sigma $ is said to be a \textit{braiding} of $A$
 if the following conditions are satisfied: for all $x,y,z\in A$
\begin{enumerate}
\item[(B1)] $\sum\sigma  (x_{(1)},y_{(1)})x_{(2)}y_{(2)}=\sum\sigma  (x_{(2)},y_{(2)})y_{(1)}x_{(1)}$,
\item[(B2)] $\sigma  (xy,z)=\sum\sigma  (x,z_{(1)})\sigma (y,z_{(2)})$,
\item[(B3)] $\sigma  (x,yz)=\sum \sigma  (x_{(1)},z)\sigma (x_{(2)},y)$. 
\end{enumerate}
Here, we use the sigma notation such as $\Delta (x)=\sum x_{(1)}\otimes x_{(2)}$ for $x\in A$.  
It is easy to see that any braiding $\sigma $ of $A$ satisfies  
\begin{enumerate}
\item[(B4)] $\sigma (1_A, x)=\sigma (x, 1_A)=\varepsilon (x)$ for all $x\in A$. 
\end{enumerate}

Let $(A,\sigma  )$ be a braided Hopf algebra over $\boldsymbol{k}$. 
Then the braiding $\sigma $ defines a braiding $c$ of the monoidal category $\mathbb{M}^A$ consisting of right $A$-comodules and $A$-colinear maps as follows. 
For two right $A$-comodules $V$ and $W$, a $\boldsymbol{k}$-linear isomorphism $c_{V,W}:V\otimes W\longrightarrow W\otimes V$ is defined by 
$$c_{V,W}(v\otimes w)=\sum\sigma (v_{(1)},w_{(1)})w_{(0)}\otimes v_{(0)} \qquad (v\in V,\ w\in W),$$
where we use the notations $\rhou_V(v)=\sum v_{(0)}\otimes v_{(1)}$ and $ \rhou_W(w)=\sum w_{(0)}\otimes w_{(1)}$ for the given right coactions $\rhou_V$ and $\rhou_W$ of $V$ and $W$, respectively. 
From the axiom of braiding (B1) -- (B3), we see that $c_{V,W}$ is a right $A$-comodule isomorphism, and the collection $c=\{ c_{V,W}:V\otimes W\longrightarrow W\otimes V\}_{V,W\in \mathbb{M}^A}$ gives a braiding of $\mathbb{M}^A$. 
\par 
Let us consider the element in the braided Hopf algebra which plays a role of the Drinfel'd element in a quasitriangular Hopf algebra.   

\par \medskip 
\begin{lem}[{\cite[Theorem 1.3]{Doi}} or {\cite[3.3.2]{Sc1}}]\label{4.2}
Let $(A,\sigma )$ be a braided Hopf algebra over $\boldsymbol{k}$, and define $\mu  \in A^{\ast }$ by  
$$\mu (a)=\sum \sigma (a_{(2)},S(a_{(1)})),\quad a\in A.$$
Then $\mu $ is convolution-invertible, and the following equation holds for any element $a\in A$: 
$$S^2(a)=\sum \mu  (a_{(1)})\mu ^{-1}(a_{(3)})a_{(2)}.$$ 
The $\boldsymbol{k}$-linear functional $\mu $ is called the (dual) Drinfel'd element of $(A,\sigma )$.  
\end{lem}

\par 
Let $\mathcal{V}=(\mathcal{C},\otimes ,\mathbb{I},a,r,l,c)$ be a left rigid braided monoidal category. 
For each object $X\in \mathcal{C}$ we choose a left dual $X^{\ast}$ with an  evaluation morphism $e_X: X^{\ast}\otimes X\longrightarrow \mathbb{I}$ and a coevaluation morphism $n_X: \mathbb{I}\longrightarrow X\otimes X^{\ast}$. 
Then for an endomorphism $f:X\longrightarrow X$ in $\mathcal{C}$, 
the braided trace of $f$ in $\mathcal{V}$, denoted by $\underline{\textrm{Tr}}_{\kern0.1em c}\kern0.1em f$, is defined by the composition 
$$\mathbb{I}\xrightarrow{\ n_X\ }X\otimes X^{\ast}\xrightarrow{f\otimes \textrm{id}}
X\otimes X^{\ast}\xrightarrow{c_{X,X^{\ast}}}X^{\ast}\otimes X\xrightarrow{\ e_X\ }\mathbb{I}.$$

In particular, the braided trace of the identity morphism $\textrm{id}_X$ is denoted by  $\underline{\dim}_{\kern0.1em c}\kern0.1em X$, and called the \textit{braided dimension} of $X$ in $\mathcal{V}$. 
\par 
Applying this to the braided monoidal category $(\mathbb{M}^A, c)$ constructed from a braiding $\sigma \in (A\otimes A)^{\ast}$, we have the following.  

\par 
\begin{lem}\label{4.3}
Let $(A,\sigma )$ be a braided Hopf algebra over $\boldsymbol{k}$, and 
$c$ the braiding of $\mathbb{M}^A$ constructed from $\sigma $. 
Then for a finite-dimensional right $A$-comodule $V$, 
the braided dimension $\underline{\dim}_{\kern0.1em c}\kern0.1em V$ is given by $\underline{\dim}_{\kern0.1em \sigma}\kern0.1em V=\mu (\chiu _V)$, 
where $\mu $ is the Drinfel'd element of $(A, \sigma )$, and $\chiu _V$ is the character of the comodule $V$, which is defined by 
$$\chi _V:=\sum\limits_{i=1}^n(v_i^{\ast}\otimes \textrm{id}_C)(\rhou_V(v_i))\ \ \in \ \ C$$
by use of dual bases $\{ v_i\}_{i=1}^n$ and $\{ v_i^{\ast}\}_{i=1}^n$.  
\end{lem}

\par 
\begin{lem}\label{4.4}
 Let $A$ be a Hopf algebra over $\boldsymbol{k}$ of finite dimension, and 
 $\iota :A^{\ast }\otimes A^{\ast }\longrightarrow (A\otimes A)^{\ast }$ is the canonical $\boldsymbol{k}$-linear isomorphism. 
 Let  $\sigma $ be an element of $(A\otimes A)^{\ast }$ and set $R:=\iota ^{-1}(\sigma )$. Then 
\par 
\noindent 
(1) $\sigma $ is convolution-invertible if and only if  $R$ is invertible as an element of the algebra $A^{\ast }\otimes A^{\ast }$. 
\par \noindent 
(2) $\sigma $ is a braiding of $A$ if and only if $R$ is a universal $R$-matrix of the dual Hopf algebra $A^{\ast }$. 
In this case, the Drinfel'd element $\mu \in A^{\ast }$ of the quasitriangular Hopf algebra $(A^{\ast}, R)$ is given by $\mu (a)=\sum \sigma (a_{(2)},S(a_{(1)}))$ for all $a\in A$. 
\par 
(3) for a finite-dimensional right $A$-comodule $V$, the equation $\underline{\textrm{dim}}_R V=\underline{\textrm{dim}}_{\kern0.1em \sigma } V$ holds, where $\underline{\textrm{dim}}_R V$ is the $R$-dimension of the  left $A^{\ast}$-module $V$ with the action  
$$p\cdot v:=\sum p(v_{(1)}) v_{(0)} \qquad (p\in A^{\ast},\ v\in V). $$ 
\end{lem}

\par \medskip 
Let $C$ be a coalgebra over a field $\boldsymbol{k}$. 
A right $C$-comodule $V$ is said to be an \textit{absolutely simple} if the right $C^K$-comodule $V^K$ is simple for an arbitrary field extension $K/\boldsymbol{k}$. 
This condition is equivalent to that $V$ is absolutely simple as a left $C^{\ast}$-module. We note that if a right $C$-comodule is simple, then it is automatically  finite-dimensional (see \cite[Corollary 5.1.2]{Mon}). 
So, there is a one-to-one correspondence between the absolutely simple right $C$-comodules and the  absolutely simple left $C^{\ast}$-modules. 
\par 
Let $\underline{\textrm{braid}}(A)$ denote the set of all braidings of a Hopf algebra $A$. Then by Part (2) of Lemma~\ref{4.4}, the map  
$$\underline{\textrm{braid}}(A) \longrightarrow \underline{\textrm{Braid}}(A^{\ast}),\qquad \sigma \longmapsto \iota ^{-1}(\sigma )$$
is bijective, and by Part (3) of the same lemma, 
the equation $\underline{\dim }_{\kern0.1em \iota ^{-1}(\sigma )} V=\underline{\dim }_{\kern0.1em \sigma } V$ holds for a finite-dimensional right $A$-comodule $V$. 
Hence we have the following: 

\par 
\begin{lem}\label{4.5}
Let $A$ be a semisimple and cosemisimple Hopf algebra over $\boldsymbol{k}$ of finite dimension. 
\par 
(1) For an absolutely simple right $A$-comodule $V$, 
$$P_{A^{\ast},V}(x)=\prod\limits_{\sigma \in \underline{\rm{braid}}(A)}\Bigl( x-\frac{\underline{\dim }_{\kern0.1em \sigma } V}{\dim V}\Bigr),$$
where in the left-hand side $P_{A^{\ast},V}(x)$ is the polynomial for $V$ regarded as an left $A^{\ast}$-module by usual manner. 
\par 
(2) Let $\{ V_1,\ldots ,V_t\} $ be a full set of non-isomorphic absolutely simple right $A$-comodules of dimension $d$.  Then  
$$P_{A^{\ast}}^{(d)}(x)=\prod\limits_{i=1}^tP_{A^{\ast},V_i}(x).$$  
\end{lem}

\par 
A Hopf algebra $A$ over a field $\boldsymbol{k}$ of finite dimension is called \textit{self-dual} if $A$ is isomorphic to the dual Hopf algebra $A^{\ast}$ as a Hopf algebra. 
Applying the above lemma to a self-dual Hopf algebra, we obtain immediately the following proposition. 

\par 
\begin{prop}\label{4.6}
Let $A$ be a semisimple and cosemisimple Hopf algebra over a field $\boldsymbol{k}$ of finite dimension. 
If $A$ is self-dual, then for a positive integer $d$ 
$$P_{A}^{(d)}(x)=\prod\limits_{i=1}^t\prod\limits_{\sigma \in \underline{\rm{braid}}(A)}\Bigl( x-\frac{\underline{\dim }_{\kern0.1em \sigma } V_i}{\dim V_i} \Bigr),$$
where $\{ V_1,\ldots ,V_t\} $ is a full set of non-isomorphic absolutely simple right $A$-comodules of dimension $d$. 
\end{prop}

By using the above formula, we compute the self-dual Hopf algebras $A_{Nn}^{+\lambda}$ introduced by Satoshi Suzuki \cite{Suz} in the next section. 

\par \medskip 
\section{Examples}
In this section we give several computational results of polynomial invariants of Hopf algebras. 
By comparing polynomial invariants one may find new examples of pairs of Hopf algebras such that their representation rings are isomorphic, but they are not monoidally Morita equivalent. 

\subsection{$8$-dimensional non-commutative semisimple Hopf algebras}
\par 
By Masuoka~\cite{Mas1}, it is known that there are exactly three types of $8$-dimensional non-commutative semisimple Hopf algebras over an algebraically closed field $\boldsymbol{k}$ of $\textrm{ch}(\boldsymbol{k})\not= 2$. 
They are $\boldsymbol{k}[D_8], \boldsymbol{k}[Q_8]$ and $K_8$, where 
$D_8$ and $Q_8$ are the dihedral group of order $8$ and the quaternion group, respectively, 
and $K_8$ is the unique $8$-dimensional semisimple Hopf algebra which is non-commutative and non-cocommutative, that is called the Kac-Paljutkin algebra~\cite{KP}. 
Tambara and Yamagami \cite{TY} and also Masuoka \cite{Mas2} showed that their representation rings are isomorphic meanwhile their representation categories are not. 
In this subsection we derive this result by using our polynomial invariants. 
Throughout this subsection, we fix the following group presentations of $D_8$ and $Q_8$:  
\begin{align*}
D_8&=\langle s,t\ \vert \ s^4=1,\ t^2=1,\ st=ts^{-1}\rangle , \\ 
Q_8&=\langle s,t\ \vert \ s^4=1,\ t^2=s^2,\ st=ts^{-1}\rangle . 
\end{align*}

Let us start with determining the universal $R$-matrices of the group Hopf algebras $\boldsymbol{k}[D_8]$ and $\boldsymbol{k}[Q_8]$. 
For this the following proposition is useful. 

\par \smallskip 
\begin{lem}\label{5.14}
Let $G$ be a group, and $\boldsymbol{k}$ a field. 
Then for a universal $R$-matrix $R$ of $\boldsymbol{k}[G]$ there is a commutative and normal finite subgroup $H$ such that $R\in \boldsymbol{k}[H]\otimes \boldsymbol{k}[H]$. 
\end{lem} 
\begin{proof}
We set $A=\boldsymbol{k}[G]$. 
By \cite[Proposition 2(a)]{Ra4}, 
$$B=\{ \ (\textrm{id}\otimes \alpha )(R)\ \vert \ \alpha \in A^{\ast }\ \} ,\quad H:=\{ \ (\alpha \otimes \textrm{id})(R)\ \vert \ \alpha \in A^{\ast}\ \} $$
are finite-dimensional subHopf algebras of $A$. 
Since for each $g\in G$, $\boldsymbol{k}g$ is a subcoalgebra of $A=\bigoplus_{g\in G}\boldsymbol{k}g$, by \cite[Lemma 9.0.1(b)]{Sw} 
the subcoalgebra $B$ is written in $B=\bigoplus_{g\in G}B\cap \boldsymbol{k}g$. 
We set $K:=B\cap G$. Then $K$ is a subgroup of $G$, and 
$B\cap \boldsymbol{k}g\not= \{ 0\}$ if and only if $g\in K$. 
Thus $B=\bigoplus_{g\in K}B\cap \boldsymbol{k}g=\bigoplus_{g\in K}\boldsymbol{k}g=\boldsymbol{k}[K]$. 
Since $B$ is finite-dimensional, $K$ is a finite group. 
As a similar argument, we see that there is a finite subgroup $L$ of $G$ satisfying 
$H=\boldsymbol{k}[L]$. 
Since by \cite[Proposition 2(c)]{Ra4} $\boldsymbol{k}[K]^{\ast \textrm{cop}}\cong \boldsymbol{k}[L]$ as Hopf algebras, $L$ is commutative. 
Furthermore, $L$ is a normal subgroup of $G$ since 
$Lg\subset gH$, or equivalently $g^{-1}Lg\subset H=\boldsymbol{k}[L]$ by \cite[Proposition 3]{Ra4}. 
Similarly, we see that $K$ is a commutative and normal subgroup of $G$. 
\par 
At this point, $R\in \boldsymbol{k}[K]\otimes \boldsymbol{k}[L]\subset \boldsymbol{k}[LK]\otimes \boldsymbol{k}[LK]$ is verified. 
To complete the proof it is sufficient to show that $LK$ is a commutative and normal subgroup of $G$. 
Since $L$ is normal, it follows immediately that $LK$ is normal. 
To show that $LK$ is commutative, we write $R$ in the form 
$$R=\sum\limits_{k\in K}k\otimes X_k \quad (X_k\in \boldsymbol{k}[L]).$$
Since $\Delta ^{\textrm{cop}}(k')\cdot R=R\cdot \Delta (k')$ for all $k'\in K$, 
we have $X_{k'kk'^{-1}}=k'X_kk'^{-1}$ for all $k,k'\in K$. 
Since $K$ is commutative, this condition is equivalent to $X_{k}k'=k'X_k$. 
It follows from $\boldsymbol{k}[L]=H=\textrm{Span}\{ X_k\ \vert \ k\in K\} $ that any element $l\in L$ is represented by a $\boldsymbol{k}$-linear combination of $\{ X_k\ \vert \ k\in K\} $. Thus $lk'=k'l$ holds. 
This means that $LK$ is commutative. 
\end{proof}

\par \smallskip 
For a cyclic group the universal $R$-matrices of the group Hopf algebra are given by (\ref{eq2.3}) in Example~\ref{2.3} in Section 2. 
For a direct product of two cyclic groups the universal $R$-matrices of the group Hopf algebra are given as in the following lemma. 
The lemma can be verified by use of the same method of the proof of Lemma~\ref{5.17} described later. 

\par \smallskip 
\begin{lem}\label{5.19}
Let $G$ be the direct product of the cyclic groups $C_m=\langle g\rangle $ and  $C_n=\langle h\rangle $, and let $\omega $ be a primitive $mn$-th root of unity in a field $\boldsymbol{k}$ whose characteristic does not divide $mn$. 
We set $X(m,n)=\{ d\in \{ 0,1,\ldots , m-1\} \ \vert \ dn\equiv 0\ ({\rm mod}\kern0.2em m)\} $. 
Then any universal R-matrix of $\boldsymbol{k}[G]$ is given by the formula
$$R_{pqrs}^{\boldsymbol{k}[G]}:=\sum_{i,j=0}^{m-1}\sum_{k,l=0}^{n-1}\omega ^{n(pij+rkj)+m(skl+qil)}E_{ik}\otimes E_{jl},$$ 
where $p\in X(m,m), q\in X(n,m),r\in X(m,n), s\in X(n,n)$, and \newline 
$E_{ik}=\frac{1}{mn}\sum_{j=0}^{m-1}\sum_{l=0}^{n-1}\omega^{-nij-mkl}g^jh^l$.   
\end{lem} 

By using Lemma~\ref{5.14} and Lemma~\ref{5.19}, one can determine that the quasitriangular structures of $\boldsymbol{k}[D_8]$ and $\boldsymbol{k}[Q_8]$ 
as described in the following lemma \cite{W1}. 

\par \smallskip 
\begin{lem}
Let $\boldsymbol{k}$ be a field of $\textrm{ch}(\boldsymbol{k})\not= 2$ that  contains a primitive $4$-th root  of unity $\zeta $. 
\par 
(1) The quasitriangular structures of $\boldsymbol{k}[D_8]$ are given by 
\begin{enumerate}
\item[(i)] universal $R$-matrices of $\boldsymbol{k}[\langle s\rangle ]\cong \boldsymbol{k}[C_4]$: 
$$R_d^{\boldsymbol{k}[D_8]}:=\dfrac{1}{4}\sum\limits_{i,k=0}^3\zeta ^{-ik}s^k\otimes s^{di}\quad (d=0,1,2,3),$$
\item[(ii)]  universal $R$-matrices of $\boldsymbol{k}[\langle t,s^2\rangle ]\cong \boldsymbol{k}[C_2\times C_2]$: 
$$R_{d+4}^{\boldsymbol{k}[D_8]}:=\dfrac{1}{4}\sum\limits_{i,j,k,l=0}^1(-1)^{-ij-kl}t^is^{2k}\otimes t^{l}s^{2(j+dl)}\quad (d=0,1),$$
\item[(iii)]  universal $R$-matrices of $\boldsymbol{k}[\langle ts,s^2\rangle ]\cong \boldsymbol{k}[C_2\times C_2]$: 
$$R_{d+6}^{\boldsymbol{k}[D_8]}:=\dfrac{1}{4}\sum\limits_{i,j,k,l=0}^1(-1)^{-ij-kl}(ts)^is^{2k}\otimes (ts)^{l}s^{2(j+dl)}\quad (d=0,1). $$
\end{enumerate}
The Drinfel'd element $u_d^{\boldsymbol{k}[D_8]}$ of $R_d^{\boldsymbol{k}[D_8]}$ is given by 
$$u_d^{\boldsymbol{k}[D_8]}
=\begin{cases}
1\quad & (d=0,4,6), \\[0.1cm]  
\frac{1}{2}(1+\zeta ^{-1})+\frac{1}{2}(1+\zeta )s^2 \quad & (d=1), \\[0.1cm]   
s^2 \quad & (d=2,5,7), \\[0.1cm]   
\frac{1}{2}(1+\zeta )+\frac{1}{2}(1+\zeta ^{-1})s^2 \quad & (d=3).
\end{cases}$$
\indent 
(2) The quasitriangular structures of $\boldsymbol{k}[Q_8]$ are given by 
\begin{enumerate}
\item[(i)] universal $R$-matrices of $\boldsymbol{k}[\langle s\rangle ]\cong \boldsymbol{k}[C_4]$: 
$$R_d^{\boldsymbol{k}[Q_8]}:=\dfrac{1}{4}\sum\limits_{i,k=0}^3\zeta ^{-ik}s^k\otimes s^{di}\quad (d=0,1,2,3),$$
\item[(ii)] universal $R$-matrices of $\boldsymbol{k}[\langle t\rangle ]\cong \boldsymbol{k}[C_4]$: 
$$R_{d+4}^{\boldsymbol{k}[Q_8]}:=\dfrac{1}{4}\sum\limits_{i,k=0}^3\zeta ^{-ik}t^k\otimes t^{(2d+1)i}\quad (d=0,1),$$
\item[(iii)] universal $R$-matrices of $\boldsymbol{k}[\langle ts\rangle ]\cong \boldsymbol{k}[C_4]$:
$$R_{d+6}^{\boldsymbol{k}[Q_8]}:=\dfrac{1}{4}\sum\limits_{i,k=0}^3\zeta ^{-ik}(ts)^k\otimes (ts)^{(2d+1)i}\quad (d=0,1).$$
\end{enumerate}
The Drinfel'd element $u_d^{\boldsymbol{k}[Q_8]}$ of $R_d^{\boldsymbol{k}[Q_8]}$ is given by 
$$u_d^{\boldsymbol{k}[Q_8]}
=\begin{cases}
1\quad & (d=0), \\[0.1cm]   
\frac{1}{2}(1+\zeta ^{-1})+\frac{1}{2}(1+\zeta )s^2 \quad & (d=1,4,6), \\[0.1cm]   
s^2 \quad & (d=2), \\[0.1cm]   
\frac{1}{2}(1+\zeta )+\frac{1}{2}(1+\zeta ^{-1})s^2 \quad & (d=3,5,7). 
\end{cases}
$$  
\end{lem}
\begin{proof}
In either case of $\boldsymbol{k}[D_8]$ or $\boldsymbol{k}[Q_8]$, the proof of the lemma can be done by the same method. 
So, we will determine the universal $R$-matrices only in the case of $\boldsymbol{k}[Q_8]$ (for $\boldsymbol{k}[D_8]$ see also Proposition~\ref{5.20} described later). 
\par 
It is easy to show that the maximal commutative and normal subgroups of $Q_8$ coincide with one of $H_1=\langle s\rangle $, $H_2=\langle t \rangle $ and $H_3=\langle ts \rangle $. 
Therefore, by Lemma~\ref{5.14} any universal $R$-matrix of $\boldsymbol{k}[Q_8]$ is that of $\boldsymbol{k}[H_i]$ for some $i=1,2,3$. 
Let $R$ be a universal $R$-matrix of $\boldsymbol{k}[H_1]$. 
Then $\Delta^{\rm{cop}}(t)\cdot R=R\cdot \Delta (t)$ holds, and whence $R$ is a universal $R$-matrix of $\boldsymbol{k}[Q_8]$. 
Similarly, we see that  any universal $R$-matrix of $\boldsymbol{k}[H_i]$ for $i=2,3$ satisfies $\Delta^{\rm{cop}}(s)\cdot R=R\cdot \Delta (s)$, and whence it is a universal $R$-matrix of $\boldsymbol{k}[Q_8]$. 
\end{proof}

\par \medskip 
Next, we describe the quasitriangular structures of the Kac-Paljutkin algebra $K_8$, that are determined by S. Suzuki \cite{Suz}. 
As an algebra the Kac-Paljutkin algebra $K_8$ coincides with the group algebra $\boldsymbol{k}[D_8]$. 
Let us consider the primitive orthogonal idempotents $e_0=\frac{1+s^2}{2},\ e_1=\frac{1-s^2}{2}$ in $K_8=\boldsymbol{k}[D_8]$. 
Then the Hopf algebra structure of $K_8$ is described as follows \cite{Mas2}: 
\begin{align*}
\Delta (t)&=t\otimes e_{0}t+st\otimes e_{1}t,\quad \Delta (s)=s\otimes e_0s+s^{-1}\otimes e_1s,\\ 
\varepsilon (t)&=1,\quad \varepsilon (s)=1, \\ 
S(t)&=e_{0}t+e_1st,\quad S(s)=e_0s^{-1}+e_1s. 
\end{align*}
We note that 
$\Delta (e_0)=e_0\otimes e_0+e_1\otimes e_1,\ 
\Delta (e_1)=e_0\otimes e_1+e_1\otimes e_0$, 
$\varepsilon (e_0)=1,\ \varepsilon (e_1)=0$, 
$S(e_0)=e_0,\ S(e_1)=e_1$, 
and therefore, 
$\boldsymbol{k}e_0+\boldsymbol{k}e_1$ is a subHopf algebra of $K_8$ which is isomorphic to the group Hopf algebra $\boldsymbol{k}[C_2]$. 
Let $\zeta \in \boldsymbol{k}$ be a primitive $4$-th root of unity. 
Then 
$$g:=\frac{1}{2}(1+\zeta )s+\frac{1}{2}(1-\zeta )s^{-1},\ \ h:=\frac{1}{2}(1-\zeta )s+\frac{1}{2}(1+\zeta )s^{-1}$$
satisfy $g^2=h^2=1$, 
and $\boldsymbol{k}[\langle s\rangle ]=\boldsymbol{k}1+\boldsymbol{k}g+\boldsymbol{k}h+\boldsymbol{k}gh$ hold. 
Moreover, since $g$ and $h$ are group-like, 
the subHopf algebra $\boldsymbol{k}[\langle s\rangle ]$ of $K_8$ is isomorphic to the group Hopf algebra $\boldsymbol{k}[C_2\times C_2]$. 

\par \smallskip 
\begin{lem}
Let $\boldsymbol{k}$ be a field of $\textrm{ch}(\boldsymbol{k})\not= 2$ that  contains a primitive $8$-th root  of unity $\omega $. 
Then the universal $R$-matrices of $K_8$ are given as follows: 
\begin{enumerate}
\item[(i)] universal $R$-matrices of $\boldsymbol{k}[\langle g,h\rangle ]\cong \boldsymbol{k}[C_2\times C_2]$: 
$$R_{pq}^{K_8}:=\frac{1}{4}\sum _{i ,j,k,l=0}^{1} (-1)^{-(ij+kl)}g^ih^k\otimes g^{pj+(q+1)l}h^{qj+pl}\quad (p,q\in \{ 0,1\} ), $$
\item[(ii)] minimal universal $R$-matrices of $K_8$: 
$$R_{l}^{K_8}:=\frac{1}{8}\sum\limits_{i,j,p,q,r,s=0,1}\kern-1em \omega ^{(2l+1)\frac{1-(-1)^i}{2}\cdot \frac{1-(-1)^j}{2}}(-1)^{jp+ir+(j(i+1)+lj+r+s)(li+p+q)}\kern0.1em 
t^ig^ph^q\otimes t^jg^rh^s$$
$(l=0,1,2,3)$.
\end{enumerate}
The Drinfel'd elements $u_{p,q}^{K_8}$ and $u_{l}^{K_8}$ of $R_{p,q}^{K_8}$ and $R_l^{K_8}$, respectively, are given by 
\begin{align*}
u_{pq}^{K_8}
&=\frac{1}{2}\sum\limits_{i,l=0}^1 (-1)^{(i+p)(l+p)}g^{i}h^l \quad  (p,q\in \{ 0,1\} ) ,\\[0.1cm]  
u_l^{K_8}
&=\frac{\omega ^{2l-1}}{2}(1-gh)+\frac{1}{2}(g+h)\quad (l=0,1,2,3). 
\end{align*}  
\end{lem}
\begin{proof}
Since the universal $R$-matrix $R_{pq}^{K_8}$ of $\boldsymbol{k}[\langle g,h\rangle ]$ satisfies $\Delta^{\rm{cop}}(t)\cdot R_{pq}^{K_8}=R_{pq}^{K_8}\cdot \Delta (t)$, we see immediately that $R_{pq}^{K_8}$ is a universal $R$-matrix of $K_8$. 
It can be also verified straightforward that $R_{l}^{K_8}$ are indeed a universal $R$-matrix of $K_8$, although the proof is tedious. 
Since $K_8$ is isomorphic to the Suzuki's Hopf algebra $A_{12}^{+-}$ (for example see \cite{CDMM}), by \cite[Proposition 3.10(ii)]{Suz} the number of universal $R$-matrices of $K_8$ is $8$. 
Thus there is no universal $R$-matrix of $K_8$ other than $R_{pq}^{K_8}\ (p,q\in \{ 0,1\} )$ and $R_{l}^{K_8}\ (l=0,1,2,3)$. 
\end{proof}

\par \smallskip 
Let $\boldsymbol{k}$ be a field of $\textrm{ch}(\boldsymbol{k})\not= 2$ that contains a primitive $4$-th root of unity. 
Then for any of the algebras $\boldsymbol{k}[D_8], \boldsymbol{k}[Q_8], K_8$ the number of isomorphism classes of (absolutely) simple modules is $5$. 
They consist of four $1$-dimensional modules and one $2$-dimensional simple module. 
\par 
The $1$-dimensional modules of $\boldsymbol{k}[D_8],\ \boldsymbol{k}[Q_8]$ and $K_8$ are given by $V_{ij}=\boldsymbol{k}\ (i,j=0,1)$ equipped with the left actions $\rho_{ij}$ defined by 
$$\rho_{ij}(s)=(-1)^i,\qquad \rho_{ij}(t)=(-1)^j.$$
For both $\boldsymbol{k}[D_8]$ and  $K_8$ a $2$-dimensional simple module, which is unique up to isomorphism, is given by $V=\boldsymbol{k}\oplus \boldsymbol{k}$ equipped with the left action $\rho$ defined by 
$$\rho (s)=\begin{pmatrix} 0 & -1 \\ 1 & 0 \end{pmatrix},\quad 
\rho (t)=\begin{pmatrix} -1 & 0 \\ 0 & 1 \end{pmatrix}, $$
and for $\boldsymbol{k}[Q_8]$ that is given by $V=\boldsymbol{k}\oplus \boldsymbol{k}$ equipped with the left action $\rho$  defined by 
$$\rho (s)=\begin{pmatrix} \zeta  & 0 \\ 0 & \zeta ^{-1} \end{pmatrix},\quad 
\rho (t)=\begin{pmatrix} 0 & -1 \\ 1 & 0 \end{pmatrix},$$
where $\zeta$ is a primitive $4$-th roof of unity in $\boldsymbol{k}$. 
\par 
Suppose that $\boldsymbol{k}$ contains a primitive $8$-th root of unity $\omega $. Then by the above two lemmas we have $P_{\boldsymbol{k}[D_8], V_{ij}}(x)=P_{\boldsymbol{k}[Q_8], V_{ij}}(x)=(x-1)^8,\ P_{K_8, V_{ij}}(x)=(x+(-1)^i)^8$ and 
\begin{align*}
P_{\boldsymbol{k}[D_8], V}(x)&=(x-1)^3(x+1)^3(x-\zeta )(x+\zeta )=x^8-2x^6+2x^2-1, \\ 
P_{\boldsymbol{k}[Q_8], V}(x)&=(x-1)(x+1)(x-\zeta )^3(x+\zeta )^3=x^8+2x^6-2x^2-1, \\ 
P_{K_8, V}(x)&=\prod\limits_{p,q=0,1}(x-(-1)^p)\cdot \prod\limits_{l=0}^3(x-\omega ^{2l-1})=x^8-2x^6+2x^4-2x^2+1.   
\end{align*}
Therefore, the polynomial invariants of $\boldsymbol{k}[D_8], \boldsymbol{k}[Q_8], K_8$ are given by 
\begin{align*}
P_{\boldsymbol{k}[D_8]}^{(1)}(x)=P_{\boldsymbol{k}[Q_8]}^{(1)}(x)=&(x-1)^{32},\qquad P_{K_8}^{(1)}(x)=(x-1)^{16}(x+1)^{16}, \displaybreak[1] \\ 
P_{\boldsymbol{k}[D_8]}^{(2)}(x)&=x^8-2x^6+2x^2-1, \\ 
P_{\boldsymbol{k}[Q_8]}^{(2)}(x)&=x^8+2x^6-2x^2-1, \\ 
P_{K_8}^{(2)}(x)&=x^8-2x^6+2x^4-2x^2+1. 
\end{align*}
Since polynomials $P_{\boldsymbol{k}[D_8]}^{(2)}(x),\ P_{\boldsymbol{k}[Q_8]}^{(2)}(x),\ P_{K_8}^{(2)}(x)$ are mutually distinct, we conclude that the Hopf algebras $\boldsymbol{k}[D_8], \boldsymbol{k}[Q_8], K_8$ are not mutually monoidally Morita equivalent by Theorem~\ref{2.5}. 

\par \medskip 
\subsection{The Hopf algebra $A_{Nn}^{\nu \lambda }$} 
Satoshi Suzuki  introduced a family of cosemisimple Hopf algebras of finite dimension parametrized by $\nu , \lambda , N, n$, where $\nu , \lambda =\pm1 $, and $N\geq 1$ and $n\geq 2$ are integers. 
This family includes not only the Kac-Paljutkin algebra $K_8$, but also Hopf algebras which can be regarded as a generalization of $K_8$.  
In this subsection we compute the polynomial invariants for Suzuki's Hopf algebras. 
\par 
Let us recall the definition of Suzuki's Hopf algebras $A_{Nn}^{\nu \lambda }$  \cite{Suz}. 
Let $\boldsymbol{k}$ be a field of $\textrm{ch}(\boldsymbol{k})\not= 2$, which contains a primitive $4nN$-th root of unity, and let $C$ be the $2\times 2$-matrix coalgebra over $\boldsymbol{k}$. 
By definition $C$ has a basis $\{ X_{11},X_{12},X_{21},X_{22}\} $ 
which satisfies the equation 
$$\Delta (X_{ij})=X_{i1}\otimes X_{1j}+X_{i2}\otimes X_{2j},\qquad \varepsilon (X_{ij})=\delta _{ij}.$$
Since $C$ is a coalgebra, the tensor algebra $\mathcal{T}(C)$ of $C$ has  a bialgebra structure in a natural way. 
Let $I$ be the  coideal of $\mathcal{T}(C)$ defined by 
$$I=\boldsymbol{k}(X_{11}^2-X_{22}^2)+\boldsymbol{k}(X_{12}^2-X_{21}^2)+\sum\limits_{i-j\not\equiv l-m\ (\textrm{mod}\ 2)}\boldsymbol{k}(X_{ij}X_{lm}),$$
and consider the quotient bialgebra $B:=\mathcal{T}(C)/\langle I\rangle $. 
We write $x_{ij}$ for the image of  $X_{ij}$  under the natural projection $\mathcal{T}(C)\longrightarrow B$.  
We fix $N\geq 1,\ n\geq 2$ and $\nu ,\lambda =\pm 1$, 
and consider the following subspace  $J_{Nn}^{\nu \lambda }$ of $B$: 
\begin{equation}
\begin{aligned}
J_{Nn}^{\nu \lambda }:=&\boldsymbol{k}(x_{11}^{2N}+\nu x_{12}^{2N}-1) \notag \\ 
+&\boldsymbol{k}(\underbrace{x_{11}x_{22}x_{11}\ldots \ldots }_{\textrm{$n$ }}-\underbrace{x_{22}x_{11}x_{22}\ldots \ldots }_{\textrm{$n$}})+\boldsymbol{k}(-\lambda \underbrace{x_{12}x_{21}x_{12}\ldots \ldots }_{\textrm{$n$}}+\underbrace{x_{21}x_{12}x_{21}\ldots \ldots }_{\textrm{$n$}}). \notag 
\end{aligned}
\end{equation}
Here, 
$$\underbrace{x_{11}x_{22}x_{11}\ldots \ldots }_{\textrm{$n$}}=\begin{cases} (x_{11}x_{22})^{\frac{n}{2}} \ \ & \textrm{if \ $n$ is even}, \\ 
(x_{11}x_{22})^{\frac{n-1}{2}}x_{11} \ \ & \textrm{if \ $n$ is odd}, \end{cases}  $$ 
and the other notation has the same meaning of that, too. 
Since the subspace $J_{Nn}^{\nu \lambda }$ is a coideal of $B$, we obtain the quotient bialgebra $A_{Nn}^{\nu \lambda }:=B/\langle J_{Nn}^{\nu \lambda }\rangle $. 
This bialgebra $A_{Nn}^{\nu \lambda }$ becomes a $4nN$-dimensional cosemisimple Hopf algebra over $\boldsymbol{k}$.  
For the image of  $x_{ij}$  under the natural projection  $\pi :B\longrightarrow A_{Nn}^{\nu \lambda }$ we write $x_{ij}$, again.  
Then $A_{Nn}^{\nu \lambda }$ equips with the basis 
\begin{equation}\label{eq5.1}
\{ x_{11}^s\overbrace{x_{22}x_{11}x_{22}\ldots \ldots }^t,\ x_{12}^s\overbrace{x_{21}x_{12}x_{21}\ldots \ldots }^t\ \vert \ 1\leq s\leq 2N,\ 0\leq t\leq n-1\}  ,
\end{equation}
and the Hopf algebra structure is given by 
$$\Delta (x_{ij})=x_{i1}\otimes x_{1j}+x_{i2}\otimes x_{2j}, \quad 
\varepsilon (x_{ij})=\delta _{ij}, \quad 
S(x_{ij})=x_{ji}^{2N-1}. $$
If $\textrm{ch}(\boldsymbol{k})\kern-0.2em \nmid 2nN $, then $A_{Nn}^{\nu \lambda }$ is semisimple \cite[Theorem 3.1(viii)]{Suz}. 
Moreover, $A_{12}^{+-}$ is isomorphic to the Kac-Paljutkin algebra $K_8$, and the Hopf algebras $A_{1n}^{++}$ and $A_{1n}^{+-}$ are isomorphic to the Hopf algebras $\mathcal{A}_{4n}$ and $\mathcal{B}_{4n}$, respectively, that are introduced by Masuoka \cite{Mas2, CDMM}. 
The following lemma, which is a direct consequence of the definition of  $A_{Nn}^{\nu \lambda }$, is useful for computation.  

\par \smallskip 
\begin{lem}\label{5.1}
For each $i,j=1,2$,  the square $x_{ij}^2$ is in the center of $A_{Nn}^{\nu \lambda }$, and the following equations hold in $A_{Nn}^{\nu \lambda }$. 
\par 
(1) $x_{11}^2=x_{22}^2,\  x_{12}^2=x_{21}^2,\  x_{ij}x_{lm}=0 \ (i-j\not\equiv l-m\ (\textrm{mod}\ 2))$. 
\par 
(2) If $n$ is even, then 
$(x_{22}x_{11})^{\frac{n}{2}}= (x_{11}x_{22})^{\frac{n}{2}}$, $(x_{21}x_{12})^{\frac{n}{2}}=\lambda (x_{12}x_{21})^{\frac{n}{2}}$. 
\par 
If $n$ is odd, then 
$(x_{22}x_{11})^{\frac{n-1}{2}}x_{22}=(x_{11}x_{22})^{\frac{n-1}{2}}x_{11}$, $(x_{21}x_{12})^{\frac{n-1}{2}}x_{21}= \lambda (x_{12}x_{21})^{\frac{n-1}{2}}x_{12}$.  
\par 
(3) $x_{ii}^{2N+1}=x_{ii},\ x_{i\kern0.2em i+1}^{2N+1}=\nu x_{i\kern0.2em i+1}$.
\par 
(4) $x_{11}^{4N}+x_{12}^{4N}=1$. 
\par 
(5) $(x_{11}x_{22})^n=x_{11}^{2n},\ (x_{21}x_{12})^n=\lambda x_{12}^{2n}$.   \end{lem}

\par \smallskip 
To describe the simple right $A_{Nn}^{\nu \lambda }$-comodules, we use the following notations. 
For $m\geq 1$, we set 
$$\chi _{11}^m:=\overbrace{x_{11}x_{22}x_{11}\ldots\ldots }^{\textrm{$m$ }}\ ,\ \chi _{22}^m:=\overbrace{x_{22}x_{11}x_{22}\ldots\ldots }^{\textrm{$m$ }},$$
$$\chi _{12}^m:=\overbrace{x_{12}x_{21}x_{12}\ldots\ldots }^{\textrm{$m$ }}\ ,\ \chi _{21}^m:=\overbrace{x_{21}x_{12}x_{21}\ldots\ldots }^{\textrm{$m$ }}.$$
Then the basis (\ref{eq5.1}) of $A_{Nn}^{\nu \lambda}$ can be represented by 
\begin{equation}\label{eq5.2}
\{ \ x_{11}^s\chi _{22}^t,\ x_{12}^s\chi _{21}^t\ \vert \ 1\leq s\leq 2N,\ 0\leq t\leq n-1\ \}  ,
\end{equation}
and the following equations hold: 
\begin{equation}\label{eq5.3}
\Delta (\chi_{ij}^m)=\chi_{i1}^m\otimes \chi_{1j}^m+\chi_{i2}^m\otimes \chi_{2j}^m\qquad (m\geq 1,\ i,j=1,2).
\end{equation}
Thus for $s,t\geq 0$ with $s+t\geq 1$, 
\begin{align*}
\Delta (x_{11}^s\chi _{22}^t)
&=x_{11}^s\chi _{22}^t\otimes x_{11}^s\chi _{22}^t
+x_{12}^s\chi _{21}^t\otimes x_{21}^s\chi _{12}^t,\\ 
\Delta (x_{12}^s\chi _{21}^t)
&=x_{11}^s\chi _{22}^t\otimes x_{12}^s\chi _{21}^t
+x_{12}^s\chi _{21}^t\otimes x_{22}^s\chi _{11}^t. 
\end{align*}
Furthermore,  
\begin{align*}
S(x_{11}^s\chi_{22}^t)
&=\begin{cases}
x_{11}^{(2N-2)(t+s)+s}\chi_{11}^t \quad (\textrm{$s, t$ are even}), \\ 
x_{11}^{(2N-2)(t+s)+s}\chi_{22}^t \quad (\textrm{$s$ is odd, and $t$ is even}) , \\ x_{11}^{(2N-2)(t+s)+s}\chi_{22}^t \quad (\textrm{$s$ is even, and $t$ is odd}), \\ 
x_{22}^{(2N-2)(t+s)+s}\chi_{11}^t \quad (\textrm{$s,t$ are odd}), 
\end{cases} \displaybreak[0] \\ 
S(x_{12}^s\chi_{21}^t)
&=\begin{cases}
x_{12}^{(2N-2)(t+s)+s}\chi_{21}^t \quad (\textrm{$s,t$ are even}), \\ 
x_{21}^{(2N-2)(t+s)+s}\chi_{12}^t \quad (\textrm{$s$ is odd, and $t$ is even}), \\ x_{12}^{(2N-2)(t+s)+s}\chi_{12}^t \quad (\textrm{$s$ is even, and $t$ is odd}), \\ 
x_{12}^{(2N-2)(t+s)+s}\chi_{21}^t \quad (\textrm{$s,t$ are odd}). 
\end{cases} 
\end{align*}

\par 
Hereafter to the end of this subsection we suppose that $N\geq 1,\ n\geq 2$ and $\nu ,\lambda =\pm 1$, and $\boldsymbol{k}$ is a field of $\textrm{ch}(\boldsymbol{k})\not= 2$, which contains a primitive $4nN$-th root of unity. 
The following proposition was proved by S. Suzuki. 

\par 
\begin{prop}[{\cite[Theorem 3.1]{Suz}}]\label{5.2}\ 
\par 
(1) The dimension of a simple subcoalgebra of $A_{Nn}^{\nu \lambda }$ is $1$ or $4$. 
\par 
(2) The order of the group $G:=G(A_{Nn}^{\nu \lambda })$ is $4N$, and $G$ is explicitly given by 
$$G=\{ \ 
x_{11}^{2s}\pm x_{12}^{2s},\ 
x_{11}^{2s+1}\chi _{22}^{n-1}\pm \sqrt{\lambda}x_{12}^{2s+1}\chi_{21}^{n-1}\ \vert \ 1\leq s\leq N\ \}.$$
\par 
(3) There are exactly $N(n-1)$ simple subcoalgebras of dimension $4$, and they are given by 
$$C_{st}=\boldsymbol{k}x_{11}^{2s}\chi_{11}^t+\boldsymbol{k}x_{12}^{2s}\chi_{12}^t+\boldsymbol{k}x_{12}^{2s}\chi_{21}^t+\boldsymbol{k}x_{11}^{2s}\chi_{22}^t \quad (0\leq s\leq N-1,\ 1\leq t\leq n-1).$$
Therefore, the cosemisimple Hopf algebra $A_{Nn}^{\nu \lambda }$ is decomposed to the direct sum of  simple subcoalgebras such as 
$$A_{Nn}^{\nu \lambda }=\bigoplus\limits_{g\in G}\boldsymbol{k}g\ \ \oplus \bigoplus\limits_{\substack{0\leq s\leq N-1\\[0.1cm] 1\leq t\leq n-1}}C_{st}.$$
 \end{prop}

\par 
\begin{cor}\label{5.3}
The set 
$$\{ \ \boldsymbol{k}g\ \vert \ g\in G(A_{Nn}^{\nu \lambda })\ \} 
\cup \{ \ \boldsymbol{k}x_{11}^{2s}\chi_{22}^t+\boldsymbol{k}x_{12}^{2s}\chi_{21}^t  \ \vert \ 0\leq s\leq N-1,\ 1\leq t\leq n-1\ \} $$
is a full set of non-isomorphic (absolutely) simple right $A_{Nn}^{\nu \lambda }$-comodules, where the coactions of the comodules listed above are given by  restrictions of the coproduct $\Delta$. 
\end{cor}
\begin{proof}
Since $A_{Nn}^{\nu \lambda }$ is cosemisimple, 
a full set of non-isomorphic simple right $A_{Nn}^{\nu \lambda }$-comodules can be obtained by taking a simple right $D$-comodule for each simple subcoalgebra $D$ of $A_{Nn}^{\nu \lambda }$ and by collecting them. 
\par 
Let $D$ be a simple subcoalgebra of $A_{Nn}^{\nu \lambda }$. 
If $\dim D=1$, then $D=\boldsymbol{k}g$ for some $g\in G(A_{Nn}^{\nu \lambda })$. 
So, we can take $D$ itself as a simple right $D$-comodule. 
If $\dim D=4$, then by Proposition~\ref{5.2} (3), 
$D=C_{st}$ for some $s\in \{0,1,\ldots , N-1 \}$ and $t\in \{ 1,\ldots , n-1 \}$.  
Let us consider the subspace 
$$V_{st}:=\boldsymbol{k}x_{11}^{2s}\chi_{22}^t+\boldsymbol{k}x_{12}^{2s}\chi_{21}^t$$
of $D$. 
Then the subspace $V_{st}$ is a right $D$-comodule. 
For 
\begin{align*}
\Delta (x_{11}^{2s}\chi_{22}^t)
&=x_{11}^{2s}\chi_{22}^t\otimes x_{11}^{2s}\chi_{22}^t
+x_{12}^{2s}\chi_{21}^t\otimes x_{21}^{2s}\chi_{12}^t, \\ 
\Delta (x_{12}^{2s}\chi_{21}^t)
&=x_{11}^{2s}\chi_{22}^t\otimes x_{12}^{2s}\chi_{21}^t
+x_{12}^{2s}\chi _{21}^t\otimes x_{22}^{2s}\chi_{11}^t,  
\end{align*}
and $x_{21}^{2s}\chi_{12}^t=x_{12}^{2s}\chi_{12}^t,\ x_{22}^{2s}\chi_{11}^t =x_{11}^{2s}\chi_{11}^t\in C_{st}$. 
Furthermore, the right $D$-comodule $V_{st}$ is simple, since 
there is no non-zero element $v\in V_{st}$ such that $\Delta (v)\in \boldsymbol{k}v\otimes C_{st}$. 
\end{proof}

\par \bigskip 
Let us explain on the braiding structures of $A_{Nn}^{\nu \lambda }$, which is determined by S. Suzuki~\cite{Suz}. 
Suppose that $\alpha ,\beta \in \boldsymbol{k}$ satisfy $(\alpha \beta )^N=\nu ,\  (\alpha \beta ^{-1})^n=\lambda $. 
Then there is a braiding $\sigma _{\alpha \beta }$ of $A_{Nn}^{\nu \lambda }$ such that the values $\sigma _{\alpha \beta }(x_{ij},x_{kl})\ (i,j,k,l=1,2)$ are  given by the list below. 

\begin{center}
\begin{tabular}{c|cccc}
$x\diagdown y$ & $x_{11}$ & $x_{12}$ & $x_{21}$ & $x_{22}$ \\
\hline 
 $x_{11}$  & $0$ & $0$ & $0$ & $0$ \\
 $x_{12}$  & $0$ & $\alpha $ & $\beta $ & $0$ \\
 $x_{21}$  & $0$ & $\beta $ & $\alpha $ & $0$ \\
 $x_{22}$  & $0$ & $0$ & $0$ & $0$ \\ 
\end{tabular}
\end{center}

In fact, by using the braiding conditions (B2) and (B3) described in Section 4, repeatedly, we see that the values of $\sigma _{\alpha \beta }$ on the basis (\ref{eq5.2}) are given as follows: for integers $s, s', t, t'\geq 0$ with $s+t\geq 1$ and $s^{\prime}+t^{\prime}\geq 1$, and $j^{\prime}=1,2$, 
\begin{align*}
&\sigma _{\alpha \beta }(x_{11}^s\chi _{22}^t,\ x_{1j'}^{s'}\chi _{2,j'+1}^{t'}) \\ 
&\qquad =\begin{cases}
\delta _{0,t'}\delta _{1+j',t}(\alpha \beta )^{\frac{t's+(s+t)s'}{2}}\beta ^{tt'} & \textrm{($s$ and $s'$ are even),} \\ 
\delta _{1,t'}\delta _{1+j',t}(\alpha \beta )^{\frac{t's+(s+t)s'-t}{2}}\alpha ^{t(t'+1)} & \textrm{($s$ is even, and $s'$ is odd),} \\ 
\delta _{0,t'}\delta _{j',t}(\alpha \beta )^{\frac{t's+(s+t)s'-t'}{2}}\alpha ^{(t+1)t'} & \textrm{($s$ is odd, and $s'$ is even),} \\ 
\delta _{1,t'}\delta _{j',t}(\alpha \beta )^{\frac{t's+(s+t)s'-(1+t+t')}{2}}\beta ^{(t+1)(t'+1)} 
& \textrm{($s$ and $s'$ are odd)},  
\end{cases} \displaybreak[0] \\ 
&\sigma _{\alpha \beta }(x_{12}^s\chi _{21}^t,\ x_{1j'}^{s'}\chi _{2,j'+1}^{t'}) \\ 
&\qquad =\begin{cases}
\delta _{1,t'}\delta _{1+j',t}(\alpha \beta )^{\frac{t's+(s+t)s'}{2}} \alpha ^{tt'} & \textrm{($s$ and $s'$ are even),} \\ 
\delta _{0,t'}\delta _{1+j',t}(\alpha \beta )^{\frac{t's+(s+t)s'-t}{2}} \beta ^{t(t'+1)}  & \textrm{($s$ is even, and $s'$ is odd),} \\ 
\delta _{1,t'}\delta _{j',t}(\alpha \beta )^{\frac{t's+(s+t)s'-t'}{2}}\beta ^{(t+1)t'} & \textrm{($s$ is odd, and $s'$ is even),} \\ 
\delta _{0,t'}\delta _{j',t}(\alpha \beta )^{\frac{t's+(s+t)s'-(1+t+t')}{2}} \alpha ^{(t+1)(t'+1)} & \textrm{($s$ and $s'$ are odd)},  
\end{cases}
\end{align*}
where the indices of $\chi$ and $\delta $ are treated as modulo $2$.  
\par 
When $n=2$, in addition to the above braidings $\sigma _{\alpha \beta }$, 
for $\gamma , \xi \in \boldsymbol{k}$ which satisfy $\gamma ^2=\xi ^2,\ \gamma ^{2N}=1$ there is a braiding $\tau _{\gamma \xi }$ of $A_{N2}^{\nu \lambda }$ such that the values $\tau _{\gamma \xi }(x_{ij},x_{kl})\ (i,j,k,l=1,2)$ are given by the list below, where $x_{ij}$ and $x_{kl}$ correspond to a row and a column, respectively. 

\begin{center}
\begin{tabular}{c|cccc}
$x\diagdown y$ & $x_{11}$ & $x_{12}$ & $x_{21}$ & $x_{22}$ \\
\hline 
 $x_{11}$  & $\gamma $ & $0$ & $0$ & $\xi $ \\
 $x_{12}$  & $0$ & $0$ & $0$ & $0$ \\
 $x_{21}$  & $0$ & $0$ & $0$ & $0$ \\
 $x_{22}$  & $\lambda \xi $ & $0$ & $0$ & $\gamma $ \\ 
\end{tabular}
\end{center} 

We see also that the values of $\tau _{\gamma \xi }$ on the basis (\ref{eq5.2}) are given as follows: 
for integers $s, s', t, t'\geq 0$ with $s+t\geq 1$ and $s^{\prime}+t^{\prime}\geq 1$, and $j^{\prime}=1,2$, 
\begin{align*}
&\tau _{\gamma \xi }(x_{11}^s\chi _{22}^t,\ x_{1j'}^{s'}\chi _{2,j'+1}^{t'}) \\ 
&\quad =\begin{cases}
\delta _{j^{\prime}1}\gamma ^{ss^{\prime}}
(\gamma \xi )^{\frac{tt^{\prime}}{4}+\frac{st^{\prime}}{2}}
(\gamma \lambda \xi )^{\frac{tt^{\prime}}{4}+\frac{ts^{\prime}}{2}}
& (\textrm{$t$ and $t^{\prime}$ are even}),\\
\delta _{j^{\prime}1}\gamma ^{ss^{\prime}}
(\lambda \xi )^{s^{\prime}}
(\gamma \xi )^{\frac{(t-1)t^{\prime}}{4}+\frac{st^{\prime}}{2}}
(\gamma \lambda \xi )^{\frac{(t+1)t^{\prime}}{4}+\frac{(t-1)s^{\prime}}{2}} 
& (\textrm{$t$ is odd, and $t^{\prime}$ is even}),\\
\delta _{j^{\prime}1}\gamma ^{ss^{\prime}}\xi ^s
(\gamma \xi )^{\frac{t(t^{\prime}+1)}{4}+\frac{s(t^{\prime}-1)}{2}}
(\gamma \lambda \xi )^{\frac{t(t^{\prime}-1)}{4}+\frac{ts^{\prime}}{2}} 
& (\textrm{$t$ is even, and $t^{\prime}$ is odd}),\\
\delta _{j^{\prime}1}\gamma ^{ss^{\prime}+1} \xi ^s
(\lambda \xi)^{s^{\prime}}
(\gamma \xi )^{\frac{(t-1)(t^{\prime}+1)}{4}+\frac{s(t^{\prime}-1)}{2}}
(\gamma \lambda \xi )^{\frac{(t+1)(t^{\prime}-1)}{4}+\frac{(t-1)s^{\prime}}{2}}
& (\textrm{$t$ and $t^{\prime}$ are odd}), 
\end{cases}\\ 
&\tau _{\gamma \xi }(x_{12}^s\chi _{21}^t,\ x_{1j'}^{s'}\chi _{2,j'+1}^{t'})
=0, 
\end{align*}
where the indices of $\chi$ and $\delta $ are treated as modulo $2$.  

\par \smallskip 
\begin{thm}[{\cite[Proposition 3.10]{Suz}}]\label{5.5}
If $n\geq 3$, then the braidings of $A_{Nn}^{\nu \lambda }$ are given by
$$\{ \ \sigma _{\alpha \beta }\ \vert \ \alpha , \beta \in \boldsymbol{k},\ (\alpha \beta )^N=\nu ,\ (\alpha \beta ^{-1})^n=\lambda \ \},  $$
and if $n=2$, then that are given by
$$\{ \ \sigma _{\alpha \beta }\ \vert \ \alpha , \beta \in \boldsymbol{k},\ (\alpha \beta )^N=\nu ,\ (\alpha \beta ^{-1})^2=\lambda \ \} \cup \{ \ \tau _{\gamma \xi }\ \vert \ \gamma , \xi \in \boldsymbol{k},\ \gamma ^2=\xi ^2,\ \gamma ^{2N}=1\ \} . 
$$
 \end{thm}

\par \smallskip 
\begin{lem}\label{5.8}
\par 
(1) For $\alpha , \beta \in \boldsymbol{k}$ satisfying $(\alpha \beta )^N=\nu ,\ (\alpha \beta ^{-1})^n=\lambda$,  
the Drinfel'd element $\mu _{\alpha \beta }$ of the 
braided Hopf algebra $(A_{Nn}^{\nu \lambda }, \sigma _{\alpha \beta })$ is given as follows: for all integers $s, t\geq 0$ with $s+t\geq 1$, 
\begin{align*}
\mu _{\alpha \beta }(x_{11}^s\chi_{22}^t)
&=\begin{cases}
\nu ^t(\alpha \beta )^{\frac{-s^2}{2}-st-t^2}\alpha ^{t^2} &\quad (\textrm{$s$ is even}),\\ 
\nu ^{t+1}
(\alpha \beta )^{\frac{-1-s^2}{2}-t-st-t^2}\alpha ^{(t+1)^2}
&\quad (\textrm{$s$ is odd}), 
\end{cases} \\ 
\mu _{\alpha \beta }(x_{12}^s\chi_{21}^t)
&=0. 
\end{align*}
(2) For $\gamma , \xi \in \boldsymbol{k}$ satisfying $\gamma ^2=\xi ^2,\ \gamma ^{2N}=1$, the Drinfel'd element 
$\mu _{\gamma \xi }$ of the braided Hopf algebra $(A_{N2}^{\nu \lambda }, \tau _{\gamma \xi })$ is given as follows: for all integers $s, t\geq 0$ with $s+t\geq 1$, 
\begin{align*}
\mu _{\gamma \xi }(x_{11}^s\chi_{22}^t)
&=\begin{cases}
\gamma ^{-s^2-2st-t^2}
\lambda ^{\frac{t^2}{4}} & \quad (\textrm{$s$ and $t$ are even}),\\ 
\gamma ^{-s^2-2st-t^2}
\lambda ^{\frac{t^2-1}{4}}
 & \quad (\textrm{$t$ is odd, and $s$ is even}), \\ 
\gamma ^{-s^2-2st-t^2}
\lambda ^{\frac{t(t+2)}{4}}  & \quad (\textrm{$t$ is even, and $s$ is odd}), \\ 
\gamma ^{-s^2-2st-t^2}
\lambda ^{\frac{(t+1)^2}{4}}  & \quad (\textrm{$s$ and $t$ are odd}),
\end{cases}\\ 
\mu _{\gamma \xi }(x_{12}^s\chi_{21}^t)
&=0 .
\end{align*}
\end{lem}
\begin{proof}
In the case where $s$ and $t$ are even, and $t\geq 2$, 
the values $\mu _{\alpha \beta }(a)$ and $\mu _{\gamma \xi }(a)$ for $a=x_{11}^s\chi_{22}^t,\ x_{12}^s\chi_{21}^t$ can be calculated as follows. 
\begin{align*}
\mu _{\alpha \beta }(x_{11}^s\chi_{22}^t)
&=\sigma _{\alpha \beta }(x_{11}^s\chi_{22}^t,\ x_{11}^{(2N-2)(s+t)+s+1}\chi_{22}^{t-1})
+\sigma _{\alpha \beta }(x_{12}^{s+1}\chi_{21}^{t-1},\ x_{12}^{(2N-2)(s+t)+s}\chi_{21}^t) \\ 
&=(\alpha \beta )^{\frac{(t-1)s+(s+t)((2N-2)(s+t)+s+1)-t}{2}}\alpha ^{t^2} \\ 
&=(\alpha \beta )^{(s+t)^2N}
(\alpha \beta )^{-\frac{s^2}{2}-st-t^2}\alpha ^{t^2} \\ 
&=(\alpha \beta )^{-\frac{s^2}{2}-st-t^2}\alpha ^{t^2}, \displaybreak[0] \\[0.1cm]  
\mu _{\alpha \beta }(x_{12}^s\chi_{21}^t)
&=\sigma _{\alpha \beta }(x_{12}^s\chi_{21}^t,\ x_{11}^{(2N-2)(s+t)+s+1}\chi_{22}^{t-1})
+\sigma _{\alpha \beta }(x_{11}^{s+1}\chi_{22}^{t-1},\ x_{12}^{(2N-2)(s+t)+s}\chi_{21}^t) \\ 
&=0, \displaybreak[0] \\[0.1cm]   
\mu _{\gamma \xi }(x_{11}^s\chi_{22}^t)
&=\tau _{\gamma \xi }(x_{11}^s\chi_{22}^t,\ x_{11}^{(2N-2)(s+t)+s+1}\chi_{22}^{t-1})
+\tau _{\gamma \xi }(x_{12}^{s+1}\chi_{21}^{t-1},\ x_{12}^{(2N-2)(s+t)+s}\chi_{21}^t) \\ 
&=\gamma ^{s((2N-2)(s+t)+s+1)}\xi ^s
(\gamma \xi )^{\frac{t^2}{4}+\frac{s(t-2)}{2}}
(\gamma \lambda \xi )^{\frac{t(t-2)}{4}+\frac{t((2N-2)(s+t)+s+1)}{2}} \\ 
&=\gamma ^{-2(s+t)^2+s^2+ts+\frac{t(t+s)}{2}}
\xi ^{\frac{t(t+s)}{2}}
\lambda ^{\frac{t^2}{4}}  \\ 
&=\gamma ^{(s+t)(-s-t)}\lambda ^{\frac{t^2}{4}}, \displaybreak[0] \\[0.1cm]   
\mu _{\gamma \xi }(x_{12}^s\chi_{21}^t)
&=\tau _{\gamma \xi }(x_{12}^s\chi_{21}^t,\ x_{11}^{(2N-2)(s+t)+s+1}\chi_{22}^{t-1})
+\tau _{\gamma \xi }(x_{11}^{s+1}\chi_{22}^{t-1},\ x_{12}^{(2N-2)(s+t)+s}\chi_{21}^t) \\ 
&=0.
\end{align*}

In other cases, the values $\mu _{\alpha \beta }(a)$ and $\mu _{\gamma \xi }(a)$ for $a=x_{11}^s\chi_{22}^t, x_{12}^s\chi_{21}^t$ can be  calculated by a similar manner, so we leave the rest of calculation to the reader. 
\end{proof}

\par \medskip 
By using  Lemma~\ref{5.8}, we know the braided dimensions of the simple right $A_{Nn}^{\nu \lambda }$-comodules. 

\par 
\begin{lem}\label{5.9}
\par 
(1) Let $\alpha , \beta $ be  elements in $\boldsymbol{k}$ satisfying $(\alpha \beta )^N=\nu $ and $ (\alpha \beta ^{-1})^n=\lambda$. 
\par 
(i) For an element $g\in G(A_{Nn}^{\nu \lambda })$ the character $\chi_g\in A_{Nn}^{\nu \lambda }$ of the simple right 
$A_{Nn}^{\nu \lambda }$-comodule $\boldsymbol{k}g$ is given by 
$\chi_g=g$, and the braided dimension $\underline{\dim}_{\kern0.1em \sigma _{\alpha \beta}} \boldsymbol{k}g $ with respect to the braiding $\sigma _{\alpha \beta}$ is given by 
$$\underline{\dim}_{\kern0.1em \sigma _{\alpha \beta}} \boldsymbol{k}g 
=\begin{cases}
(\alpha \beta )^{-2s^2} & (\textrm{$g=x_{11}^{2s}\pm x_{12}^{2s} \ (1\leq s\leq N)$}), \\[0.1cm]  
\nu ^{n}
(\alpha \beta )^{-2s^2-2sn-n^2}\alpha ^{n^2} & (\textrm{$g=
x_{11}^{2s+1}\chi _{22}^{n-1}\pm \sqrt{\lambda}x_{12}^{2s+1}\chi_{21}^{n-1}\ (1\leq s\leq N)$}).
\end{cases} 
$$
\indent 
(ii) For the simple  right $A_{Nn}^{\nu \lambda }$-comodule 
$V_{st}=\boldsymbol{k}x_{11}^{2s}\chi_{22}^t+\boldsymbol{k}x_{12}^{2s}\chi_{21}^t\ (0\leq s\leq N-1,\ 1\leq t\leq n-1)$, the character $\chi_{st}\in A_{Nn}^{\nu \lambda }$ of $V_{st}$ is given by $\chi_{st}=x_{11}^{2s+1}\chi_{22}^{t-1}+x_{11}^{2s}\chi_{22}^t$, and the braided dimension $\underline{\dim}_{\kern0.1em \sigma _{\alpha \beta}} V_{st}$ is given by 
$$\underline{\dim}_{\kern0.1em \sigma _{\alpha \beta}} V_{st}=2\nu ^{t}\alpha ^{t^2}(\alpha \beta )^{-2s^2-2st-t^2}.$$
(2) Let $\gamma , \xi $ be  elements in $\boldsymbol{k}$ such that $\gamma ^2=\xi ^2,\ \gamma ^{2N}=1$. 
\par 
(i) For an element $g\in G(A_{N2}^{\nu \lambda })$ the braided dimension $\underline{\dim}_{\kern0.1em \tau _{\gamma \xi}} \boldsymbol{k}g$ is given by 
$$\underline{\dim}_{\kern0.1em \tau _{\gamma \xi}} \boldsymbol{k}g 
=\begin{cases}
\gamma ^{-4s^2} & (\textrm{$g=x_{11}^{2s}\pm x_{12}^{2s} \ (1\leq s\leq N)$}),  \\[0.1cm]  
\gamma ^{-4(s+1)^2}\lambda 
 & (\textrm{$g=
x_{11}^{2s+1}\chi _{22}\pm \sqrt{\lambda}x_{12}^{2s+1}\chi_{21}\ (1\leq s\leq N)$}).
\end{cases} 
$$
\indent 
(ii) For the simple  right $A_{N2}^{\nu \lambda }$-comodule 
$V_{s1}=\boldsymbol{k}x_{11}^{2s}x_{22}+\boldsymbol{k}x_{12}^{2s}x_{21}\ (0\leq s\leq N-1)$ ,  the braided dimension $\underline{\dim}_{\kern0.1em \tau _{\gamma \xi}} V_{s1}$ is given by 
$$\underline{\dim}_{\kern0.1em \tau _{\gamma \xi}} V_{s1}=2\gamma ^{-(2s+1)^2}.$$  
\end{lem}

\par \smallskip 
\begin{prop}\label{5.10} 
Let $\omega \in \boldsymbol{k}$ be a primitive  $4nN$-th root of unity. 
\par 
(1) For each $g\in G(A_{Nn}^{\nu \lambda })$, the polynomial $P_{(A_{Nn}^{\nu \lambda })^{\ast}, \boldsymbol{k}g}(x)$ is given as follows. 
If $g=x_{11}^{2s}\pm x_{12}^{2s} \ (1\leq s\leq N)$, then 
$$P_{(A_{Nn}^{\nu \lambda })^{\ast}, \boldsymbol{k}g}(x)
=\begin{cases} 
\prod\limits_{i=0}^{N-1}
(x-\omega ^{-4n(2i+\frac{1-\nu }{2})s^2})^{2n} \quad & (n\geq 3), \vspace{0.1cm} \\  
\prod\limits_{i=0}^{N-1}
(x-\omega ^{-8(2i+\frac{1-\nu }{2})s^2})^{4} 
(x-\omega ^{-16is^2})^4 \quad & (n=2), \end{cases} $$
and if $g=x_{11}^{2s+1}\chi _{22}^{n-1}\pm \sqrt{\lambda}x_{12}^{2s+1}\chi_{21}^{n-1}\ (1\leq s\leq N)$, then 
$$P_{(A_{Nn}^{\nu \lambda })^{\ast}, \boldsymbol{k}g}(x)
=\begin{cases}
\prod\limits_{i=0}^{N-1}
(x^2-\omega ^{-2n(2s+n)^2(2i+\frac{1-\nu }{2})}
(-1)^{\frac{1-\lambda }{2}})^n \quad (\textrm{$n$ is odd}), \vspace{0.1cm} \\  
\prod\limits_{i=0}^{N-1}
(x-\omega ^{-n(2s+n)^2(2i+\frac{1-\nu }{2})}
(-1)^{\frac{n}{2}\frac{1-\lambda }{2}})^{2n} \quad (\textrm{$n\geq 4$ is even}), \vspace{0.1cm} \\  
\prod\limits_{i=0}^{N-1}
(x-\omega ^{-4(s+1)^2(2i+\frac{1-\nu }{2})}
(-1)^{\frac{1-\lambda }{2}})^{4} 
(x-\omega ^{-16i(s+1)^2}\lambda )^4 \ \ \ (n=2). 
\end{cases} $$
(2) For the simple right $A_{Nn}^{\nu \lambda }$-comodule $V_{st}=\boldsymbol{k}x_{11}^{2s}\chi_{22}^t+\boldsymbol{k}x_{12}^{2s}\chi_{21}^t\ (0\leq s\leq N-1,\ 1\leq t\leq n-1)$, the polynomial $P_{(A_{Nn}^{\nu \lambda })^{\ast}, V_{st}}(x)$ is given by 
$$P_{(A_{Nn}^{\nu \lambda })^{\ast}, V_{st}}(x)
=\begin{cases}
\prod\limits_{i=0}^{N-1}\prod\limits_{j=0}^{n-1}
(x^2-\omega ^{-2n(2s+t)^2(2i+\frac{1-\nu }{2})-2Nt^2(2j+\frac{\lambda -1}{2})}) \quad (\textrm{$n\geq 3$, and $t$ is odd}), \vspace{0.1cm} \\  
\prod\limits_{i=0}^{N-1}\prod\limits_{j=0}^{n-1}
(x-\omega ^{-n(2s+t)^2(2i+\frac{1-\nu }{2})
-Nt^2(2j+\frac{\lambda -1}{2})})^2
\quad (\textrm{$n\geq 3$, and $t$ is even}), \vspace{0.1cm} \\  
\prod\limits_{i=0}^{N-1}
(x^4-\omega ^{-8(2i+\frac{1-\nu }{2})(2s+1)^2+2N(1-\lambda)}) 
(x^2-\omega ^{-8i(2s+1)^2})^2  \ \ \ (n=2). 
\end{cases}$$
\end{prop}
\begin{proof}
The set $I_{\nu \lambda}=\{ \ (\alpha ,\beta )\in \boldsymbol{k}\times \boldsymbol{k}\ \vert \ (\alpha \beta )^N=\nu ,\ (\alpha \beta ^{-1})^n=\lambda \ \} $ can be expressed as 
\begin{align*}
I_{\nu \lambda}&=\Bigl\{ \ (\omega ^{n(2i+\frac{1-\nu }{2})+N(2j+\frac{1-\lambda }{2})},\ \omega ^{n(2i+\frac{1-\nu }{2})-N(2j+\frac{1-\lambda }{2})})\ \Bigl\vert \ \begin{matrix}\begin{array}{l} i=0,1,\ldots ,N-1,\\ j=0,1,\ldots ,n-1 \end{array} \end{matrix} \ \Bigr\} \\ 
& \quad \cup \Bigl\{ \ (-\omega ^{n(2i+\frac{1-\nu }{2})+N(2j+\frac{1-\lambda }{2})},\ -\omega ^{n(2i+\frac{1-\nu }{2})-N(2j+\frac{1-\lambda }{2})})\ \Bigl\vert \ \begin{matrix}\begin{array}{l} i=0,1,\ldots ,N-1,\\ j=0,1,\ldots ,n-1 \end{array} \end{matrix} \ \Bigr\} .   
\end{align*}
Because, for an element $(\alpha ,\beta )\in I_{\nu \lambda}$, $\alpha \beta $ and $\alpha \beta ^{-1}$ can be written as 
$$\begin{cases}
\alpha \beta =\omega ^{2n(2i+\frac{1-\nu }{2})} & \ \ \textrm{for some $i=0,1,\ldots ,N-1$},\\ 
\alpha \beta ^{-1}=\omega ^{2N(2j+\frac{1-\lambda}{2})} & \ \ \textrm{for some $j=0,1,\ldots ,n-1$}. 
\end{cases}$$
Similarly, we see that the set $J=\{ \ (\gamma , \xi )\in \boldsymbol{k}\times \boldsymbol{k}\ \vert \ \gamma ^2=\xi ^2,\ \gamma ^{2N}=1 \ \} $ 
can be expressed as 
$$J=\{ \ (\omega ^{4i},\ \omega ^{4i})\ \vert \ i=0,1,\ldots ,2N-1\ \} \cup \{ \ (\omega ^{4i},\ -\omega ^{4i})\ \vert \ i=0,1,\ldots ,2N-1\ \} .$$
We put $A=A_{Nn}^{\nu \lambda }$. 
\par 
(1) We consider the case of $g=x_{11}^{2s}\pm x_{12}^{2s}$. 
If $n\geq 3$, then  by Lemma~\ref{5.9}(1)(i)
\begin{align*}
P_{A^{\ast}, \boldsymbol{k}g}(x)
&=\prod\limits_{(\alpha ,\beta )\in I_{\nu \lambda}}
(x-(\alpha \beta )^{-2s^2}) \\ 
&=\prod\limits_{i=0}^{N-1}\prod\limits_{j=0}^{n-1}
(x-\omega ^{-4n(2i+\frac{1-\nu }{2})s^2})\cdot \prod\limits_{i=0}^{N-1}\prod\limits_{j=0}^{n-1}
(x-\omega ^{-4n(2i+\frac{1-\nu }{2})s^2}) \\ 
&=\prod\limits_{i=0}^{N-1}
(x-\omega ^{-4n(2i+\frac{1-\nu }{2})s^2})^{2n} , 
\end{align*}
and if $n=2$, then by Lemma~\ref{5.9}(2)(i) 
\begin{align*}
P_{A^{\ast}, \boldsymbol{k}g}(x)
&=\prod\limits_{(\alpha ,\beta )\in I_{\nu \lambda}}
(x-(\alpha \beta )^{-2s^2})
\cdot \prod\limits_{(\gamma ,\xi )\in J}
(x-\gamma ^{-4s^2}) \\ 
&=\prod\limits_{i=0}^{N-1}
(x-\omega ^{-8(2i+\frac{1-\nu }{2})s^2})^{4} 
\cdot 
\prod\limits_{i=0}^{2N-1}(x-\omega ^{-16is^2})^2 \\ 
&=\prod\limits_{i=0}^{N-1}
(x-\omega ^{-8(2i+\frac{1-\nu }{2})s^2})^{4} 
\cdot 
\prod\limits_{i=0}^{N-1}(x-\omega ^{-16is^2})^4. 
\end{align*}
Next we consider the case of $g=x_{11}^{2s+1}\chi _{22}^{n-1}\pm \sqrt{\lambda}x_{12}^{2s+1}\chi_{21}^{n-1}$. 
If $n\geq 3$, then by Lemma~\ref{5.9}(1)(i) 
\begin{align*}
P_{A^{\ast}, \boldsymbol{k}g}(x)
&=\prod\limits_{(\alpha ,\beta )\in I_{\nu \lambda}}
(x-\nu ^{n}(\alpha \beta )^{-2s^2-2sn-n^2}\alpha ^{n^2}) \displaybreak[0] \\ 
&=\prod\limits_{i=0}^{N-1}\prod\limits_{j=0}^{n-1}
(x-(-1)^{nj}\nu ^{n}\omega ^{-n(2i+\frac{1-\nu }{2})(2s+n)^2+n^2N\frac{1-\lambda }{2}})\\ 
&\quad \times \prod\limits_{i=0}^{N-1}\prod\limits_{j=0}^{n-1}
(x-(-1)^{n}(-1)^{nj}\nu ^n\omega ^{-n(2i+\frac{1-\nu }{2})(2s+n)^2+n^2N\frac{1-\lambda }{2}}) \\ 
&=
\begin{cases}
\prod\limits_{i=0}^{N-1}
(x^2-\omega ^{-2n(2s+n)^2(2i+\frac{1-\nu }{2})}
(-1)^{\frac{1-\lambda }{2}})^n\quad & (\textrm{$n$ is odd}), \vspace{0.1cm} \\ 
\prod\limits_{i=0}^{N-1}
(x-\omega ^{-n(2s+n)^2(2i+\frac{1-\nu }{2})}
(-1)^{\frac{n}{2}\frac{1-\lambda }{2}})^{2n} \quad & (\textrm{$n$ is even}), 
\end{cases} 
\end{align*}
and if $n=2$, then by Lemma~\ref{5.9}(2)(i) 
\begin{align*}
P_{A^{\ast}, \boldsymbol{k}g}(x)
&=\prod\limits_{(\alpha ,\beta )\in I_{\nu \lambda}}
(x-(\alpha \beta )^{-2s^2-4s-4}\alpha ^{4})
\cdot 
\prod\limits_{(\gamma ,\xi )\in J}
(x-\gamma ^{-4(s+1)^2}\lambda ) \\ 
&=\prod\limits_{i=0}^{N-1}
(x-\omega ^{-4(s+1)^2(2i+\frac{1-\nu }{2})}
(-1)^{\frac{1-\lambda }{2}})^{4} 
\cdot 
\prod\limits_{i=0}^{2N-1}
(x-\omega ^{-16i(s+1)^2}\lambda )^2 \\ 
&=\prod\limits_{i=0}^{N-1}
(x-\omega ^{-4(s+1)^2(2i+\frac{1-\nu }{2})}
(-1)^{\frac{1-\lambda }{2}})^{4} 
\cdot 
\prod\limits_{i=0}^{N-1}
(x-\omega ^{-16i(s+1)^2}\lambda )^4. 
\end{align*}
(2) If $n\geq 3$, then by Lemma~\ref{5.9}(1)(ii), 
\begin{align*}
P_{A^{\ast}, V_{st}}(x)
&=\prod\limits_{(\alpha ,\beta )\in I_{\nu \lambda}}
(x-\nu ^{t}\alpha ^{t^2}(\alpha \beta )^{-2s^2-2st-t^2}) \\ 
&=\prod\limits_{i=0}^{N-1}\prod\limits_{j=0}^{n-1}
(x-\nu ^{t}\omega ^{n(2i+\frac{1-\nu }{2})(-4s^2-4st-t^2)
+N(2j+\frac{1-\lambda }{2})t^2}) \\ 
&\quad \times \prod\limits_{i=0}^{N-1}\prod\limits_{j=0}^{n-1}
(x-(-1)^{t}\nu ^{t}\omega ^{n(2i+\frac{1-\nu }{2})(-4s^2-4st-t^2)
+N(2j+\frac{1-\lambda }{2})t^2}) \vspace{0.1cm}\\ 
&=\begin{cases}
\prod\limits_{i=0}^{N-1}\prod\limits_{j=0}^{n-1}
(x^2-\omega ^{-2n(2s+t)^2(2i+\frac{1-\nu }{2})+2Nt^2(2j+\frac{1-\lambda }{2})}) & \quad (\textrm{$t$ is odd}) \vspace{0.1cm}\\ 
\prod\limits_{i=0}^{N-1}\prod\limits_{j=0}^{n-1}
(x-\omega ^{-n(2s+t)^2(2i+\frac{1-\nu }{2})
+Nt^2(2j+\frac{1-\lambda }{2})})^2
 & \quad (\textrm{$t$ is even}) 
 \end{cases} \vspace{0.1cm}\\ 
&=\begin{cases}
\prod\limits_{i=0}^{N-1}\prod\limits_{j=0}^{n-1}
(x^2-\omega ^{-2n(2s+t)^2(2i+\frac{1-\nu }{2})+2Nt^2(2(n-j)+\frac{1-\lambda }{2})}) & \quad (\textrm{$t$ is odd}) \vspace{0.1cm}\\ 
\prod\limits_{i=0}^{N-1}\prod\limits_{j=0}^{n-1}
(x-\omega ^{-n(2s+t)^2(2i+\frac{1-\nu }{2})
+Nt^2(2(n-j)+\frac{1-\lambda }{2})})^2
 & \quad (\textrm{$t$ is even})  
 \end{cases} \displaybreak[0] \\ 
&=\begin{cases}
\prod\limits_{i=0}^{N-1}\prod\limits_{j=0}^{n-1}
(x^2-\omega ^{-2n(2s+t)^2(2i+\frac{1-\nu }{2})-2Nt^2(2j+\frac{\lambda -1}{2})}) & \quad (\textrm{$t$ is odd}) \vspace{0.1cm}\\ 
\prod\limits_{i=0}^{N-1}\prod\limits_{j=0}^{n-1}
(x-\omega ^{-n(2s+t)^2(2i+\frac{1-\nu }{2})
-Nt^2(2j+\frac{\lambda -1}{2})})^2
 & \quad (\textrm{$t$ is even}),  
 \end{cases}
\end{align*}
and if $n=2$, then by Lemma~\ref{5.9}(2)(ii), 
\begin{align*}
P_{A^{\ast}, V_{s1}}(x)
&=\prod\limits_{(\alpha ,\beta )\in I_{\nu \lambda}}
(x-\nu \alpha (\alpha \beta )^{-2s^2-2s-1}) 
\cdot 
\prod\limits_{(\gamma ,\xi )\in J}
(x-\gamma ^{-(2s+1)^2}) \\ 
&=\prod\limits_{i=0}^{N-1}\prod\limits_{j=0}^{1}
(x^2-\omega ^{-4(2i+\frac{1-\nu }{2})(2s+1)^2+2N(2j+\frac{1-\lambda }{2})})\cdot 
\prod\limits_{i=0}^{2N-1}(x-\omega ^{-4i(2s+1)^2})^2 \vspace{0.1cm}\\ 
&=\prod\limits_{i=0}^{N-1}
(x^4-\omega ^{-8(2i+\frac{1-\nu }{2})(2s+1)^2+2N(1-\lambda)})\cdot 
\prod\limits_{i=0}^{N-1}(x^2-\omega ^{-8i(2s+1)^2})^2. 
\end{align*}
\end{proof}

\par 
\medskip 
In the case where $N$ is odd and $\nu= +$, 
if $\lambda $ and $n$ satisfy the condition 
\begin{enumerate}
\item[(A)] $\lambda =-1$,\ or 
\item[(B)] $\lambda =1$, and $n$ is odd, 
\end{enumerate}
then the Hopf algebra 
$A_{Nn}^{+\lambda }$ is self-dual (see Corollary~\ref{6.7} in the next section). 
Therefore, by using Proposition~\ref{4.6} we have: 

\par 
\medskip 
\begin{cor}\label{5.11}
Let $N\geq 1$ be an odd integer and $\omega \in \boldsymbol{k}$ a primitive $4nN$-th root of unity. 
Suppose that $\lambda $ and $n$ satisfy the above condition (A) or (B). Then 
$$P_{A_{Nn}^{+\lambda}}^{(1)}(x)=
\begin{cases}
\prod\limits_{s=0}^{N-1}\prod\limits_{i=0}^{N-1}
(x-\omega ^{-8nis^2})^{4n}
(x^2-\omega ^{-4in(2s+1)^2}(-1)^{\frac{1-\lambda}{2}})^{2n} & \textrm{($n$ is odd),}\\[0.3cm]   
\prod\limits_{s=0}^{N-1}\prod\limits_{i=0}^{N-1}
(x-\omega ^{-8nis^2})^{4n}(x-\omega ^{-8ins^2}(-1)^{\frac{n}{2}})^{4n} & \textrm{($n\geq 4$ is even),} \\[0.3cm]   
\prod\limits_{s=0}^{N-1}\prod\limits_{i=0}^{N-1}
(x-\omega ^{-16is^2})^{16}(x+\omega ^{-8is^2})^8(x+\omega ^{-16is^2})^8 & (n=2), 
\end{cases} $$
$$P_{A_{Nn}^{+\lambda}}^{(2)}(x)=\begin{cases}
\prod\limits_{s=0}^{N-1}\prod\limits_{t=1}^{\frac{n-\epsilon (n)}{2}}
\prod\limits_{i=0}^{N-1}\prod\limits_{j=0}^{n-1}
(x^2-\omega ^{-4in(2s+1)^2-2N(2t-1)^2(2j+\frac{\lambda-1}{2})}) \\ 
 \qquad \times  
 \prod\limits_{s=0}^{N-1}\prod\limits_{t=1}^{\frac{n-2+\epsilon (n)}{2}}
\prod\limits_{i=0}^{N-1}\prod\limits_{j=0}^{n-1}
(x-\omega ^{-8ins^2-4Nt^2(2j+\frac{\lambda -1}{2})})^2 \quad (n\geq 3) , \\[0.3cm]   
\prod\limits_{s=0}^{N-1}\prod\limits_{i=0}^{N-1}
(x^4+\omega ^{-16i(2s+1)^2})(x^2-\omega ^{-8i(2s+1)^2})^2 \quad (n=2), 
\end{cases} $$
where 
$$\epsilon (n)=\begin{cases}
0 & \textrm{($n$ is even)}, \\ 
1 & \textrm{($n$ is odd)}. 
\end{cases}$$
 \end{cor}

\par 
\medskip 
From the above corollary the polynomial invariants of the Kac-Paljutkin algebra $K_8\cong A_{12}^{+-}$ are computed, again. 

\par 
\medskip 
\subsection{The group Hopf algebra $\boldsymbol{k}[G_{Nn}]$}
\par 
If $N$ is odd, $n\geq 2$, and $\lambda =\pm 1$, then 
$A_{Nn}^{+\lambda }$ is isomorphic to the group algebra of the finite group 
$$G=\langle h,\ t,\ w\ \vert \ t^2=h^{2N}=1,\ w^n=h^{(n+\frac{\lambda -1}{2})N},\ tw=w^{-1}t,\ ht=th,\ hw=wh \rangle $$
as an algebra (for the proof see Proposition~\ref{5.24} in the next section). 
The order of $G$ is $4nN$. 
If $(n, \lambda )=(\textrm{even}, 1)$ or $(n, \lambda )=(\textrm{odd}, -1)$, then 
the group $G$ is isomorphic to the direct product $D_{2n}\times C_{2N}$. 
If $(n, \lambda )=(\textrm{even}, -1)$ or $(n, \lambda )=(\textrm{odd}, 1)$, then 
$G$ is isomorphic to the semidirect product of 
$$H:=\langle \ w,\ h\ | \ h^{2N}=1,\ w^n=h^N,\ wh=hw\ \rangle$$ 
and $C_2$, where the action of $C_2=\langle t\rangle $ on $H$ is given by $t\cdot w:=w^{-1}$ and $t\cdot h:=h$. 
In particular, when $N=1$, the group $G$ is the dihedral group of order $4n$.  
\par 
We shall determine the universal $R$-matrices of the group Hopf algebra $\boldsymbol{k}[G]$, and calculate the polynomial invariants of it in the case where $(n, \lambda )=(\textrm{even}, -1)$ or $(n, \lambda )=(\textrm{odd}, 1)$. In this case, $G$ coincides with the group 
$$G_{Nn}=\langle h,\ t,\ w\ \vert \ t^2=h^{2N}=1,\ w^n=h^N,\ tw=w^{-1}t,\ ht=th,\ hw=wh \rangle .$$

\par \smallskip 
\begin{lem}\label{5.15}
Let $N$ be an odd integer, and $n\geq 2$ an integer. 
\par 
(1) If $n\geq 3$, then the finite group $G_{Nn}$ has a unique maximal commutative normal subgroup, which is given by $H=\langle h, w\rangle $. 
\par 
(2) If $n=2$, then the finite group $G_{Nn}$ has exactly three maximal commutative normal subgroups, which are given by $H_1=\langle h, w\rangle ,\  
H_2=\langle h, t\rangle $ and $H_3=\langle h, tw\rangle $. 
\end{lem}
\begin{proof}
We set $G=G_{Nn}$. 
It is clear that $H\ (=H_1)$ is a commutative and normal subgroup of $G$, and it is not hard  to show that $H$ is maximal between commutative subgroups. 
Hence $H$ is a maximal commutative normal subgroup of $G$. 
In the case of $n=2$ we see that $H_2$ and $H_3$ are also commutative and normal subgroups of $G$, and  maximal between commutative subgroups. 
\par 
We show that the converse is true. 
\par 
(1) Let $K$ be a maximal commutative normal subgroup $K$ of $G=G_{Nn}$. 
Suppose that $K\not\subset H$. 
Then 
$K\cap (G-H)\not= \varnothing $. 
So, if we take an element 
$tw^ih^j\in K\cap (G-H) \ (0\leq i<n,\ 0\leq j<N)$, then 
$w(tw^ih^j)w^{-1}=tw^{i-2}h^j\in K$, 
that is, $tw^ih^j\in K$ implies $h^{-j}w^{-i+2}t=(tw^{i-2}h^j)^{-1}\in K$. 
Since $K$ is commutative, we have the equation
$$w^2=h^{-j}w^{-i+2}t\cdot tw^ih^j
=tw^ih^j\cdot h^{-j}w^{-i+2}t
=w^{-2}=w^{2n-2}=w^{(n-2)+n}=w^{n-2}h^N. $$
This is a contradiction since $n\geq 3$. 
Thus $K\subset H$. 
This implies that $K=H$ from maximality of $K$. 
\par  
(2) Let $K$ be a maximal commutative normal subgroup $K$ of $G=G_{Nn}$. 
Then $\langle h\rangle \subset K$ holds from maximality of $K$ since $K\langle h \rangle $ is a commutative and normal subgroup of $G$. 
If $K\subset H_1$, then $K=H_1$ since $H_1$ is a maximal commutative normal subgroup of $G$. 
So, we suppose that $K\not\subset H_1$. 
Then $tw^ih^j\in K$ for some $i,j\ (i=0,1,\ j=0,1,\ldots ,N-1)$. 
Since $\langle h\rangle \subset K$, we have $tw^i\in K$. 
\par 
If $i=0$, then $t\in K$, and hence $H_2\subset K$. 
By maximality of $H_2$, we see that $K=H_2$. 
By the same argument, we see that if $i=1$, then $K=H_3$. 
Thus there is no maximal commutative normal subgroup of $G$ except for $H_1, H_2, H_3$. 
\end{proof}

\par \smallskip 
In what follows, we assume that the characteristic $\rm{ch}(\boldsymbol{k})$ does not divide $2nN$. 
To determine the universal $R$-matrices of $\boldsymbol{k}[H]$, we use the basis consisting of the primitive idempotents of $\boldsymbol{k}[H]$. 

\par \medskip 
\begin{lem} \label{5.16}  
Let $N\geq 1$ be an odd integer, and $n\geq 2$ an integer. 
Let $H$ be the commutative group of order $2nN$ defined by 
$H=\langle h,\ w\ \vert \ h^{2N}=1,\ w^n=h^N,\ hw=wh \rangle $, and 
$\omega $ a primitive $4nN$-th root of unity in $\boldsymbol{k}$. 
For $i, k\in \mathbb{Z}$, we set 
$$E_{ik}=\frac{1}{2nN}\sum\limits_{j=0}^{n-1}\sum\limits_{l=0}^{2N-1}\omega ^{-2Nj(k+2i)-2nkl}w^jh^l \quad \in \ \ \boldsymbol{k}[H].$$
Then 
$\{ \ E_{ik}\ \vert \ i=0,1,\ldots ,n-1,\ k=0,1,\ldots ,2N-1\ \} $ is the set of primitive idempotents of $\boldsymbol{k}[H]$, and the following equations hold: for all $i,j, k,l\in \mathbb{Z}$
\begin{align}
&E_{i+n, k}=E_{ik},\qquad E_{i-N, k+2N}=E_{ik}, \label{eq5.4}\\ 
&E_{ik}E_{jl}=\delta _{kl}^{(2N)}\delta _{k+2i, l+2j}^{(2n)}E_{jl}=\delta _{kl}^{(2N)}\delta _{k+2i, l+2j}^{(2n)}E_{ik},  \label{eq5.5}
\end{align}
where 
$$\delta ^{(m)}_{kl}=\begin{cases}
1 & (k\equiv l\ (\textrm{mod}\kern0.2em m))\\ 
0 & (k\not\equiv l\ (\textrm{mod}\kern0.2em m)) 
\end{cases}$$
for $m=2N$ or $m=2n$. 
Furthermore, the coproduct $\Delta$, the counit $\varepsilon $, and the antipode $S$ of the group Hopf algebra $\boldsymbol{k}[H]$ are given as follows: 
\begin{align}
\Delta (E_{ik})&=\sum\limits_{\substack{0\leq p,q\leq 2N-1\\ p+q\equiv k\ (\textrm{mod}\ 2N)}}
\sum\limits_{\substack{0\leq a,b\leq n-1\\ a+b+\frac{-k+p+q}{2}\equiv i\ (\textrm{mod}\ n)}}E_{ap}\otimes E_{bq} , \label{eq5.6}\\ 
\varepsilon (E_{ik})&=\delta _{i,0}\delta _{k,0}, \label{eq5.7}\\ 
S(E_{ik})&=E_{-i,-k}. \label{eq5.8}
\end{align}
\end{lem}
\begin{proof}
For integers $i$ and $k$, 
let $\chi_{jk}$ be the group homomorphism from $H$ to $\boldsymbol{k}$ defined by $\chi_{ik}(w)=\omega ^{2N(k+2i)}, \ \chi_{ik}(h)=\omega ^{2nk}$. 
Then the set 
$\{ \ \chi _{ik}\ \vert \ i=0,1,\ldots ,n-1,\ k=0,1,\ldots ,2N-1\ \} $ 
consists of all irreducible characters of $H$.  
So, 
$$\frac{1}{2nN}\sum\limits_{j=0}^{n-1}\sum\limits_{l=0}^{2N-1}\chi _{ik}(w^{-j}h^{-l}) w^jh^l
=\frac{1}{2nN}\sum\limits_{j=0}^{n-1}\sum\limits_{l=0}^{2N-1}\omega ^{-2Nj(k+2i)-2nkl}\kern0.1em w^jh^l=E_{jk}$$
is a primitive idempotent of $\boldsymbol{k}[H]$, 
and the set $\{ E_{ik}\} $ consists of all primitive idempotents of $\boldsymbol{k}[H]$. 
By definition the equations (\ref{eq5.4}) and (\ref{eq5.7}) are obtained immediately, and  the equations (\ref{eq5.5}), (\ref{eq5.6}) and (\ref{eq5.8}) follow from that $w^jh^l$ is written as 
\begin{equation}\label{eq5.9}
w^jh^l=\sum\limits_{i=0}^{n-1}\sum\limits_{k=0}^{2N-1}\omega ^{2Nj(2i+k)+2nkl} E_{ik}, 
\end{equation}
which comes from orthogonality of characters. 
\end{proof} 

\par \bigskip 
\begin{lem} \label{5.17}  
Let $N\geq 1$ be an odd integer, and $n\geq 2$ an integer. 
Let $H$ be the commutative group of order $2nN$ defined in Lemma~\ref{5.16}, and $\omega \in \boldsymbol{k}$ a primitive $4nN$-th root of unity. Then any universal $R$-matrix $R$ of $\boldsymbol{k}[H]$ is 
given by 
\begin{equation}\label{eq5.10}
R=\sum\limits_{i,j=0}^{n-1}\sum\limits_{k,l=0}^{2N-1}\nu ^{kl}\omega^{2aN(2i+k)(2j+l)+2n(qkl+2pjk+2rij)} \kern0.2em E_{ik}\otimes E_{jl}
\end{equation}
for some $\nu \in \{ \pm 1\}$, $a\in \{ 0,1,\ldots ,n-1\} $ and $p,q,r\in \{ 0,1,\ldots ,N-1\} $ such that $pn, rn$ are multiples of $N$, where $\{ E_{ik}\} $ is the set of primitive idempotents of $\boldsymbol{k}[H]$ defined in Lemma~\ref{5.16}. 
Conversely, the above $R$ given by (\ref{eq5.10}) is a universal $R$-matrix of $\boldsymbol{k}[H]$.  
\end{lem}
\begin{proof}
Let $R$ be an element of $\boldsymbol{k}[H]\otimes \boldsymbol{k}[H]$, and 
write $R$ in the form 
$$R=\sum\limits_{i,j=0}^{n-1}\sum\limits_{k,l=0}^{2N-1}R^{ik}_{jl}\kern0.2em E_{ik}\otimes E_{jl} \qquad (R^{ik}_{jl}\in \boldsymbol{k}).$$
We treat the indices $i$ and $j$ of $R^{ik}_{jl}$ as modulo $n$, therefore $R^{ik}_{jl}$ has a meaning for all integers $i$ and $j$. 
For $0\leq m\leq 4N-1$, we define $\delta (m)$ by 
$$\delta (m)=\begin{cases}
0 & (0\leq m\leq 2N-1),\\[0.1cm]  
1 & (2N\leq m\leq 4N-1).\end{cases}$$
Then we have 
\begin{align*}
(\Delta \otimes \textrm{id})(R)
&=\sum\limits_{j,a,b=0}^{n-1}\sum\limits_{l,p,q=0}^{2N-1}
R^{a+b+N\delta (p+q),\ p+q-2N\delta (p+q)}_{jl}\kern0.1em E_{ap}\otimes E_{bq} \otimes E_{jl},
 \\ 
R_{13}R_{23}
&=\sum\limits_{j,a,b=0}^{n-1}\sum\limits_{l,p,q=0}^{2N-1}
R^{ap}_{jl}R^{bq}_{jl}
E_{ap}\otimes E_{bq} \otimes E_{jl}. 
\end{align*}
Hence $(\Delta \otimes \textrm{id})(R)=R_{13}R_{23}$ if and only if 
\begin{equation}\label{eq5.11}
R^{a+b+N\delta (p+q),\ p+q-2N\delta (p+q)}_{jl}=R^{ap}_{jl}R^{bq}_{jl}
\end{equation}
for all $j,a,b=0,1,\ldots ,n-1$ and $ l,p,q=0,1,\ldots ,2N-1$. 
Similarly, $(\textrm{id}\otimes \Delta )(R)=R_{13}R_{12}$ if and only if 
\begin{equation}\label{eq5.12}
R^{ik}_{a+b+N\delta (p+q),\ p+q-2N\delta (p+q)}=R^{ik}_{bq}R^{ik}_{ap}
\end{equation}
for $i,a,b=0,1,\ldots ,n-1$ and $k,p,q=0,1,\ldots ,2N-1$, and also we have 
$(\varepsilon \otimes \textrm{id})(R)=(\textrm{id}\otimes \varepsilon )(R)=1$ if and only if 
\begin{equation}\label{eq5.13}
R^{00}_{jl}=R^{jl}_{00}=1
\end{equation}
for $j=0,1,\ldots ,n-1$ and $l=0,1,\ldots ,2N-1$. 
Thus $R$ is a universal $R$-matrix of $\boldsymbol{k}[H]$ if and only if the equations (\ref{5.11}), (\ref{eq5.12}), (\ref{eq5.13}) hold. 
From the equations (\ref{eq5.11}) and  (\ref{eq5.12}), 
$R^{ik}_{jl}$ can be expressed as the form 
$$R^{ik}_{jl}=(R_{01}^{01})^{kl}(R_{10}^{01})^{jk}(R_{01}^{10})^{il}(R_{10}^{10})^{ij}.$$
Since $(R_{01}^{10})^{n}=R_{01}^{n0}=R_{01}^{00}=1$ and $(R_{10}^{01})^{n}=R_{n0}^{01}=R_{00}^{01}=1$ by (\ref{eq5.11}), we see that 
$R^{10}_{01}$ and $R^{01}_{10}$ are can be written as 
\begin{equation}\label{eq5.14}
R^{10}_{01}=\omega ^{4aN},\quad  R^{01}_{10}=\omega ^{4bN}\qquad (0\leq a,b\leq n-1). 
\end{equation}
By using the equations (\ref{eq5.11}) and (\ref{eq5.12}), repeatedly, we have
$$\begin{cases}
(R_{01}^{01})^{2N}=(R_{01}^{01})^{2N-2}R_{01}^{02}=
\cdots =R_{01}^{01}R_{01}^{0,2N-1}
=R_{01}^{N0}=(R_{01}^{10})^N=\omega^{4aN^2},\\[0.1cm]  
(R_{01}^{01})^{2N}=(R_{01}^{01})^{2N-2}R_{02}^{01}=
\cdots =R_{01}^{01}R_{0,2N-1}^{01}
=R_{N0}^{01}=(R_{10}^{01})^N=\omega^{4bN^2}. 
\end{cases}$$
Therefore $\omega^{4bN}=\omega^{4aN+4pn}$ must be required for some $0\leq p\leq N-1$ such that $pn$ is a multiple of $N$. 
From the equation $(R_{01}^{01})^{2N}=\omega ^{4aN^2}$, we may set 
$(R_{01}^{01})^{2}=\omega^{4aN+4qn}\ (0\leq q\leq N-1)$, and hence 
\begin{equation}\label{eq5.16}
R_{01}^{01}=\pm \omega^{2aN+2qn}. 
\end{equation}

By using the equation (\ref{eq5.11}), repeatedly, we have 
$(R_{10}^{10})^N=R_{N0}^{10}=(R_{01}^{10})^{2N}=\omega ^{8aN^2}$. 
So, $R_{10}^{10}$ can be expressed as 
\begin{equation}\label{eq5.17}
R_{10}^{10}=\omega^{8aN+4rn}
\end{equation}
for some $0\leq r\leq N-1$ such that $rn$ is a multiple of $N$. 
From (\ref{eq5.14}), (\ref{eq5.16}) and (\ref{eq5.17}) 
we see that a universal $R$-matrix $R$ of $\boldsymbol{k}[H]$ is written in the form 
$$R=\sum\limits_{i,j=0}^{n-1}\sum\limits_{k,l=0}^{2N-1}R^{ik}_{jl}\kern0.2em E_{ik}\otimes E_{jl},\quad R^{ik}_{jl}
=\nu ^{kl}\omega ^{2aN(2i+k)(2j+l)+2n(qkl+2pjk+2rij)} $$
for some $\nu \in \{ \pm 1\} $, $a\in \{ 0,1,\ldots ,n-1\} $ and  $p,q,r\in \{ 0,1,\ldots ,N-1 \} $ that $pn, rn$ are multiples of $N$. 
\par 
Conversely, it is not hard to check that $R\in \boldsymbol{k}[H]\otimes \boldsymbol{k}[H]$ which is given by the  form above is a universal $R$-matrix of $\boldsymbol{k}[H]$.
\end{proof}

\par \medskip 
By use of Lemma~\ref{5.14} and Lemma~\ref{5.17} one can determine the universal $R$-matrices of $\boldsymbol{k}[G_{Nn}]$. 

\par \smallskip 
\begin{prop} \label{5.18}  
Let $N\geq 1$ be an odd integer, and $n\geq 2$ an integer. 
Let $\omega $ be a primitive $4nN$-th root of unity in $\boldsymbol{k}$, and $\{ E_{ik}\} $ be the set of primitive idempotent of $\boldsymbol{k}[H]$ defined in Lemma~\ref{5.16}. 
Then for any $\nu \in \{ \pm 1\},\ a\in \{ 0,1,\ldots ,n-1\},\ q\in \{ 0,1,\ldots , N-1\} $, 
\begin{equation}\label{eq5.18}
R_{aq\nu }:=\sum\limits_{i,j=0}^{n-1}\sum\limits_{k,l=0}^{2N-1}
\nu ^{kl}\omega ^{2aN(2i+k)(2j+l)+2qkln} 
\kern0.2em E_{ik}\otimes E_{jl}
\end{equation}
is a universal $R$-matrix of the group Hopf algebra $\boldsymbol{k}[G_{Nn}]$. Furthermore,  
\par 
$\bullet$ $R_{aq\nu }$ is a universal $R$-matrix of the group Hopf algebra $\boldsymbol{k}[\langle h\rangle ]$ if and only if $a=0$, and  
\par 
$\bullet$ if $n\geq 3$, then any universal $R$-matrix is given by the above form, therefore, the number of the universal $R$-matrices of $\boldsymbol{k}[G_{Nn}]$ is  $2nN$. 
\end{prop}
\begin{proof}
Let $R$ be a universal $R$-matrix of $\boldsymbol{k}[H]$, where $H$ is the subgroup of $G_{Nn}$ defined in Lemma~\ref{5.17}, 
and write it in the form 
$$R=\sum\limits_{i,j=0}^{n-1}\sum\limits_{k,l=0}^{2N-1}R^{ik}_{jl}\kern0.2em E_{ik}\otimes E_{jl}.$$
By using $E_{ik}t=tE_{-i-k, k}$ we have 
\begin{align*}
\Delta ^{\textrm{cop}}(t)\cdot R
&=\sum\limits_{i,j=0}^{n-1}\sum\limits_{k,l=0}^{2N-1}R^{ik}_{jl}\kern0.2em tE_{ik}\otimes tE_{jl},\\ 
R\cdot \Delta (t)
&=\sum\limits_{i,j=0}^{n-1}\sum\limits_{k,l=0}^{2N-1}R^{-i-k, k}_{-j-l, l}\kern0.2em tE_{ik}\otimes tE_{jl}. 
\end{align*}
Hence $R$ is a universal $R$-matrix of $\boldsymbol{k}[G_{Nn}]$ if and only if $R^{-i-k, k}_{-j-l, l}=R^{ik}_{jl}$ for all $i,j,k,l$. 
Since $R$ is a universal $R$-matrix of $\boldsymbol{k}[H]$, $R^{ik}_{jl}$ is written in the form
$$R^{ik}_{jl}=\nu ^{kl}\omega ^{2aN(2i+k)(2j+l)+2n(qkl+2pjk+2rij)}$$
for some $a\in \{ 0,1,\ldots ,n-1\}$ and $p, q, r \in \{ 0,1,\ldots , N-1\}$. Therefore, we  have 
\begin{align*}
R^{-i-k, k}_{-j-l, l}=R^{ik}_{jl}\ \ 
& \Longleftrightarrow \ \ 
\omega ^{2n(-2p(2j+l)k+2r(kj+il+kl))}=1. 
\end{align*}

Considering the equations in R.H.S. for $(i,j,k,l)=(1,0,0,1)$ and $(i,j,k,l)=(0,0,1,1)$, we see that  
$R$ is a universal $R$-matrix of $\boldsymbol{k}[G_{Nn}]$ if and only if $\omega ^{4pn}=\omega ^{4rn}=1$. 
This condition is equivalent to that the both $p$ and $r$ are multiples of $N$. 
It follows from $0\leq p, r\leq N-1$ that $p=r=0$. 
\end{proof}

\par \smallskip 
\begin{prop} \label{5.20}  
Let $N\geq 1$ be an odd integer, and let $\omega $ be a primitive $8N$-th root of unity in a field $\boldsymbol{k}$ whose characteristic dose not divide $2N$. 
Then the number of universal $R$-matrices of the group Hopf algebra $\boldsymbol{k}[G_{N2}]$ is $8N$, and they are given by the list below.  
\par 
$\bullet$ universal $R$-matrices of $\boldsymbol{k}[\langle h\rangle ]$: 
$$
R_d=\sum\limits_{k,l=0}^{2N-1}\omega ^{4dkl} E_k\otimes E_l\quad (d=0,1,\ldots , 2N-1),$$
where $E_k=\frac{1}{2N}\sum_{l=0}^{2N-1}\omega ^{-4kl}h^l$. 
\par 
$\bullet$ universal $R$-matrices of $\boldsymbol{k}[H_1]$, where $H_1=\langle h, w\rangle $: 
$$
R_{1q\nu }=\sum\limits_{i,j=0,1}\sum\limits_{k,l=0}^{2N-1}
\nu ^{kl}\omega ^{2N(2i+k)(2j+l)+4qkl} \kern0.2em E_{ik}\otimes E_{jl} \quad (q=0,1,\ldots , N-1,\ \nu =\pm 1),$$
where $E_{ik}=\frac{1}{4N}\sum_{j=0,1}\sum_{l=0}^{2N-1}(-1)^{ij}\omega^{-2Njk-4kl}w^jh^l$.
\par 
$\bullet$ universal $R$-matrices of $\boldsymbol{k}[H_2]$, where $H_2=\langle h, t\rangle $: 
\begin{equation}\label{eq5.19}
R_d^{(1)}:=R_{0N1d}^{\boldsymbol{k}[H_2]}=\sum_{i,j=0,1}\sum_{k,l=0}^{2N-1}(-1)^{jk+il}\omega ^{4dkl} E_{ik}\otimes E_{jl}\quad (d=0,1,\ldots , 2N-1),
\end{equation}
where $E_{ik}=\frac{1}{4N}\sum_{j=0,1}\sum_{l=0}^{2N-1}(-1)^{ij}\omega ^{-4kl}t^jh^l$. 
\par 
$\bullet$ universal $R$-matrices of $\boldsymbol{k}[H_3]$, where $H_3=\langle h, tw\rangle $: 
\begin{equation}\label{eq5.20}
R_d^{(2)}:=R_{0N1d}^{\boldsymbol{k}[H_3]}=\sum_{i,j=0,1}\sum_{k,l=0}^{2N-1}(-1)^{jk+il}\omega ^{4dkl} E_{ik}\otimes E_{jl} \quad (d=0,1,\ldots , 2N-1),
\end{equation}
where $E_{ik}=\frac{1}{4N}\sum_{j=0,1}\sum_{l=0}^{2N-1}(-1)^{ij}\omega ^{-4kl}(tw)^jh^l$. 
\end{prop}
\begin{proof}
By Proposition~\ref{5.18} it is sufficient to determine the universal $R$-matrices of $\boldsymbol{k}[G_{N2}]$ which come from that of $\boldsymbol{k}[H_2]$ or $\boldsymbol{k}[H_3]$.   
By Lemma~\ref{5.19} it follows from $H_i\cong C_2\times C_{2N}$ for $i=2,3$ that a universal $R$-matrix of $\boldsymbol{k}[H_i]$ is given by 
$$R_{pqrs}^{\boldsymbol{k}[H_i]}=\sum_{i,j=0}^{m-1}\sum_{k,l=0}^{n-1}(-1)^{(pi+rk)j}\omega ^{2(skl+qil)}E_{ik}\otimes E_{jl}, $$
where $p,r\in \{ 0,1\},\ q\in \{ 0,N\} ,\ s\in \{ 0,1,\ldots ,2N-1\} $. 
\par 
Let us consider the case of $i=2$ and $R=R_{pqrs}^{\boldsymbol{k}[H_2]}$. 
We set $R^{ik}_{jl}=(-1)^{j(pi+rk)}\omega^{2l(sk+qi)}$. 
Then by using $E_{ik}w=wE_{i+k, k}$, we have 
\begin{align*}
\Delta ^{\textrm{cop}}(w)\cdot R=R\cdot \Delta (w) 
\Longleftrightarrow \ \ &  R^{i-k, k}_{j-l, l}=R^{ik}_{jl} \ \ \textrm{for all $i,j,k,l$} \\ 
\Longleftrightarrow \ \ &
 \ \ (-1)^{-(j-l)pk-l(pi+rk)}\omega ^{-2lqk}=1\ \ \textrm{for all $i,j,k,l$} \\ 
\Longleftrightarrow \ \ & (p,q,r)=(0,0,0)\ \ \textrm{or} \ \ (p,q,r)=(0, N,1). 
\end{align*}
Thus $R_{pqrs}^{\boldsymbol{k}[H_2]}$ is a universal $R$-matrix of $\boldsymbol{k}[G_{N2}]$ if and only if $(p,q,r)=(0,0,0)$ or $(p,q,r)=(0, N,1)$. 
It is easy proved that $R_{000s}^{\boldsymbol{k}[H_2]}$ is a universal $R$-matrix of $\boldsymbol{k}[\langle h\rangle ]$, and $R_{0N1s}^{\boldsymbol{k}[H_2]}$ is not. 
By a similar manner, it is  shown that a universal $R$-matrix 
$R_{pqrs}^{\boldsymbol{k}[H_3]}$ of $\boldsymbol{k}[H_3]$ is a universal $R$-matrix of $\boldsymbol{k}[G]$ if and only if $(p,q,r)=(0,0,0), (0,N,1)$, 
and $R_{pqrs}^{\boldsymbol{k}[H_3]}$ is that of $\boldsymbol{k}[\langle h\rangle ]$ if and only if $(p,q,r)=(0,0,0)$. 
\end{proof}

\par \smallskip 
\begin{lem}
Let $N\geq 1$ be an odd integer, and $n\geq 2$ an integer. 
\par 
(1) The Drinfel'd element $u_{aq\nu }$ associated to $R_{aq\nu }$ in Proposition~\ref{5.18} is given by 
$$u_{aq\nu }=\sum\limits_{i=0}^{n-1}\sum\limits_{k=0}^{2N-1}
\nu ^k\omega ^{-2aN(2i+k)^2-2nqk^2} \kern0.2em E_{ik},$$
where $\{ E_{ik}\} $ be the set of primitive idempotents of $\boldsymbol{k}[H]$ defined in Lemma~\ref{5.16}. 
\par 
(2) In the case of $n=2$, for each $i=1,2$ 
the Drinfel'd element $u_d^{(i)}$ associated to $R_d^{(i)}$ given by (\ref{eq5.19}) and (\ref{eq5.20})  is given by $$u_d^{(i)}=\sum_{k=0}^{2N-1}\omega ^{-4dk^2} E_k,$$
where 
$E_k=\frac{1}{2N}\sum_{l=0}^{2N-1}\omega ^{-4kl}h^l$. 
\end{lem}
\begin{proof}
(1) By definition, 
\begin{align*}
u_{aq\nu }
&=\sum\limits_{i,j=0}^{n-1}\sum\limits_{k,l=0}^{2N-1}
\nu ^{kl}\omega^{2aN(2i+k)(2j+l)+2qkln} \delta _{k,-l}^{(2N)}\delta _{2i+k, -2j-l}^{(2n)} \kern0.2em E_{ik}. 
\end{align*}

For a fixed integer $k\in \{ 0,1,\ldots ,2N-1\} $ 
there is a unique integer $l\in \{0,1,\ldots ,2N-1\} $ satisfying 
$-l\equiv k\ (\textrm{mod}\kern0.1em 2N)$, and that is given by 
$$l=\begin{cases} 
2N-k & (k=1,\ldots , 2N-1), \\ 
0 & (k=0), \end{cases}$$
and under this condition, for a fixed  integer $i\in \{0,1,\ldots ,n-1\}$ there is a unique integer $j\in \{ 0,1,\ldots ,n-1\}$ satisfying $-2j-l\equiv 2i+k\ (\textrm{mod}\kern0.1em 2n)$, and that is given by 
$$j\equiv \begin{cases} 
-N-i & (k=1,\ldots , 2N-1), \\ 
-i & (k=0) 
\end{cases} \quad (\textrm{mod}\ n).$$
Thus we have 
\begin{align*}
u_{aq\nu }
&=\sum\limits_{i=0}^{n-1}\sum\limits_{k=1}^{2N-1}
\nu ^{k}\omega ^{-2aN(2i+k)^2-2qk^2n} \kern0.2em E_{ik} 
+ \sum\limits_{i=0}^{n-1}\omega ^{-8ai^2N} \kern0.2em E_{i0}  \\ 
&=\sum\limits_{i=0}^{n-1}\sum\limits_{k=0}^{2N-1}
\nu ^{k}\omega^{-2aN(2i+k)^2-2qk^2n} \kern0.2em E_{ik}. 
\end{align*}
(2) By definition, 
$$
u_d^{(1)}=\sum_{i,j=0,1}\sum_{k,l=0}^{2N-1}(-1)^{jk+il}\omega ^{4dkl} S(E_{jl})E_{ik}=\sum_{i=0,1}\sum_{k=0}^{2N-1}\omega ^{-4dk^2} E_{ik}=\sum_{k=0}^{2N-1}\omega ^{-4dk^2} E_k.$$
By a similar computation, we have the formula for $u_d^{(2)}$. 
\end{proof}

\par \smallskip 
Let $\omega \in \boldsymbol{k}$ be a primitive $4nN$-th root of unity. 
Then a full set of non-isomorphic simple left $\boldsymbol{k}[G_{Nn}]$-modules is given by 
\begin{align*}
& \{ \ V_{ijk}\ \vert \ i,j=0,1,\ k=0,2,\ldots ,2N-2\} \\ 
\cup &
\{ \ V_{jk}\ \vert \ k=0,1,\ldots ,2N-1,\ j=1,2,\ldots ,n-1,\ j\equiv k \ (\textrm{mod}\ 2) \} , 
\end{align*}
\noindent 
where 
the action $\chi _{ijk}$ of $\boldsymbol{k}[G_{Nn}]$ on $V_{ijk}=\boldsymbol{k}$ is given by 
\begin{equation}\label{eq5.21}
\chi _{ijk}(t)=(-1)^i,\ \chi _{ijk}(w)=(-1)^j,\ \chi _{ijk}(h)=\begin{cases} \omega ^{2kn} & (\textrm{$n$ is even}),\\ \omega ^{2(j+k)n} & (\textrm{$n$ is odd}),  \end{cases}
\end{equation}
and the left action $\rho _{jk}$ of $\boldsymbol{k}[G_{Nn}]$ on $V_{jk}=\boldsymbol{k}\oplus \boldsymbol{k}$ is given by 
\begin{equation}\label{eq5.22}
\rho _{jk}(t)=\begin{pmatrix} 0 & 1 \\ 1 & 0 \end{pmatrix},\ 
\rho _{jk}(w)=\begin{pmatrix} \omega ^{2jN} & 0 \\ 0 & \omega^{-2jN} \end{pmatrix},\ 
\rho _{jk}(h)=\begin{pmatrix} \omega ^{2kn} & 0 \\ 0 & \omega ^{2kn} \end{pmatrix}.
\end{equation}

\par \medskip 
\begin{lem}\label{5.21}
For each universal R-matrix $R$ of $\boldsymbol{k}[G_{Nn}]$, 
the $R$-dimensions of the simple left $\boldsymbol{k}[G_{Nn}]$-modules $V_{ijk}$ and $V_{jk}$ are given as follows. 
\par 
(1) For $n\geq 2$, 
\begin{align}
\underline{\dim}_{R_{aq\nu }}\kern0.1em V_{ijk}
&=\begin{cases} \omega ^{-2nqk^2} & \quad (\textrm{$n$ is even}),\\ 
\nu ^{j}(-1)^{aj}\omega ^{-2nq(j+k)^2} & \quad (\textrm{$n$ is odd}), 
\end{cases} \label{eq5.23} \\ 
\underline{\dim}_{R_{aq\nu }}\kern0.1em V_{jk}&=2\nu ^{k}\omega ^{-2nqk^2-2Naj^2}. \label{eq5.24}
\end{align}
(2) For $n=2$, 
\begin{align*}
\underline{\dim}_{R_d^{(1)}}\kern0.1em V_{ijk}&=\underline{\dim}_{R_d^{(2)}}\kern0.1em V_{ijk}=\omega ^{-4dk^2}, \\ 
\underline{\dim}_{R_d^{(1)}}\kern0.1em V_{1k}&=\underline{\dim}_{R_d^{(2)}}\kern0.1em V_{1k}=2\omega ^{-4dk^2}.
\end{align*}
\end{lem}
\begin{proof}
(1) Let $n$ be an even integer. 
Since 
\begin{align*}
\chi_{ijk}(E_{ml})
&=\frac{1}{2nN}\sum\limits_{b=0}^{n-1}\sum\limits_{p=0}^{2N-1}\omega ^{-2bN(l+2m)-2lpn}(-1)^{jb}(\omega )^{2kpn} \\ 
&=\delta _{k,l}
\frac{1}{n}\sum\limits_{b=0}^{n-1}\omega ^{4bN(\frac{nj-k}{2}-m)} \\ 
&=\delta _{k,l}\delta _{\frac{nj-k}{2},m}^{(n)}, 
\end{align*}
we have 
\begin{align*}
\chi_{ijk}(u_{aq\nu})
=\sum\limits_{m=0}^{n-1}
\nu ^{k}\omega ^{-2aN(2m+k)^2-2qk^2n} \delta _{\frac{nj-k}{2},m}^{(n)} 
=\sum\limits_{m=0}^{n-1}
\omega^{-2aN(2m+k)^2-2qk^2n} \delta _{\frac{nj-k}{2},m}^{(n)}. 
\end{align*}

Since there is a unique integer in 
$\{ 0,1,\ldots ,n-1\} $ which is equal to $\frac{nj-k}{2}$ as modulo $n$, 
\begin{align*}
\chi_{ijk}(u_{aq\nu})
&=\omega^{-2aN((nj-k)+k)^2-2qk^2n}  =\omega ^{-2qk^2n}. 
\end{align*}

This shows that 
$\underline{\dim}_{R_{aq\nu }}\kern0.1em V_{ijk}=\omega^{-2qk^2n}$. 
\par 
Let $n$ be an odd integer. In this case, 
$\chi_{ijk}(E_{ml})=\delta _{j+k,l}\delta _{\frac{nj-j-k}{2},m}^{(n)}$ since $n-1,\ k$ are even, and thus 
we have 
\begin{align*}
\chi_{ijk}(u_{aq\nu})
&=\sum\limits_{m=0}^{n-1}
\nu ^{j}\omega^{-2aN(2m+j+k)^2-2q(j+k)^2n} \delta _{\frac{nj-j-k}{2},m}^{(n)} \\ 
&=\nu ^j\omega^{-2aN((nj-j-k)+j+k)^2-2q(j+k)^2n}  
=\nu ^j(-1)^{aj}\omega ^{-2q(j+k)^2n}. 
\end{align*}
This shows that 
$\underline{\dim}_{R_{aq\nu }}\kern0.1em V_{ijk}=\nu ^j(-1)^{aj}\omega^{-2q(j+k)^2n}$. 
\par 
Next we calculate $\underline{\dim}_{R_{aq\nu }}\kern0.1em V_{jk}$. 
Since 
\begin{align*}
\rho _{jk}(E_{ml})
&=\frac{1}{2nN}\sum\limits_{b=0}^{n-1}\sum\limits_{p=0}^{2N-1}\omega ^{-2bN(l+2m)-2lpn}
\begin{pmatrix} 
\omega ^{2jN} & 0 \\ 
0 & \omega^{-2jN} \end{pmatrix}^b
\begin{pmatrix} 
\omega ^{2kn} & 0 \\ 
0 & \omega ^{2kn} \end{pmatrix}^p \\ 
&=\delta _{k,l}\begin{pmatrix} 
\dfrac{1}{n}\sum\limits_{b=0}^{n-1}\omega ^{4bN(\frac{j-k}{2}-m)} & 0 \\[0.1cm]  
0 & \dfrac{1}{n}\sum\limits_{b=0}^{n-1}\omega ^{4bN(\frac{-j-k}{2}-m)} \end{pmatrix} \\ 
&=\delta _{k,l}\begin{pmatrix} 
\delta _{\frac{j-k}{2},m}^{(n)} & 0 \\[0.1cm]  
0 & \delta _{\frac{-j-k}{2},m}^{(n)} \end{pmatrix}, 
\end{align*}
we have 
\begin{align*}
\rho_{jk}(u_{aq\nu})
&=\sum\limits_{m=0}^{n-1}
\nu ^{k}\omega ^{-2aN(2m+k)^2-2qk^2n}\begin{pmatrix} 
\delta _{\frac{j-k}{2},m}^{(n)} & 0 \\[0.1cm]  
0 & \delta _{\frac{-j-k}{2},m}^{(n)} \end{pmatrix}. 
\end{align*}
Thus we have 
$$
\underline{\dim}_{R_{aq\nu }}\kern0.1em V_{jk}
=2\nu ^{k}\omega^{-2aN(2m+k)^2-2qk^2n}
\delta _{\frac{j-k}{2},m}^{(n)} =2\nu ^{k}\omega^{-2aNj^2-2qk^2n}.$$
(2) From  
$\chi_{ijk}(E_{l})=\delta _{k,l}$ and $
\rho_{1k}(E_l)=\delta _{k,l}\begin{pmatrix} 
1 & 0 \\ 
0 & 1
\end{pmatrix}$ we have 
\begin{align*}
\underline{\dim}_{R_d^{(1)}} V_{ijk}
&=\underline{\dim}_{R_d^{(2)}} V_{ijk}
=\sum\limits_{l=0}^{2N-1}\omega ^{-4dl^2} \chi_{ijk}(E_l) 
=\sum\limits_{l=0}^{2N-1}\omega ^{-4dl^2} \delta _{k,l} 
=\omega ^{-4dk^2}, \\ 
\underline{\dim}_{R_d^{(1)}} V_{1k}
&=\underline{\dim}_{R_d^{(2)}} V_{1k}
=\sum\limits_{l=0}^{2N-1}\omega ^{-4dl^2} \textrm{Tr}\kern0.1em \rho_{1k}(E_l) =\sum\limits_{l=0}^{2N-1}\omega ^{-4dl^2} \cdot 2\delta _{k,l} 
=2\omega ^{-4dk^2} . 
\end{align*}
\end{proof}

\par \medskip 
Combining the results in Proposition~\ref{5.18}, Proposition~\ref{5.20}  and Lemma~\ref{5.21}, we have the following. 

\par \medskip 
\begin{prop}\label{5.22}
Let $N\geq 1$ be an odd integer, and $n\geq 2$ an integer, and consider the group 
$$G_{Nn}=\langle h,\ t,\ w\ \vert \ t^2=h^{2N}=1,\ w^n=h^{N},\ tw=w^{-1}t,\ ht=th,\ hw=wh \rangle .$$
Let $\omega $ be a primitive $4nN$-th root of unity in a field $\boldsymbol{k}$ whose characteristic dose not divide $2nN$. 
Then the polynomial invariants of the group Hopf algebra $\boldsymbol{k}[G_{Nn}]$ are  given by the following. 
\begin{align*}
P_{\boldsymbol{k}[G_{Nn}]}^{(1)}(x)
&=\begin{cases}   
\prod\limits_{s=0}^{N-1} 
\prod\limits_{q=0}^{N-1}
(x-\omega^{-8nqs^2})^{4n}
(x^2-\omega^{-4nq(2s+1)^2})^{2n}
& (\textrm{$n$ is odd}),   \\[0.4cm]  
\prod\limits_{s=0}^{N-1} 
\prod\limits_{q=0}^{N-1}(x-\omega^{-8nqs^2})^{8n} 
& (\textrm{$n\geq 4$ is even}), \\[0.4cm]  
\prod\limits_{s=0}^{N-1} 
\prod\limits_{q=0}^{N-1}(x-\omega^{-16qs^2})^{32} 
& (\textrm{$n=2$}),
\end{cases} \displaybreak[1] \\ 
P_{\boldsymbol{k}[G_{Nn}]}^{(2)}(x)
&=\begin{cases} 
\prod\limits_{s=0}^{N-1}\prod\limits_{t=1}^{\frac{n-\epsilon (n)}{2}}
\prod\limits_{a=0}^{n-1}\prod\limits_{q=0}^{N-1}
( x^2- \omega ^{-4(nq(2s+1)^2+Na(2t-1)^2)}) \\ 
\qquad \qquad  \times \prod\limits_{s=0}^{N-1}\prod\limits_{t=1}^{\frac{n-2+\epsilon (n)}{2}}
\prod\limits_{a=0}^{n-1}\prod\limits_{q=0}^{N-1}( x- \omega ^{-8(nqs^2+Nat^2)})^2  & (\textrm{$n\geq 3$}), \\[0.4cm]  
\prod\limits_{s=0}^{N-1}\prod\limits_{q=0}^{N-1}
( x^4- \omega ^{-16q(2s+1)^2})( x^2- \omega ^{-8q(2s+1)^2})^2
 & (\textrm{$n=2$}), 
\end{cases} 
\end{align*}
where 
$$\epsilon (n)=\begin{cases} 
0 & \textrm{$n$ is even}, \\ 
1 & \textrm{$n$ is odd}. \end{cases}$$
\end{prop}
\begin{proof}
First, we consider the case of $n\geq 3$. 
By Proposition~\ref{5.18} any universal $R$-matrix of $\boldsymbol{k}[G_{Nn}]$ coincides with exactly one of $R_{aq\nu }\ (a=0,1,\ldots , n-1,\ q=0,1,\ldots , N-1,\ \nu =\pm 1)$.  
\par 
Suppose that $n\ (\geq 4)$ is even. 
Then by Lemma~\ref{5.21}(1), the polynomial invariant $P_{\boldsymbol{k}[G_{Nn}]}^{(1)}(x)$ 
is given by 
$$P_{\boldsymbol{k}[G_{Nn}]}^{(1)}(x)
=\prod\limits_{i,j=0}^1\prod\limits_{s=0}^{N-1} 
P_{\boldsymbol{k}[G_{Nn}],\ V_{ij,2s}}(x) 
=\prod\limits_{i,j=0}^1\prod\limits_{s=0}^{N-1} 
\prod\limits_{q=0}^{N-1}(x-\omega^{-8nqs^2})^{2n} 
=\prod\limits_{s=0}^{N-1} 
\prod\limits_{q=0}^{N-1}(x-\omega^{-8nqs^2})^{8n}.$$
By the same lemma, since $P_{\boldsymbol{k}[G_{Nn}],\ V_{jk}}(x)$ is given by $P_{\boldsymbol{k}[G_{Nn}],\ V_{jk}}(x)=
\prod\limits_{a=0}^{n-1}\prod\limits_{q=0}^{N-1}
\prod\limits_{\nu =\pm 1}( x- \nu ^{k}\omega ^{-2(nqk^2+Naj^2)})$
for the simple left $\boldsymbol{k}[G_{Nn}]$-module $V_{jk}$, 
we have 
\begin{align*}
P_{\boldsymbol{k}[G_{Nn}]}^{(2)}(x)
&=
\Bigl( \prod\limits_{s=0}^{N-1}\prod\limits_{t=1}^{\frac{n}{2}}
P_{\boldsymbol{k}[G_{Nn}],\ V_{2t-1, 2s+1}}(x) \Bigr) \cdot 
\Bigl( \prod\limits_{s=0}^{N-1}\prod\limits_{t=1}^{\frac{n-2}{2}}
P_{\boldsymbol{k}[G_{Nn}],\ V_{2t, 2s}}(x) \Bigr) \\ 
&=
\prod\limits_{s=0}^{N-1}\prod\limits_{t=1}^{\frac{n}{2}}
\prod\limits_{a=0}^{n-1}\prod\limits_{q=0}^{N-1}
\prod\limits_{\nu =\pm 1}( x- \nu \omega ^{-2(nq(2s+1)^2+Na(2t-1)^2)}) \\ 
&\qquad \qquad \qquad \qquad 
 \times \prod\limits_{s=0}^{N-1}\prod\limits_{t=1}^{\frac{n-2}{2}}
\prod\limits_{a=0}^{n-1}\prod\limits_{q=0}^{N-1}( x- \omega ^{-8(nqs^2+Nat^2)})^2 \displaybreak[0] \\ 
&=
\prod\limits_{s=0}^{N-1}\prod\limits_{t=1}^{\frac{n}{2}}
\prod\limits_{a=0}^{n-1}\prod\limits_{q=0}^{N-1}
( x^2- \omega ^{-4(nq(2s+1)^2+Na(2t-1)^2)}) \\ 
&\qquad \qquad \qquad \qquad 
 \times \prod\limits_{s=0}^{N-1}\prod\limits_{t=1}^{\frac{n-2}{2}}
\prod\limits_{a=0}^{n-1}\prod\limits_{q=0}^{N-1}( x- \omega ^{-8(nqs^2+Nat^2)})^2. 
\end{align*}

By a quite similar consideration, we calculate the polynomial invariants of  $\boldsymbol{k}[G_{Nn}]$ in the case where $n\ (\geq 3)$ is odd. 
\par 
Next, we consider the case of $n=2$. 
Then by Proposition~\ref{5.20}, any universal $R$-matrix of $\boldsymbol{k}[G_{N2}]$ coincides with exactly one of $R_{aq\nu }\ (a=0,1,\ q=0,1,\ldots , N-1,\ \nu =\pm 1),\ R^{(1)}_d, R^{(2)}_d\ (d=0,1,\ldots , 2N-1)$.  
Thus by Lemma~\ref{5.21}, the polynomial invariant $P_{\boldsymbol{k}[G_{N2}]}^{(1)}(x)$ 
is given by 
\begin{align*}
P_{\boldsymbol{k}[G_{N2}]}^{(1)}(x)
&=\prod\limits_{i,j=0}^1\prod\limits_{s=0}^{N-1} 
P_{\boldsymbol{k}[G_{N2}],\ V_{ij,2s}}(x) \\ 
&=\prod_{i,j=0,1}\prod_{s=0}^{N-1}\biggl( \prod_{q=0}^{N-1}(x-\omega ^{-4q(2s)^2})^4\cdot \prod_{d=0}^{2N-1}(x-\omega ^{-4d(2s)^2})^2\biggr) \\ 
&=\prod_{s=0}^{N-1}\prod_{q=0}^{N-1}(x-\omega ^{-16qs^2})^{16}\cdot \prod_{s=0}^{N-1}\prod_{d=0}^{2N-1}(x-\omega ^{-16ds^2})^8 \\ 
&=\prod_{s=0}^{N-1}\prod_{q=0}^{N-1}(x-\omega ^{-16qs^2})^{32}.
\end{align*}
Similarly, by Proposition~\ref{5.20} and Lemma~\ref{5.21}, we have 
\begin{align*}
P_{\boldsymbol{k}[G_{N2}]}^{(2)}(x)
&=\prod_{s=0}^{N-1}P_{\boldsymbol{k}[G_{N2}], V_{1, 2s+1}}(x) \displaybreak[0] \\ 
&=\prod_{s=0}^{N-1}\biggl( \prod_{a=0}^{1}\prod_{q=0}^{N-1}\prod_{\nu =\pm 1}(x-\nu \omega ^{-2aN-4q(2s+1)^2})\cdot 
\prod_{d=0}^{2N-1}(x-\omega ^{-4d(2s+1)^2})^2\biggr) \displaybreak[0] \\ 
&=\prod_{s=0}^{N-1}\biggl( \prod_{a=0}^{1}\prod_{q=0}^{N-1}(x^2- \omega ^{-4aN-8q(2s+1)^2})\cdot 
\prod_{d=0}^{2N-1}(x-\omega ^{-4d(2s+1)^2})^2\biggr) . 
\end{align*}
Here, 
\begin{align*}
\prod_{d=0}^{2N-1}(x-\omega ^{-4d(2s+1)^2})
&=\prod_{q=0}^{N-1}(x-\omega ^{-4q(2s+1)^2})\cdot \prod_{q=0}^{N-1}(x-\omega ^{-4(N+q)(2s+1)^2}) \\ 
&=\prod_{q=0}^{N-1}(x-\omega ^{-4q(2s+1)^2})\cdot \prod_{q=0}^{N-1}(x+\omega ^{-4q(2s+1)^2}) \displaybreak[0] \\ 
&=\prod_{q=0}^{N-1}(x^2-\omega ^{-8q(2s+1)^2}), 
\end{align*}
and whence  
\begin{align*}
P_{\boldsymbol{k}[G_{N2}]}^{(2)}(x)
&=\prod_{s=0}^{N-1}\biggl( \prod_{q=0}^{N-1}(x^2- \omega ^{-4N-8q(2s+1)^2})\cdot 
\prod_{q=0}^{N-1}(x^2-\omega ^{-8q(2s+1)^2})^3\biggr) \\ 
&=\prod_{s=0}^{N-1}\prod_{q=0}^{N-1}(x^4-\omega ^{-16q(2s+1)^2})(x^2-\omega ^{-8q(2s+1)^2})^2. 
\end{align*}
\end{proof}

\par \medskip 
For an odd integer $N\geq 1$ and an integer $n\geq 2$, 
we set 
$$A_{Nn}=\begin{cases} 
A_{Nn}^{++}  & \textrm{if $n$ is odd}, \\[0.1cm]  
A_{Nn}^{+-}  & \textrm{if $n$ is even}. 
\end{cases}$$
We see immediately that if $n$ is odd, then $P_{A_{Nn}}^{(d)}(x)=P_{\boldsymbol{k}[G_{Nn}]}^{(d)}(x)$ for $d=1,2$. 
So, our polynomial invariants do not detect the representation categories of $A_{Nn}$ and $\boldsymbol{k}[G_{Nn}]$.  
However, for an even integer $n$ we have: 

\par 
\medskip 
\begin{cor}
Let $N\geq 1$ be an odd integer, and $n\geq 2$ an even integer. 
Let $\omega $ be a primitive $4nN$-th root of unity in a field $\boldsymbol{k}$ whose characteristic dose not divide $2nN$. 
Then two Hopf algebras 
$A_{Nn}$ and  $\boldsymbol{k}[G_{Nn}]$ are not monoidally Morita equivalent. 
\end{cor} 
\begin{proof} 
First, let us consider the case of $n\geq 3$, and compare $P_{A_{Nn}}^{(2)}(x)$ and $P_{\boldsymbol{k}[G_{Nn}]}^{(2)}(x)$. 
By Corollary~\ref{5.11} we see that $\omega ^{-N}$ is a root of the polynomial $P_{A_{Nn}}^{(2)}(x)$ since $n$ is even. 
On the other hand, by Proposition~\ref{5.22} an arbitrary root of $P_{\boldsymbol{k}[G_{Nn}]}^{(2)}(x)$ is written in the form $\omega ^{2k}$ for some integer $k$. 
If $P_{A_{Nn}}^{(2)}(x)=P_{\boldsymbol{k}[G_{Nn}]}^{(2)}(x)$, then $P_{\boldsymbol{k}[G_{Nn}]}^{(2)}(x)$ has to possess $\omega ^{-N}$ as a root. 
Then $\omega ^{-N}=\omega ^{2k}$ for some $k$, that is, $N+2k\equiv 0\ \ (\textrm{mod}\kern0.1em 4nN)$. This leads to a contradiction such that $N$ is even. 
So, $P_{A_{Nn}}^{(2)}(x)\not= P_{\boldsymbol{k}[G_{Nn}]}^{(2)}(x)$, and by Theorem~\ref{2.5} 
${}_{A_{Nn}}\mathbb{M}$ and ${}_{\boldsymbol{k}[G_{Nn}]}\mathbb{M}$ are not equivalent as $\boldsymbol{k}$-linear monoidal categories. 
\par 
Next, let us consider the case of $n=2$, and compare $P_{A_{N2}}^{(1)}(x)$ and $P_{\boldsymbol{k}[G_{N2}]}^{(1)}(x)$. 
By Corollary~\ref{5.11}  we see that $-1$ is a root of the polynomial $P_{A_{Nn}}^{(1)}(x)$. 
However, by Proposition~\ref{5.22} an arbitrary root of $P_{\boldsymbol{k}[G_{N2}]}^{(1)}(x)$ is written in the form $\omega ^{16k}$ for some integer $k$. 
By a similar argument above, we see that $-1=\omega ^{4N}$ is not a root of $P_{\boldsymbol{k}[G_{N2}]}^{(1)}(x)$. 
Thus 
$P_{A_{N2}}^{(1)}(x)\not= P_{\boldsymbol{k}[G_{N2}]}^{(1)}(x)$, and hence by Theorem~\ref{2.5} 
${}_{A_{N2}}\mathbb{M}$ and ${}_{\boldsymbol{k}[G_{N2}]}\mathbb{M}$ are not equivalent as $\boldsymbol{k}$-linear monoidal categories. 
\end{proof} 

\par \smallskip 
\begin{exam}\rm 
For a non-negative integer $h$, $\Phi_h$ denotes the $h$-th cyclotomic polynomial.  
Then by using Maple12 software,  we see that the polynomial invariants of Hopf algebras $\boldsymbol{k}[G_{Nn}]$ and $A_{Nn}$ for $N=1,3,5$ and $n=2,3,4$ are given as in the following table. 

\par \medskip 
\begin{center}
\renewcommand{\arraystretch}{1.3}
\begin{tabular}{c||c|c}\hline 
Hopf algebra $A$ & $P_A^{(1)}(x)$ & $P_A^{(2)}(x)$ \\ \hline 
$\boldsymbol{k}[G_{12}]$ & $\Phi_1^{32}$ &  $\Phi_4\Phi_2^3\Phi_1^3$ \\ 
$A_{12}$ & $\Phi_2^{16}\Phi_1^{16}$ &  $\Phi_8\Phi_2^2\Phi_1^2$ \\ \hline 
$\boldsymbol{k}[G_{32}]$ & $\Phi_3^{64}\Phi_1^{160}$ &  $\Phi_{12}^2\Phi_4^5\Phi_6^6\Phi_3^6\Phi_2^{15}\Phi_1^{15}$ \\ 
$A_{32}$ & $\Phi_6^{32}\Phi_3^{32}\Phi_2^{80}\Phi_1^{80}$ & $\Phi_{24}^2\Phi_8^5\Phi_6^4\Phi_3^4\Phi_2^{10}\Phi_1^{10}$  \\ \hline 
$\boldsymbol{k}[G_{52}]$ & $\Phi_5^{128}\Phi_1^{288}$ &  $\Phi_{20}^4\Phi_{10}^{12}\Phi_5^{12}\Phi_4^9\Phi_2^{27}\Phi_1^{27}$ \\ 
$A_{52}$ & $\Phi_{10}^{64}\Phi_5^{64}\Phi_2^{144}\Phi_1^{144}$ &  $\Phi_{40}^4\Phi_{10}^8\Phi_5^8  \Phi_8^9\Phi_2^{18}\Phi_1^{18}$ \\ \hline 
$\begin{matrix} 
\boldsymbol{k}[G_{13}] \\[0.2mm] 
A_{13}\end{matrix}$ & $\Phi_2^6\Phi_1^{18}$ &  $\Phi_6\Phi_3^3\Phi_2\Phi_1^3$ \\ \hline 
$\begin{matrix} 
\boldsymbol{k}[G_{33}] \\[0.2mm] 
A_{33}\end{matrix}$ & $\Phi_6^{12}\Phi_3^{36}\Phi_2^{30}\Phi_1^{90}$ &  $\Phi_6^9\Phi_3^{27}\Phi_2^9\Phi_1^{27}$ \\ \hline 
$\begin{matrix} 
\boldsymbol{k}[G_{53}] \\[0.2mm] 
A_{53}\end{matrix}$ & $\Phi_{10}^{24}\Phi_5^{72}\Phi_2^{54}\Phi_1^{162}$ &  $\Phi_{30}^4\Phi_{15}^{12}\Phi_{10}^4\Phi_5^{12}\Phi_6^9\Phi_3^{27}\Phi_2^9\Phi_1^{27}$ \\ \hline 
$\begin{matrix} 
\boldsymbol{k}[G_{14}] \\[0.2mm] 
A_{14}\end{matrix}$ & $\Phi_1^{32}$ & $\begin{matrix} 
\Phi_8^2\Phi_4^2\Phi_2^6\Phi_1^6 \\[0.2mm] 
\Phi_{16}^2\Phi_4^4 
\end{matrix} $ \\ \hline 
$\begin{matrix} 
\boldsymbol{k}[G_{34}] \\[0.2mm] 
A_{34}\end{matrix}$ & $\Phi_3^{64}\Phi_1^{160}$ & 
$\begin{matrix} 
\Phi_{24}^4\Phi_{12}^4\Phi_8^{10}\Phi_6^{12}\Phi_3^{12}\Phi_4^{10}\Phi_2^{30}\Phi_1^{30} \\[0.2mm]  
\Phi_{48}^4\Phi_{16}^{10}\Phi_{12}^8\Phi_4^{20}
\end{matrix} $ \\ \hline 
$\begin{matrix} 
\boldsymbol{k}[G_{54}] \\[0.2mm] 
A_{54}\end{matrix}$ & $\Phi_5^{128}\Phi_1^{288}$ &  
$\begin{matrix} 
\Phi_{40}^8\Phi_{20}^8\Phi_{10}^{24}\Phi_5^{24}\Phi_8^{18}\Phi_4^{18}\Phi_2^{54}\Phi_1^{54} \\[0.2mm] 
\Phi_{80}^8\Phi_{20}^{16}\Phi_{16}^{18}\Phi_4^{36}
\end{matrix} $ \\ \hline 
\end{tabular}
\end{center}
\end{exam}

\par \medskip 
We note that the representation rings of $A_{Nn}$ and $\boldsymbol{k}[G_{Nn}]$ are isomorphic as rings with $\ast$-structure (see Proposition~\ref{5.23} for details  in the next section). 
Thus the pair of two Hopf algebras $A_{Nn}$ and $\boldsymbol{k}[G_{Nn}]$ gives an example of that their representation rings are isomorphic, meanwhile their  representation categories are not.

\par \bigskip 
\section{Appendix: The representation ring and self-duality of $A_{Nn}^{+\lambda }$}
\par 
In this appendix, in the case where $N\geq 1$ is odd, by analyzing the algebraic structure of $A_{Nn}^{+\lambda }$ we introduce a \lq\lq convenient" basis of $A_{Nn}^{+\lambda }$ to compute the braidings $\sigma _{\alpha \beta}$ given in the previous section, and  determine the structure of the representation ring of it. 
As an application, we determine when $A_{Nn}^{+\lambda }$ is self-dual. 
As one more application, we show that if $n$ is even, then the representation ring of $A_{Nn}^{++}$ is non-commutative. 
This means that the dual Hopf algebra of $A_{Nn}^{+\lambda }$ has no quasitriangular structure. 
These results have already shown in \cite{CDMM} in the case of $N=1$. 
\par 
Throughout this section we assume that $N\geq 1$ is an odd integer,  $n\geq 2$ is an integer, and $\lambda =\pm 1$. 

\par \smallskip 
\subsection{The algebra structure of $A_{Nn}^{+\lambda}$} 
First of all, we determine the algebra structure of the Hopf algebra $A_{Nn}^{+\lambda }$. 
This is done by Masuoka \cite{Mas2} for the case of $N=1$ (see also \cite{CDMM}). 
\par  \smallskip 
\begin{prop}\label{5.24}
Let $G$ be the finite group presented by 
$$G=\langle h,\ t,\ w\ \vert \ t^2=h^{2N}=1,\ w^n=h^{(n+\frac{\lambda -1}{2})N},\ tw=w^{-1}t,\ ht=th,\ hw=wh \rangle .$$
Then there is an algebra isomorphism 
$\varphi :\boldsymbol{k}[G]\longrightarrow A^{+ \lambda }_{Nn}$ such that 
\begin{align}
\varphi (h)&=x_{11}^2-x_{12}^2, \label{eq6.1}\\ 
\varphi (t)&=x_{12}^N+x_{22}^N, \label{eq6.2}\\ 
\varphi (w)&=x_{11}^{2N-1}x_{22}-x_{21}^{2N-1}x_{12}.  \label{eq6.3}
\end{align}
\end{prop}
\begin{proof} 
First, we will check the well-definedness of $\varphi $. 
To do this it is sufficient to show that $\varphi $ preserves the defining relations of $G$. 
By definition we have two equations 
$\varphi (t)^2=(x_{12}^N+x_{22}^N)^2=x_{12}^{2N}+x_{22}^{2N}=1$ and 
$\varphi (h)^{2N}=(x_{11}^2-x_{12}^2)^{2N}=x_{11}^{4N}+x_{12}^{4N}=1$, and 
since $\varphi (h)=x_{11}^2-x_{12}^2$ is in the center of $A^{+ \lambda }_{Nn}$ by Lemma~\ref{5.1}, two equations 
$\varphi (h)\varphi (t)=\varphi (t)\varphi (h)$ and $\varphi (h)\varphi (w)=\varphi (w)\varphi (h)$ are obtained.  
Now, we set 
$z=hw$, and define 
$$\varphi (z)=x_{11}x_{22}+x_{21}x_{12}.$$
Then we have $\varphi (z)=\varphi (h)\varphi (w)$, 
and by using Lemma~\ref{5.1}(5), 
\begin{align*}
\varphi (z)^n
&=(x_{11}x_{22}+x_{21}x_{12})^n 
=(x_{11}x_{22})^{n}+(x_{21}x_{12})^n \\ 
&=x_{11}^{2n}+x_{12}^{2n} 
=(x_{11}^{2}+\lambda x_{12}^{2})(x_{11}^{2}+x_{12}^{2})^{n-1} \\ 
&=\varphi (h)^{(N+1)n+\frac{\lambda -1}{2}N}. 
\end{align*} 
\noindent 
Therefore, $\varphi (w)=\varphi (h)^{(n+\frac{\lambda-1}{2})N}$. 
To show the equation $\varphi (t)\varphi (w)=\varphi (w)^{-1}\varphi (t)$, it is sufficient to show that $\varphi (z)\varphi (t)\varphi (z)$ $=\varphi (h)^2\varphi (t)$ since 
\begin{align*}
\varphi (t)\varphi (w)=\varphi (w)^{-1}\varphi (t) 
& \ \ \Longleftrightarrow \ \ \varphi (t)\varphi (z)=\varphi (h)\varphi (w)^{-1}\varphi (t) \\ 
& \ \ \Longleftrightarrow \ \ \varphi (w)\varphi (t)\varphi (z)=\varphi (h)\varphi (t) \\ 
& \ \ \Longleftrightarrow \ \ \varphi (z)\varphi (t)\varphi (z)=\varphi (h)^2\varphi (t). 
\end{align*}

The last equation can be derived as follows. 
\begin{align*}
\varphi (z)\varphi (t)\varphi (z)
&=(x_{11}x_{22}+x_{21}x_{12})(x_{12}^N+x_{22}^N)(x_{11}x_{22}+x_{21}x_{12})\\ 
&=x_{11}x_{22}^{N+1}x_{11}x_{22}+x_{21}x_{12}^{N+1}x_{21}x_{12}\\ 
&=x_{11}^2x_{22}^{N+2}+x_{21}^2x_{12}^{N+2}\\ 
&=(x_{11}^2x_{22}^2+x_{21}^2x_{12}^2)(x_{22}^N+x_{12}^N)\\ 
&=(x_{11}^4+x_{12}^4)(x_{22}^N+x_{12}^N)\\ 
&=\varphi (h)^2\varphi (t). 
\end{align*}

\noindent 
Thus there is an algebra map $\varphi :\boldsymbol{k}[G]\longrightarrow A^{+ \lambda }_{Nn}$ satisfying (\ref{eq6.1}), (\ref{eq6.2}) and (\ref{eq6.3}). 
\par 
Next, we show that $\varphi $ is surjective. 
For this, it is enough to show that 
$x_{11},\ x_{12},\ x_{21},\ x_{22}$ belong to the image of $\varphi $. 
Since $N$ is odd, the equation $x_{22}^2+x_{12}^2=x_{11}^2+x_{12}^2=\varphi (h)^{N+1}$ holds. 
Putting $N=2m+1$, we have 
$$\varphi (t)=(x_{12}+x_{22})(x_{12}^2+x_{22}^2)^m=(x_{12}+x_{22})\varphi (h)^{m(N+1)},$$
and whence 
\begin{equation}\label{eq5.27}
x_{12}+x_{22}=\varphi (th^{-m(N+1)}). 
\end{equation}
Since $\varphi (h)=(x_{22}-x_{12})(x_{22}+x_{12})=(x_{22}-x_{12})\varphi (th^{-m(N+1)})$, 
we also have 
\begin{equation}\label{eq5.28}
x_{22}-x_{12}=\varphi (th^{m(N+1)+1}).
\end{equation}
From the equations (\ref{eq5.27}) and (\ref{eq5.28}), 
we see that $x_{12}, x_{22}$ are contained in the image of $\varphi $. 
\par 
Using the equation  
$x_{11}\varphi (z)=x_{11}(x_{11}x_{22}+x_{21}x_{12})=x_{11}^2x_{22}=x_{22}^3$, we have 
$x_{11}=x_{22}^3\varphi (z)^{-1}=x_{22}^3\varphi (z^{-1})$. 
This shows that $x_{11}$ is in the image of $\varphi $. 
Similarly, using the equation $x_{21}\varphi (z)=x_{21}^2x_{12}=x_{12}^3$ 
we have
$x_{21}=x_{12}^3\varphi (z)^{-1}=x_{12}^3\varphi (z^{-1})$, 
and whence $x_{21}$ is in the image of $\varphi $. 
\par 
Finally, we show that $\varphi $ is injective. 
From $\dim A_{Nn}^{+\lambda }=4nN$ and the surjectivity of 
$\varphi $, it follows that $4nN=\dim A_{Nn}^{+ \lambda }\leq |G|$. 
On the other hand, since 
$$G=\{ t^iw^jh^k\ \vert \ i=0,1,\ j=0,1,\ldots ,n-1,\ k=0,1,\ldots , 2N-1\} $$
as a set, we see that $|G|\leq 4nN$. 
Therefore, $|G|=4nN$, and $\varphi $ is injective. 
\end{proof}

\par \smallskip 
From the above proposition we see that $A_{Nn}\cong \boldsymbol{k}[G_{Nn}]$ as algebras. 

\par \smallskip 
\begin{cor}\label{5.25}
Set $N=2m+1$, and consider the following elements in $A_{Nn}^{+\lambda }$: 
$$h:=x_{11}^2-x_{12}^2,\quad t:=x_{12}^N+x_{22}^N,\quad  
w:=x_{11}^{2N-1}x_{22}-x_{21}^{2N-1}x_{12}.$$ 
If $\textrm{ch}(\boldsymbol{k})\not= 2$, then the following relations hold:
\par 
(1) $t^2=h^{2N}=1,\ w^n=h^{(n+\frac{\lambda -1}{2})N},\ tw=w^{-1}t,\ ht=th,\ hw=wh$.
\vspace{0.2cm}\par 
(2) $
x_{22}=\frac{h^{-m(N+1)}+h^{m(N+1)+1}}{2}t, \quad 
x_{12}=\frac{h^{-m(N+1)}-h^{m(N+1)+1}}{2}t$.
\vspace{0.2cm}\par 
(3) $x_{11}^{2N-1}+x_{21}^{2N-1}=wth^{-m(N+1)-1}$.
\vspace{0.2cm}\par 
(4) $\begin{cases} 
x_{11}^N=\frac{h^N+1}{2}wt, \\[0.2cm]  
x_{21}^N=\frac{h^N-1}{2}wt, 
\end{cases}
\qquad \begin{cases} 
x_{12}^N=-\frac{h^N-1}{2}t, \\[0.2cm]  
x_{22}^N=\frac{h^N+1}{2}t.
\end{cases}$
\vspace{0.2cm}\par 
(5) $\begin{cases} 
x_{11}^{2N-1}x_{22}=\frac{1+h^N}{2}w, \\[0.2cm]   
x_{21}^{2N-1}x_{12}=\frac{-1+h^N}{2}w,
\end{cases}$ 
$\begin{cases} 
x_{12}^{2N-1}x_{21}=x_{12}x_{21}^{2N-1}=\frac{h^{N}-1}{2}w^{-1},\\[0.2cm]  
x_{22}^{2N-1}x_{11}=x_{22}x_{11}^{2N-1}=\frac{h^{N}+1}{2}w^{-1}.
\end{cases}$ 
\vspace{0.2cm}\par 
In particular, 
$$w^{-1}=x_{22}^{2N-1}x_{11}-x_{12}^{2N-1}x_{21}.$$
\end{cor}
\begin{proof}
\par 
The equations in (1) are already derived in the proof of Proposition~\ref{5.24}.  
The equations in (2) are obtained by solving the system of linear equations
$$\begin{cases}
x_{22}+x_{12}=th^{-m(N+1)}, \\ 
x_{22}-x_{12}=th^{m(N+1)+1},  
\end{cases}$$
which appears in the proof of Proposition~\ref{5.24}. 
The equation (3) is obtained by substituting $x_{22}-x_{12}=th^{m(N+1)+1}$ to 
$w=(x_{11}^{2N-1}+x_{21}^{2N-1})(x_{22}-x_{12})$. 
\par 
(4) Since $x_{22}-x_{12}=th^{m(N+1)+1}$, and $N$ is odd, we have 
$x_{22}^N-x_{12}^N=t^Nh^{mN(N+1)+N}=th^N$. 
By solving the system of the equations  $x_{22}^N-x_{12}^N=th^N$,\ $x_{22}^N+x_{12}^N=t$, we get the equations 
$x_{22}^N=\frac{h^N+1}{2}t, \ 
x_{12}^N=\frac{1-h^N}{2}t$, 
and by Parts (2) and (3) we have 
\begin{align*}
x_{11}^N
&=x_{11}^{N+1}(x_{11}^{2N-1}+x_{21}^{2N-1})  \\ 
&=x_{22}x_{22}^{N}(x_{11}^{2N-1}+x_{21}^{2N-1}) \\ 
&=\frac{h^{-m(N+1)}+h^{m(N+1)+1}}{2}t\times \frac{h^N+1}{2}t\times wth^{-m(N+1)-1} \\ 
&=wt\frac{h^N+1}{2}. 
\end{align*}

We will calculate $x_{21}^N$. 
Since $N-1$ is even, we have 
$$\begin{cases}
x_{11}^Nx_{22}^N=x_{11}^Nx_{22}^{N-1}x_{22}=x_{11}^{2N-1}x_{22},\\[0.1cm]  
x_{21}^Nx_{12}^N=x_{21}^Nx_{12}^{N-1}x_{12}=x_{21}^{2N-1}x_{12}.  
\end{cases} 
$$
Hence $w$ is expressed as 
$$w=x_{11}^Nx_{22}^N-x_{21}^Nx_{12}^N=(x_{11}^N+x_{21}^N)(x_{22}^N-x_{12}^N)=(x_{11}^N+x_{21}^N)th^N.$$
By postmultiplying $th^N$ to the both sides of the above equation 
we have
$wth^N=x_{11}^N+x_{21}^N$. 
Thus $x_{21}^N=wth^N-x_{11}^N=\frac{h^N-1}{2}wt$. 
\par 
(5) From Part (3), we have 
\begin{align*}
x_{11}^{2N-1}x_{22}+x_{21}^{2N-1}x_{12}
&=(x_{11}^{2N-1}+x_{21}^{2N-1})(x_{22}+x_{12}) \\ 
&=wth^{-m(N+1)-1}\cdot th^{-m(N+1)} 
=wh^{-2m-1} 
=wh^N. 
\end{align*}
By solving the system of the equations $x_{11}^{2N-1}x_{22}+x_{21}^{2N-1}x_{12}wh^N$, $x_{11}^{2N-1}x_{22}-x_{21}^{2N-1}x_{12}=w$, we have 
$x_{11}^{2N-1}x_{22}=\frac{1+h^N}{2}w, \ 
x_{21}^{2N-1}x_{12}= \frac{-1+h^N}{2}w$, 
and whence 
$$x_{22}^{2N-1}x_{11}
=x_{22}x_{11}^{2N-1} 
=\frac{h^{-m(N+1)}+h^{m(N+1)+1}}{2}twth^{-m(N+1)-1} =\frac{h^{N}+1}{2}w^{-1}.$$
Furthermore, we have 
$x_{12}^{2N-1}x_{21}=x_{12}x_{21}^{2N-1}$,  
$x_{22}^{2N-1}x_{11}=x_{22}x_{11}^{2N-1}$, 
and 
$$x_{22}x_{11}^{2N-1}-x_{12}x_{21}^{2N-1}=(x_{22}-x_{12})(x_{11}^{2N-1}+x_{21}^{2N-1})=th^{m(N+1)+1}\cdot wth^{-m(N+1)-1}=w^{-1}.$$
So, we obtain the desired equation $x_{12}^{2N-1}x_{21}=x_{22}^{2N-1}x_{11}-w^{-1}
=\frac{h^{N}+1}{2}w^{-1}-w^{-1}
=\frac{h^{N}-1}{2}w^{-1}$. 
\end{proof}

\par \smallskip 
\begin{cor}\label{5.26}
Let $G$ be the finite group presented by 
$$G=\langle h,\ t,\ w\ \vert \ t^2=h^{2N}=1,\ w^n=h^{(n+\frac{\lambda -1}{2})N},\ tw=w^{-1}t,\ ht=th,\ hw=wh \rangle .$$
For the group algebra $\boldsymbol{k}[G]$ over $\boldsymbol{k}$ of $\textrm{ch}(\boldsymbol{k})\not= 2$, we define algebra maps  $\Delta :\boldsymbol{k}[G]\otimes \boldsymbol{k}[G]\longrightarrow \boldsymbol{k}[G]$,\ $\varepsilon :\boldsymbol{k}[G]\longrightarrow \boldsymbol{k}$ and an anti-algebra map $S:\boldsymbol{k}[G]\longrightarrow \boldsymbol{k}[G]$ as follows: 
\begin{align*}
&\Delta (h)=h\otimes h,\quad \Delta (t)=h^Nwt\otimes e_1t+t\otimes e_0t,\quad \Delta (w)=w\otimes e_0w+w^{-1}\otimes e_1w, \\ 
&\varepsilon (h)=1,\quad  \varepsilon (t)=1,\quad \varepsilon (w)=1,\\ 
&S(h)=h^{-1},\quad S(t)=(e_0-e_1w) t, \quad S(w)=e_0w^{-1}+e_1w,
\end{align*}
\noindent 
where $e_0=\frac{1+h^N}{2},\ e_1=\frac{1-h^N}{2}$. 
Then the algebra isomorphism $\varphi :\boldsymbol{k}[G]\longrightarrow A_{Nn}^{+ \lambda}$ defined in Proposition~\ref{5.24}  is a Hopf algebra isomorphism.  
\end{cor}

\par \smallskip 
\subsection{The representation ring of $A_{Nn}^{+\lambda}$}
Via the algebra isomorphism $\varphi $ given in Proposition~\ref{5.24}, 
one can determine the structure of the representation ring of $A_{Nn}^{+\lambda }$. 
\par 
Let us recall the definition of the representation ring of a semisimple Hopf algebra, which is a natural extension of that of a finite group \cite{NR, Nik}. 
Let $A$ be a semisimple Hopf algebra of finite dimension over a field $\boldsymbol{k}$. 
By $\mathfrak{R}(A)$ we denote the set of isomorphism classes of finite-dimensional left $A$-modules, and for a finite-dimensional left $A$-module $V$
denote by $[V]$ the isomorphism class of $V$. 
Then $\mathfrak{R}(A)$ has a semiring structure induced by 
$$[V]+[W]=[V\oplus W],\quad [V][W]=[V\otimes W].$$
Also, $\mathfrak{R}(A)$ has the unit element, which is given by $[\boldsymbol{k}] $, where the left $A$-module action of $\boldsymbol{k}$ is due to the counit $\varepsilon $. 
Let $\textrm{Rep}(A)$ denote the Grothendieck group constructed from the enveloping group of $\mathfrak{R}(A)$ as an abelian semigroup. 
Then the semiring structure of $\mathfrak{R}(A)$ uniquely determines a ring structure of $\textrm{Rep}(A)$. 
Furthermore, the ring $\textrm{Rep}(A)$ has an anti-homomorphism of rings $\ast :\textrm{Rep}(A)\longrightarrow \textrm{Rep}(A)$, which induced from the antipode $S$. 
Explicitly, this anti-homomorphism $\ast$ is defined by $[V] \longmapsto [V^{\ast}]$ for a finite-dimensional left $A$-module $V$. 
We call the ring $\textrm{Rep}(A)$ with $\ast$ the \textit{representation ring} of $A$. 
In general, $\ast : \textrm{Rep}(A)\longrightarrow \textrm{Rep}(A)$ is not an involution, and  the representation ring $\textrm{Rep}(A)$ is not commutative. 
However, if the Hopf algebra $A$ possesses a universal $R$-matrix, then $\ast $ is an involution, 
and for two left $A$-modules $V$ and $W$ an $A$-linear isomorphism $c_{V,W}:V\otimes W\longrightarrow W\otimes V$ is defined by use of $R$, and whence we see that $\textrm{Rep}(A)$ is commutative. 
We note that 
$\textrm{Rep}(A)$ is a free $\mathbb{Z}$-module with finite rank, and 
a $\mathbb{Z}$-basis of $\textrm{Rep}(A)$ is given by the isomorphism classes of simple left $A$-modules. 

\par \medskip 
\begin{lem}\label{6.4}
Let $\boldsymbol{k}$ be a field whose characteristic dose not divide $2nN$,  and suppose that there is a primitive $4nN$-th root of unity in $\boldsymbol{k}$. 
For integers $i,j$ and an even integer $k$ let $\chi_{ijk}$ be the one-dimensional representation of the algebra $A_{Nn}=\boldsymbol{k}[G_{Nn}]$ given by (\ref{eq5.21}), and for integers $j,k$ with $j \equiv k\ (\rm{mod}\ 2)$ let $\rho _{jk}$ be the two-dimensional representation of the algebra $A_{Nn}=\boldsymbol{k}[G_{Nn}]$ given by (\ref{eq5.22}). 
We set 
$$\epsilon (n)=\begin{cases} 
0  & (\textrm{$n$ is even}), \\ 
1 & (\textrm{$n$ is odd}).
\end{cases}$$
Then as representations of the Hopf algebra $A_{Nn}$ the following hold for $i, j, k, i^{\prime}, j^{\prime}, k^{\prime} \in \mathbb{Z}$.  
\begin{enumerate}
\item[$\bullet$] 
\begin{enumerate}
\item[(i)] $[\rho_{2n+j,k}]=[\rho_{jk}]=[\rho_{-j,k}]$ for $j\equiv k\ (\rm{mod}\ 2)$, and 
$[\rho_{n+j,k}]=[\rho_{n-j,k}]$ for $n+j\equiv k\ (\rm{mod}\ 2)$,  
\item[(ii)] $[\rho_{0k}]=[\chi_{00k}\oplus \chi_{10k}]$ for $k\equiv 0\ (\rm{mod}\ 2)$, and \par  
$[\rho_{nk}]=[\chi_{01,k-\epsilon (n)}\oplus \chi_{11,k-\epsilon (n)}]$ for  $k\equiv n\ (\rm{mod}\ 2)$. 
\end{enumerate}
\item[$\bullet$] on representations of tensor products  
\begin{enumerate}
\item[(iii)] $[\chi_{ijk}\otimes \chi_{i^{\prime}j^{\prime}k^{\prime}}]=[\chi_{i+i^{\prime}, j+j^{\prime}, k+k^{\prime}}]$, where $k,k^{\prime}$ are even, 
\item[(iv)] $[\chi_{ijk}\otimes \rho_{j^{\prime}k^{\prime}}]=[\rho_{j^{\prime}k^{\prime}}\otimes \chi_{ijk}]=[\rho_{nj+j^{\prime}, k+k^{\prime}+\epsilon (n)j}]$ for $k\equiv 0,\ j^{\prime}\equiv k^{\prime}\ (\rm{mod}\ 2)$, 
\item[(v)] $
[\rho_{jk}\otimes \rho_{j^{\prime}k^{\prime}}]
=[\rho_{j+j^{\prime}, k+k^{\prime}}\oplus \rho_{j-j^{\prime}, k+k^{\prime}}]
$ for $j\equiv k,\ j^{\prime}\equiv k^{\prime}\ (\rm{mod}\ 2)$. 
\end{enumerate}
\item[$\bullet$] on contragredient representations $\chi _{ijk}^{\ast}$ and $\rho _{jk}^{\ast}$
\begin{enumerate}
\item[(vi)] $[\chi _{ijk}^{\ast}]=[\chi_{i,-j,-k}]$ for $k\equiv 0\ (\rm{mod}\ 2)$,  
\item[(vii)] $[\rho _{jk}^{\ast}]=[\rho_{j,-k}]$ for $j\equiv k\ (\rm{mod}\ 2)$. 
\end{enumerate}
\end{enumerate}
\end{lem} 
\begin{proof} 
(i) By definition $[\rho _{2n+j,k}]=[\rho _{j,k}]$ and $[\rho_{-j,k}]=[\rho _{jk}]$ 
are obtained immediately. 
By using these equations we have $[\rho _{n+j,k}]=[\rho _{-(n+j),k}]=[\rho _{2n-(n+j),k}]=[\rho _{n-j,k}]$. 
\par 
(ii) Let $\{ e_1,e_2\} $ be the standard basis of $V_{0k}=\boldsymbol{k}\oplus \boldsymbol{k}$. Then the subspaces $\boldsymbol{k}(e_1+e_2)$ and $\boldsymbol{k}(e_1-e_2)$ are $\rho _{0k}$-invariant, and 
$\boldsymbol{k}(e_1+e_2)=V_{00k}$ and $\boldsymbol{k}(e_1-e_2)=V_{10k}$ as submodules of $V_{jk}$. 
This implies that $[\rho _{0k}]=[\chi _{00k}\oplus \chi _{10k}]$. 
Similarly, the subspaces $\boldsymbol{k}(e_1+e_2)$ and $\boldsymbol{k}(e_1-e_2)$ of $V_{nk}$ are also $\rho _{nk}$-invariant, and 
$$\boldsymbol{k}(e_1+e_2)=\begin{cases} 
V_{01k} & (\textrm{$n$ is even}), \\ 
V_{01,k-1} & (\textrm{$n$ is odd}), \end{cases}
\quad \textrm{and} \quad 
\boldsymbol{k}(e_1-e_2)=\begin{cases} 
V_{11k} & (\textrm{$n$ is even}), \\ 
V_{11,k-1} & (\textrm{$n$ is odd}), \end{cases}$$
as submodules of $V_{nk}$. 
This proves $[\rho_{nk}]=[\chi_{01,k-\epsilon (n)}\oplus \chi_{11,k-\epsilon (n)}]$. 
\par 
(iii) It follows from $\chi _{ijk}\otimes \chi _{i'j'k'}=\chi _{i+i',j+j',k+k'}$.  
\par 
(iv) By using the coproduct $\Delta $ given in Corollary~\ref{5.26} we see that $(\chi _{ijk}\otimes \rho _{j'k'})(t)$, $(\chi _{ijk}\otimes \rho _{j'k'})(w)$, $(\chi _{ijk}\otimes \rho _{j'k'})(h)$ are represented by the matrices 
$$(-1)^{i+k'j}\begin{pmatrix} 0 & 1 \\ 1 & 0 \end{pmatrix}, \quad 
\begin{pmatrix} \omega ^{2(nj+j')N} & 0 \\ 0 & \omega^{-2(nj+j')N} \end{pmatrix},\quad  
\omega ^{2(k+k'+\epsilon (n)j)n}\begin{pmatrix} 1 & 0 \\ 0 & 1 \end{pmatrix},
$$
respectively. 
Let $\{ e_1, e_2\} $ be the standard basis of $\boldsymbol{k}^2$. 
Then by considering the matrix presentation of $\chi _{ijk}\otimes \rho _{j'k'}$  with respect to the basis $\{ e_1,\  (-1)^{i+k'j}e_2\} $, we see that $[\chi _{ijk}\otimes \rho _{j'k'}]=[\rho _{nj+j', k+k'+\epsilon (n)j}]$. 
\par 
Similarly, we see that if $n$ or $j$ is even, then 
$(\rho _{j'k'}\otimes \chi _{ijk})(t)$, $(\rho _{j'k'}\otimes \chi _{ijk})(w)$, $(\rho _{j'k'}\otimes \chi _{ijk})(h)$ are represented by the matrices 
$$(-1)^i\begin{pmatrix} 0 & 1 \\ 1 & 0 \end{pmatrix},\quad \begin{pmatrix} \omega^{2(nj+j')N} & 0 \\ 0 & \omega^{-2(nj+j')N} \end{pmatrix},\quad 
\omega ^{2(k+k'+\epsilon (n)j)n}\begin{pmatrix} 1 & 0 \\ 0 & 1 \end{pmatrix},$$
respectively, and if $n$ and $j$ are odd, then they are represented by the matrices 
$$(-1)^i\begin{pmatrix} 0 & \omega ^{2(nj'+j')N} \\ \omega ^{-2(nj'+j')N} & 0 \end{pmatrix},\quad 
\begin{pmatrix} \omega^{-2(nj+j')N} & 0 \\ 0 & \omega ^{2(nj+j')N} \end{pmatrix},\quad 
\omega ^{2(k+k'+j)n}\begin{pmatrix} 1 & 0 \\ 0 & 1 \end{pmatrix}, $$
respectively. 
So, in the case where $n$ or $j$ is even, by changing basis from $\{ e_1, e_2\} $ to $\{ e_1,\  (-1)^{i}e_2\} $ we see that $[\rho _{j'k'}\otimes \chi _{ijk}]=[\rho _{nj+j',k+k'+\epsilon (n)j}]$, and in the case where $n$ and $j$ are odd, by changing basis from $\{ e_1, e_2\} $ to $\{ e_2,\  (-1)^{i}\omega ^{2(nj'+j')N}e_1\} $ we have the same equation $[\rho _{j'k'}\otimes \chi _{ijk}]=[\rho _{nj+j',k+k'+j}]=[\rho _{nj+j',k+k'+\epsilon (n)j}]$. 
\par \noindent 
(v) Let $\{ e_1, e_2\}$ and $\{ e_1^{\prime}, e_2^{\prime}\}$ be the bases of $V_{jk}$ and $V_{j^{\prime}k^{\prime}}$, such that 
the matrix representations of $\rho_{jk}$ and $\rho_{j^{\prime}k^{\prime}}$ with respect to the bases are given by (\ref{eq5.22}), respectively. 
Then the action of $\rho _{jk}\otimes \rho _{j'k'}$ on $V_{jk}\otimes V_{j^{\prime}k^{\prime}}$ is given by 
 \begin{align*} 
t\cdot e_a\otimes e_b^{\prime}&=\begin{cases}
e_{3-a}\otimes e_{3-b}^{\prime} & (\textrm{if $k^{\prime}$ is even}),\\ 
\omega^{2(-1)^b(nj+j)N}e_{3-a}\otimes e_{3-b}^{\prime} & (\textrm{if $k^{\prime}$ is odd}), 
\end{cases} \\ 
w\cdot e_a\otimes e_b^{\prime}&=\begin{cases}
\omega ^{2((-1)^{1-a}j+(-1)^{1-b}j^{\prime})N}e_a\otimes e_b^{\prime}  & (\textrm{if $k^{\prime}$ is even}),\\
\omega^{2((-1)^aj+(-1)^{1-b}j^{\prime})N}e_a\otimes e_b^{\prime} & (\textrm{if $k^{\prime}$ is odd}), 
\end{cases} \\ 
h\cdot e_a\otimes e_b^{\prime}&=\omega ^{2(k+k^{\prime})n}e_a\otimes e_b^{\prime}
\end{align*}
for $a, b=1,2$. 
Therefore, we see that $[\rho _{jk}\otimes \rho _{j'k'}]=[\rho _{j+j', k+k'}\oplus \rho _{j-j',k+k'}]$ by considering the matrix presentation of  $\rho _{jk}\otimes \rho _{j'k'}$ with respect to the basis $\{ e_1\otimes e_1^{\prime}, e_2\otimes e_2^{\prime}, e_1\otimes e_2^{\prime}, e_2\otimes e_1^{\prime}\} $ or $\{ e_2\otimes e_1^{\prime}, \omega^{2(jn-j)N}e_1\otimes e_2^{\prime}, e_2\otimes e_2^{\prime}, \omega ^{2(jn+j)N}e_1\otimes e_1^{\prime}\}$ according to the case whether $k^{\prime}$ is even or odd. 
\par 
(vi) Since 
$\chi _{ijk}^{\ast }(t)=(-1)^i, \ 
\chi _{ijk}^{\ast }(w)=(-1)^{-j},\ \chi _{ijk}^{\ast}(h)=\omega ^{-2(k+\epsilon (n)j)n}$, we have 
$\chi _{ijk}^{\ast }=\chi _{i,-j,-k}$. 
\par \smallskip 
(vii) Let $\{ e_1,e_2\} $ be the standard basis of $V_{jk}=\boldsymbol{k}\oplus \boldsymbol{k}$. 
Then with respect to the dual basis $\{ e_1^{\ast}, e_2^{\ast}\} $ of $\{ e_1,e_2\} $, the contragredient representation $\rho _{jk}^{\ast}$ is represented as follows: 
$\rho _{jk}^{\ast }(h)
=\begin{pmatrix} \omega ^{-2kn} & 0 \\ 0 & \omega ^{-2kn} \end{pmatrix}$, and 
\par \smallskip \noindent 
if $k$ is even, then 
$\rho _{jk}^{\ast}(t)=\begin{pmatrix} 0 & 1 \\ 1 & 0 \end{pmatrix},\ 
\rho _{jk}^{\ast}(w)=\begin{pmatrix} \omega ^{-2jN} & 0 \\ 0 & \omega^{2jN} \end{pmatrix}$, and 
\par \smallskip \noindent 
if $k$ is odd, then 
$\rho _{jk}^{\ast}(t)
=\begin{pmatrix} 0 & -\omega^{-2jN} \\ -\omega ^{2jN} & 0 \end{pmatrix},\ 
\rho _{jk}^{\ast}(w)
=\begin{pmatrix} \omega^{2jN} & 0 \\ 0 & \omega ^{-2jN} \end{pmatrix}$. 
Thus in the case where $k$ is even, by considering the matrix presentation of $\rho _{jk}^{\ast }$ with respect to the basis $\{ e_2^{\ast}, e_1^{\ast}\} $, we see that $[\rho _{jk}^{\ast}]=[\rho _{j,-k}]$. 
In the case where $k$ is odd, by considering the matrix presentation of $\rho _{jk}^{\ast }$ with respect to the basis $\{ e_1^{\ast}, -\omega ^{2jN}e_2^{\ast}\} $, we have the same result $[\rho _{jk}^{\ast}]=[\rho _{j,-k}]$. 
\end{proof}

\par \smallskip 
From the above lemma, we see that the representation ring of $A_{Nn}$ is described as in the following proposition. 
In the case of $N=1$, this result has already proved by \cite[Proposition 3.9]{Mas2}. 
\par \smallskip 
\begin{prop}\label{5.23}
Let $\boldsymbol{k}$ be a field whose characteristic dose not divide $2nN$,  and suppose that there is a primitive $4nN$-th root of unity in $\boldsymbol{k}$. 
\par 
(1) If $n$ is even, then the representation rings of $A_{Nn}$ and $\boldsymbol{k}[G_{Nn}]$ are isomorphic as rings with $\ast$-structure, 
and both of them are isomorphic to the commutative ring generated by $a,b,c,x_1,\ldots, x_{n-1}$ subject to the commutativity relations and the following relations: 
\begin{align}
&a^2=b^2=c^{N}=1,\label{eq6.14}\\ 
&ax_i=x_{i} \ \ (i=1,2,\ldots ,n-1), \label{eq6.16}\\ 
&bx_i=x_{n-i} \ \ (i=1,2,\ldots ,n-1), \label{eq6.15}\\ 
&x_ix_j=c^{\frac{1-(-1)^{ij}}{2}}(x_{|i-j|}+x_{i+j}) \ \ (i,j=1,2,\ldots ,n-1), \label{eq6.17}
\end{align}
where indices of $x$ in R.H.S. of (\ref{eq6.17}) are treated under the rules
$$x_0=1+a, \quad x_{n}=b(1+a), \quad x_{n+i}=x_{n-i} \ \ (i=1,2,\ldots ,n-1).$$
The $\ast$-structure is given by 
$$a^{\ast}=a,\quad b^{\ast}=b,\quad c^{\ast} =c^{-1},\quad x_i^{\ast}=c^{\frac{(-1)^i-1}{2}}x_i  \ \ (i=1,\ldots ,n-1).$$
\par 
(2) If $n$ is odd, then then the representation rings of $A_{Nn}$ and $\boldsymbol{k}[G_{Nn}]$ are isomorphic as rings with $\ast$-structure, and both of them are isomorphic to the commutative ring generated by $a,b,x_1,\ldots, x_{n-1}$ subject to the commutativity relations and the following relations: 
\begin{align}
&a^2=b^{2N}=1, \label{eq6.18}\\ 
&ax_i=x_{i} \ \ (i=1,\ldots ,n-1), \label{eq6.19}\\  
&bx_i=x_{n-i} \ \ (i=2,4,\ldots ,n-1), \label{eq6.20}\\  
&x_ix_j=b^{1-(-1)^{ij}}(x_{|i-j|}+x_{i+j}) \ \ (i,j=1,\ldots ,n-1) \label{eq6.21}
\end{align}
\noindent 
where indices of $x$ in R.H.S. of (\ref{eq6.21}) are treated under the rules 
$$x_0=1+a,\ x_n=b(1+a),\ x_{n+i}=x_{n-i} \ \ (i=1,2,\ldots ,n-1).$$
The $\ast$-structure is given by 
$$a^{\ast}=a,\ b^{\ast} =b^{-1},\ x_i^{\ast}=b^{(-1)^i-1}x_i  \ \ (i=1,\ldots ,n-1).$$ 
\end{prop}
\begin{proof}
Since the same results in Lemma~\ref{6.4} hold for the group Hopf algebra $\boldsymbol{k}[G_{Nn}]$, it is sufficient to prove that the representation ring of $A_{Nn}$ is isomorphic to the commutative ring $\mathcal{R}$ which is presented by the given generators and relations described in the proposition. 
By Proposition~\ref{5.24} we may assume that $A_{Nn}=\boldsymbol{k}[G_{Nn}]$ as algebras. 
\par 
(1) We will prove that the representation  ring $\textrm{Rep}(A_{Nn})$ is the commutative ring generated by $a=[\chi _{100}],\ b=[\chi _{010}],\ c=[\chi _{002}]$, $x_i=[\rho _{i\epsilon (i)}]\ (i=1,\ldots ,n-1)$ with relations (\ref{eq6.14}), (\ref{eq6.15}), (\ref{eq6.16}) and (\ref{eq6.17}), 
where $\chi_{ijk}$ and $\rho _{jk}$ are the one-dimensional and two-dimensional representations of the algebra $A_{Nn}=\boldsymbol{k}[G_{Nn}]$ given by (\ref{eq5.21}) and (\ref{eq5.22}), respectively, and $\epsilon (i)$ is equal to $0$ or $1$ according to whether $i$ is even or odd, respectively. 
\par 
Since $[\chi _{000}]=[\varepsilon ]=1$, by Lemma~\ref{6.4}(iii) it follows that  $a^2=b^2=c^N=1$. 
By Parts (i) and (iv) of Lemma~\ref{6.4}, we have 
$$ax_i=[\chi _{100}][\rho _{i\epsilon (i)}]=[\rho _{i,\epsilon (i)}]=x_i,$$
and 
$$bx_i=[\chi _{010}][\rho _{i\epsilon (i)}]=[\rho _{n+i,\epsilon (i)}]=[\rho _{n-i,\epsilon (i)}]=x_{n-i}.$$
Furthermore, for integers $i,j\in \mathbb{Z}$, by Parts (iv) and (v) of Lemma~\ref{6.4} we have 
\begin{align*}
x_ix_j
&=[\rho _{i-j,\epsilon (i)+\epsilon (j)}]+[\rho _{i+j,\epsilon (i)+\epsilon (j)}] \\ 
&=\begin{cases} 
x_{i-j}+x_{i+j}=x_{|i-j|}+x_{i+j} \ & \textrm{($i$ or $j$ is even),} \\ 
[\rho _{i-j,2}]+[\rho _{i+j,2}]=[\chi _{002}]([\rho _{i-j,0}]+[\rho _{i+j,0}])\ & \textrm{($i$ and $j$ are odd)}, 
\end{cases} \\ 
&=c^{\frac{1-(-1)^{ij}}{2}}(x_{|i-j|}+x_{i+j}). 
\end{align*}
By Lemma~\ref{6.4}(ii) we have 
$$x_0=[\chi _{000}]+[\chi _{100}]=1+a,\ \ 
x_n=[\chi _{010}]+[\chi _{110}]=[\chi _{010}]+[\chi _{010}][\chi _{100}]=b(1+a),$$
and by Lemma~\ref{6.4}(i) we have 
$$x_{n+i}=[\rho _{n+i, \epsilon (n+i)}]=[\rho _{n-i, \epsilon (n+i)}]=[\rho _{n-i, \epsilon (n-i)}]=x_{n-i}.$$
From the results argued so far, we see that there is a ring homomorphism $f: \mathcal{R}\longrightarrow \textrm{Rep}(A_{Nn})$ such that 
$$f(a)=[\chi_{100}],\quad f(b)=[\chi_{010}],\quad f(c)=[\chi_{002}],\quad f(x_i)=[\rho_{i\epsilon (i)}]\quad (i=1,2,\ldots ,n-1).$$
To  show that $f$ is bijective, let us construct the inverse map of $f$. 
Let $g: \textrm{Rep}(A_{Nn})\longrightarrow \mathcal{R}$ be the $\mathbb{Z}$-linear map defined by 
\begin{align*}
g([\chi _{ijk}])&=a^ib^jc^{\frac{k}{2}}\quad (i,j=0,1,\ k=0,2,\ldots ,2N-2), \\ 
g([\rho _{jk}])&=c^{\frac{k-\epsilon (j)}{2}}x_j \quad (k=0,1,\ldots ,2N-1,\ j=1,2,\ldots ,n-1,\ j\equiv k \ (\textrm{mod}\ 2) ). 
\end{align*}
Then it is easily shown that $g$ is the inverse map of $f$, and this concludes that $f : \mathcal{R}\longrightarrow \textrm{Rep}(A_{Nn})$ is a ring isomorphism. 
\par 
The $\ast$-structure in $\textrm{Rep}(A_{Nn})$ is determined as follows.  
By Lemma~\ref{6.4}(vi) we see immediately that $a^{\ast }=a,\ b^{\ast}=b,\ c^{\ast}=[\chi _{00,-2}]=[\chi _{002}]^{N-1}=c^{N-1}=c^{-1}$. 
By Lemma~\ref{6.4}(vii) if $i$ is even, then $x_i^{\ast}=[\rho _{i0}]^{\ast}=[\rho _{i0}]=x_i$, and if $i$ is odd, then $x_i^{\ast}=[\rho _{i1}]^{\ast}=[\rho _{i,-1}]=[\rho _{i,2N-1}]=[\chi _{00,2N-2}\otimes \rho _{i1}]=[\chi _{002}]^{N-1}[\rho _{i1}]=c^{-1}x_i$. 
Therefore, we obtain $x_i^{\ast }=c^{\frac{(-1)^i-1}{2}}x_i$ for 
$i=1,2\ldots ,n-1$. 
\par 
(2) In the same manner as above it can be verified that there is a ring isomorphism 
$f: \mathcal{R}\longrightarrow \textrm{Rep}(A_{Nn})$ such that $f(a)=[\chi _{100}],\ f(b)=[\chi _{010}]$, $f(x_i)=[\rho _{i\epsilon (i)}]\ (i=1,\ldots ,n-1)$ preserving $\ast$-structures. 
Here, what we should take account is the following fact. 
If we set $b=[\chi _{010}]$, then $[\chi _{002}]=b^2$, and 
since $n$ is odd, it follows from Parts (i) and (iv) of Lemma~\ref{6.4} that 
$$
bx_i
=[\chi _{010}][\rho _{i\epsilon (i)}]
=[\rho _{n+i,1+\epsilon (i)}]
=[\rho _{n-i,1+\epsilon (i)}] \\ 
=\begin{cases}
[\chi _{002}\otimes \rho _{n-i,0}]=b^2x_{n-i} & (\textrm{$i$ is odd}), \\ 
[\chi _{000}\otimes \rho _{n-i,1}]=x_{n-i} & (\textrm{$i$ is even}).
\end{cases}
$$
This is equivalent to $bx_i=x_{n-i}$ for all even integers $i$. 
\end{proof}

\par \smallskip 
\begin{rem}\rm 
By Proposition~\ref{5.24} in the case where $(\lambda , n)=(+, \textrm{even})$ or $(\lambda , n)=(-, \textrm{odd})$, 
as an algebra  
$A_{Nn}^{+\lambda }$ is isomorphic to the group algebra $\boldsymbol{k}[G_{Nn}^{\prime}]$, where 
$$G_{Nn}^{\prime}=\langle h,\ t,\ w\ \vert \ t^2=h^{2N}=1,\ w^n=1,\ tw=w^{-1}t,\ ht=th,\ hw=wh \rangle .$$
As a similar manner of the proofs of Lemma~\ref{6.4} and Proposition~\ref{5.23},  one can determine the structure of $\textrm{Rep}(A_{Nn}^{++})$ in the case where $n$ is even, and the structure of $\textrm{Rep}(A_{Nn}^{+-})$ in the case where $n$ is odd. 
As a result, we see that in the case where $n$ is even the representation ring $\textrm{Rep}(A_{Nn}^{++})$ is not commutative (see Lemma~\ref{6.11} in the next subsection 6.3), whereas the representation ring $\textrm{Rep}(\boldsymbol{k}[G_{Nn}^{\prime}])$ is commutative since $\boldsymbol{k}[G_{Nn}^{\prime}]$ is cocommutative. 
Therefore, two representation rings of $A_{Nn}^{++}$ and $\boldsymbol{k}[G_{Nn}^{\prime}]$ are not isomorphic. 
On the contrary, in the case where $n$ is odd, 
we see that $\textrm{Rep}(A_{Nn}^{+-})\otimes _{\mathbb{Z}}\boldsymbol{k}$ and $\textrm{Rep}(\boldsymbol{k}[G_{Nn}^{\prime}])\otimes _{\mathbb{Z}}\boldsymbol{k}$ are isomorphic as algebras with $\ast$-structure over $\boldsymbol{k}$. 
\end{rem} 

\par \smallskip 
\subsection{Self-duality of $A_{Nn}^{+\lambda}$} 
Based on Corollary~\ref{5.26}, we identify $A_{Nn}^{+\lambda }$ with $\boldsymbol{k}[G]$ such as 
\begin{align*}
h&=x_{11}^2-x_{12}^2, \\  
t&=x_{12}^N+x_{22}^N=x_{11}^{N-1}x_{22}+x_{12}^N, \\ 
w&=x_{11}^{2N-1}x_{22}-x_{21}^{2N-1}x_{12}=x_{11}^{2N-1}\chi _{22}-x_{12}^{2N-2}\chi _{21}^2. 
\end{align*}

Then 
$$w^{-1}=x_{22}x_{11}^{2N-1}-x_{12}x_{21}^{2N-1}=x_{11}^{2N-2}\chi _{22}^2-x_{12}^{2N-1}\chi _{21}.$$  
\indent 
Let us determine the values of 
$\sigma _{\alpha \beta }$ on the elements of the basis $\{ \ h^iw^kt^p\ \vert \ 0\leq i\leq 2N-1,\ 0\leq k\leq n-1,\ p=0,1\ \} $ of $A_{Nn}^{+\lambda}$. 

\par \smallskip 
\begin{lem}\label{6.5}
Suppose that $\alpha ,\beta \in \boldsymbol{k}$ satisfy $(\alpha \beta )^N=1$ and $(\alpha \beta ^{-1})^n=\lambda $, and consider the braiding $\sigma _{\alpha \beta }$ of $A_{Nn}^{+\lambda }$ defined in Theorem~\ref{5.5}. 
We set 
$\xi :=\alpha \beta $ and $\eta :=\alpha \beta ^{-1}$. 
Then for $i,j,k,l\geq 0$ the following equations hold: 
\begin{align}
\sigma _{\alpha \beta }(h^iw^k, h^jw^l)&=\xi ^{2ij}\eta ^{-2kl}, \label{eq6.22}\\ 
\sigma _{\alpha \beta }(h^iw^k, h^jw^lt)&=(-1)^{i+k}\xi ^{2ij}\eta ^{-k(2l-1)}, \label{eq6.23}\\ 
\sigma _{\alpha \beta }(h^iw^kt, h^jw^l)&=(-1)^{j+l}\xi ^{2ij}\eta ^{(2k-1)l}, \label{eq6.24}\\ 
\sigma _{\alpha \beta }(h^iw^kt, h^jw^lt)&=(-1)^{i+j+k+l}\xi ^{2ij+\frac{N^2-1}{2}}\eta ^{2kl-k-l}\alpha . \label{eq6.25}
\end{align}
\end{lem}
\begin{proof} 
First, we show $\sigma _{\alpha \beta }(h^iw^k, h^jw^l)=\xi ^{2ij}\eta ^{-2kl}$.  
In the case of $k=0$ the above equation is written as $\sigma _{\alpha \beta }(h^i, h^jw^l)=(\alpha \beta )^{2ij}$. 
This can be proved by induction on $l$ as follows. 
By use of the equation $h^j=x_{11}^{2j}+(-1)^jx_{12}^{2j}\ (j\geq 0)$, we see that 
\begin{align*}
\sigma _{\alpha \beta }(h^i, h^j)
&=\sigma _{\alpha \beta }(x_{11}^{2i}, x_{11}^{2j})+(-1)^j\sigma _{\alpha \beta }(x_{11}^{2i}, x_{12}^{2j}) \\ 
&\quad +(-1)^i\sigma _{\alpha \beta }(x_{12}^{2i}, x_{11}^{2j})+(-1)^{i+j}\sigma _{\alpha \beta }(x_{12}^{2i}, x_{12}^{2j}) \\ 
&=\sigma _{\alpha \beta }(x_{11}^{2i}, x_{11}^{2j})=(\alpha \beta )^{2ij}. 
\end{align*}
\noindent 
Suppose that $\sigma _{\alpha \beta }(h^i, h^jw^l)=(\alpha \beta )^{2ij}$ for some $l\geq 0$. 
Since 
\begin{align*}
\sigma _{\alpha \beta }(h^i, w)
&=\sigma _{\alpha \beta }(x_{11}^{2i},\ x_{11}^{2N-1}\chi _{22})
-\sigma _{\alpha \beta }(x_{11}^{2i},\ x_{12}^{2N-2}\chi _{21}^2) \\ 
&\quad +(-1)^i\sigma _{\alpha \beta }(x_{12}^{2i},\ x_{11}^{2N-1}\chi _{22})-(-1)^i\sigma _{\alpha \beta }(x_{12}^{2i},\ x_{12}^{2N-2}\chi _{21}^2) \\ 
&=\sigma _{\alpha \beta }(x_{11}^{2i},\ x_{11}^{2N-1}\chi _{22}) 
=(\alpha \beta )^{2iN} =1,
\end{align*}
it follows that 
$$\sigma _{\alpha \beta }(h^i, h^jw^{l+1})
=\sigma _{\alpha \beta }(h^i, w)\sigma _{\alpha \beta }(h^i, h^jw^l)=\sigma _{\alpha \beta }(h^i, h^jw^l)=(\alpha \beta )^{2ij}. $$
So, it is true for $l+1$. 
\par 
Now, let $k\geq 1$, and suppose that $\sigma _{\alpha \beta }(h^iw^k, h^jw^l)=(\alpha \beta )^{2ij}(\alpha ^{-1}\beta )^{2kl}$ for all $l\geq 0$. 
Then we have 
\begin{align*}
\sigma _{\alpha \beta }(h^iw^{k+1}, h^jw^l)
&=\sigma _{\alpha \beta }(h^iw^{k}, h^jw^l)\sigma _{\alpha \beta }(w, h^je_0w^l)
+\sigma _{\alpha \beta }(h^iw^{k}, h^jw^{-l})\sigma _{\alpha \beta }(w, h^je_1w^l) \\ 
&=\sigma _{\alpha \beta }(h^iw^{k}, h^jw^l)\sigma _{\alpha \beta }(w, h^jw^l) \\ 
&=(\alpha \beta )^{2ij}(\alpha ^{-1}\beta )^{2kl}\times (\alpha ^{-1}\beta )^{2l} \\ 
&=(\alpha \beta )^{2ij}(\alpha ^{-1}\beta )^{2(k+1)l}. 
\end{align*}
Thus it is true for $k+1$, and the equation (\ref{eq6.22}) is proved by induction on $k$. 
\par 
Second, we show (\ref{eq6.23}).  
By use of (\ref{eq6.22}) we have 
\begin{align*}
\sigma _{\alpha \beta }(h^iw^k, h^jw^lt)
&=\sigma _{\alpha \beta }(h^iw^k, t)\sigma _{\alpha \beta }(h^ie_0w^k, h^jw^l)+
\sigma _{\alpha \beta }(h^iw^{-k}, t)\sigma _{\alpha \beta }(h^ie_1w^k, h^jw^l) \\ 
&=\sigma _{\alpha \beta }(h^iw^k, t)(\alpha \beta )^{2ij}(\alpha ^{-1}\beta )^{2kl}. 
\end{align*}
By induction on $k$ we can show that 
$\sigma _{\alpha \beta }(h^iw^k, t)=(-\alpha \beta ^{-1})^k(-1)^i$, and thus the equation (\ref{eq6.23}) holds. 
\par 
As a similar manner we can show that 
$\sigma _{\alpha \beta }(h^iw^k, h^jw^{-l})=(\alpha \beta )^{2ij}(\alpha \beta ^{-1})^{2kl},\ \sigma _{\alpha \beta }(t, h^jw^{l})$ $=(-\alpha ^{-1}\beta )^l(-1)^j$, and the equation (\ref{eq6.24}) holds. 
\par 
Finally, we show that (\ref{eq6.25}). 
For $a,b\in \{ 0,1\}$ we have 
\begin{align*}
\sigma _{\alpha \beta }(e_at, h^je_bw^l)
&=\frac{1}{4}\Bigl( \sigma _{\alpha \beta }(t, h^jw^l)+(-1)^b\sigma _{\alpha \beta }(t, h^{j+N}w^l) \\ 
&\quad +(-1)^a\sigma _{\alpha \beta }(h^Nt, h^jw^l)+(-1)^{a+b}\sigma _{\alpha \beta }(h^Nt, h^{j+N}w^l)\Bigr) \\ 
&=\frac{1+(-1)^{b+1}+(-1)^a+(-1)^{a+b+1}}{4}(-1)^j(-\alpha ^{-1}\beta )^l \\ 
&=\begin{cases}
(-1)^j(-\alpha ^{-1}\beta )^l &\quad \textrm{if $a$ is even, and $b$ is odd}, \\ 
0 &\quad \textrm{otherwise}.
\end{cases}
\end{align*}
Since $
\sigma _{\alpha \beta }(t,t)
=\sigma _{\alpha \beta }(x_{11}^{N-1}\chi _{22}+x_{12}^N,\  x_{11}^{N-1}\chi _{22}+x_{12}^N) 
=\sigma _{\alpha \beta }(x_{12}^N,x_{12}^N) 
=(\alpha \beta )^{\frac{N^2-1}{2}}\alpha $, 
we have 
\begin{align*}
\sigma _{\alpha \beta }(t, h^je_aw^lt)
&=\sigma _{\alpha \beta }(h^Nwt, t)\sigma _{\alpha \beta }(e_1t, h^je_aw^l)
+\sigma _{\alpha \beta }(t,t)\sigma _{\alpha \beta }(e_0t, h^je_aw^l) \\ 
&=\begin{cases}
0 &\quad \textrm{($a=0$)}, \\ 
(\alpha \beta )^{\frac{N^2-1}{2}}\alpha (-1)^j(-\alpha ^{-1}\beta )^l & \quad \textrm{($a=1$)}. 
\end{cases}
\end{align*}
So, by using the equation 
$\sigma _{\alpha \beta }(h^iw^k, h^jw^{-l}t)=(-1)^{i+k}(\alpha \beta )^{2ij}(\alpha \beta ^{-1})^{k(2l+1)}$, we have 
\begin{align*}
\sigma _{\alpha \beta }(h^iw^kt, h^jw^lt)
&=\sigma _{\alpha \beta }(h^iw^k, h^jw^lt)\sigma _{\alpha \beta }(t, h^je_0w^lt) \\ 
&\quad +\sigma _{\alpha \beta }(h^iw^k, h^{j+N}w^{-l+1}t)\sigma _{\alpha \beta }(t, h^je_1w^lt) \\ 
&=\sigma _{\alpha \beta }(h^iw^k, h^{j+N}w^{-l+1}t)\sigma _{\alpha \beta }(t, t)(-1)^j(-\alpha ^{-1}\beta )^l  \\ 
&=(-1)^{i+k}(\alpha \beta )^{2i(j+N)}(\alpha \beta ^{-1})^{k(2(l-1)+1)}\cdot  (\alpha \beta )^{\frac{N^2-1}{2}}\alpha (-1)^j(-\alpha ^{-1}\beta )^l \\ 
&=(-1)^{i+j+k+l}(\alpha \beta )^{2ij+\frac{N^2-1}{2}}(\alpha \beta ^{-1})^{2kl-k-l} \alpha . 
 \end{align*}
\end{proof}

\par \smallskip 
From here to the end of the paper, we suppose that $\boldsymbol{k}$ is a field whose characteristic does not divide $2nN$, and it contains a primitive $4nN$-th root of unity. 

\par \smallskip 
\begin{thm}\label{6.6}
Let $\alpha ,\beta $ be elements in $\boldsymbol{k}$ satisfying $(\alpha \beta )^N=1$ and $(\alpha \beta ^{-1})^n=\lambda $. 
Suppose that $\lambda =-1$ or $(\lambda , n)=(1, \textrm{odd})$. 
Then the braiding $\sigma _{\alpha \beta }$ of $A_{Nn}^{+\lambda }$ is non-degenerate if and only if 
\begin{enumerate}
\item[(i)] $\alpha \beta $ is a primitive $N$-th root of unity, and
\item[(ii)] $\alpha \beta ^{-1}$ is a primitive $n$-th root of $\lambda $.
\end{enumerate}
\end{thm}
\begin{proof}
To show \lq\lq if" part, we show the contraposition.  
\par 
Suppose that $\alpha \beta $ is not a primitive $N$-th root of unity. 
Then $N\geq 3$ is required since if $N=1$, then $\alpha \beta =(\alpha \beta )^N=1$. 
Let $\xi $ be a primitive $N$-th root of unity. 
Then $\alpha \beta $ is represented by $\alpha \beta =\xi ^m$ for some divisor $m\ (\not= 1)$ of $N$. 
Hence by setting $m^{\prime}:=N/m\ (<N)$, 
we have
\begin{align*}
\sigma _{\alpha \beta }(h^{2m^{\prime}},\ h^jw^l)
&=(\alpha \beta )^{4m^{\prime}j} 
=\xi ^{4mm^{\prime}j}
=\xi ^{4Nj} =1=\sigma _{\alpha \beta }(1, \ h^jw^l),\\ 
\sigma _{\alpha \beta }(h^{2m^{\prime}},\ h^jw^lt)
&=(-1)^{2m^{\prime}}
(\alpha \beta )^{4m^{\prime}j} 
=1=\sigma _{\alpha \beta }(1, \ h^jw^lt). 
\end{align*}
 Thus $\sigma _{\alpha \beta }(1-h^{2m^{\prime}},\ a)=0$ for all $a\in A_{Nn}^{+\lambda}$. 
Since $1<2m^{\prime}<2N$, we see that 
$1-h^{2m^{\prime}}\not= 0$. 
Therefore, $\sigma _{\alpha \beta }$ degenerates as a bilinear form on $A_{Nn}^{+\lambda}$. 
\par 
Next, suppose that 
$\alpha \beta ^{-1}$ is not a primitive $n$-th root of $\lambda$. 
Then there is an $r\in \mathbb{N}$ such that 
$1\leq r<n$ and $(\alpha \beta ^{-1})^{r}=\lambda $. 
So, $(\alpha \beta ^{-1})^{n-r}=1$. 
By Lemma~\ref{6.5} we have 
\begin{align*}
\sigma _{\alpha \beta }(w^{n-r},\ h^jw^l)
&=(\alpha \beta ^{-1})^{-2(n-r)l} 
=1,\\ 
\sigma _{\alpha \beta }(w^{n-r},\ h^jw^lt)
&=(-1)^{n-r}
(\alpha \beta ^{-1})^{-(n-r)(2l-1)} 
=(-1)^{n-r}. 
\end{align*}
\par 
If $n-r$ is even, then 
\begin{align*}
\sigma _{\alpha \beta }(w^{n-r},\ h^jw^l)&=1
=\sigma _{\alpha \beta }(1,\ h^jw^l),\\ 
\sigma _{\alpha \beta }(w^{n-r},\ h^jw^lt)&=1
=\sigma _{\alpha \beta }(1,\ h^jw^lt). 
\end{align*}
Thus 
$\sigma _{\alpha \beta }(1-w^{n-r},\ a)=0$ for all $a\in A_{Nn}^{+\lambda}$. 
It follows from $0<n-r\leq n-1$ that $1-w^{n-r}\not= 0$, and hence $\sigma _{\alpha \beta }$ degenerates. 
\par 
If $n-r$ is odd, then 
\begin{align*}
\sigma _{\alpha \beta }(w^{n-r},\ h^jw^l)&=1
=\sigma _{\alpha \beta }(h^N,\ h^jw^l),\\ 
\sigma _{\alpha \beta }(w^{n-r},\ h^jw^lt)&=-1
=\sigma _{\alpha \beta }(h^N,\ h^jw^lt).
\end{align*}
Thus 
$\sigma _{\alpha \beta }(h^N-w^{n-r},\ a)=0$ for all $a\in A_{Nn}^{+\lambda}$. 
Since $h^N-w^{n-r}\not= 0$, the braiding $\sigma _{\alpha \beta }$ also degenerates as a bilinear form. 
\par 
We will show \lq\lq only if" part. 
Let us consider the linear map $F:A_{Nn}^{+\lambda }\longrightarrow (A_{Nn}^{+\lambda })^{\ast}$ defined by 
$$F(a)=\sum\limits_{j=0}^{2N-1}\sum\limits_{l=0}^{n-1}\sigma _{\alpha \beta }(a, h^jw^l)(h^jw^l)^{\ast}+\sum\limits_{j=0}^{2N-1}\sum\limits_{l=0}^{n-1}\sigma _{\alpha \beta }(a, h^jw^lt)(h^jw^lt)^{\ast} \qquad (a\in A_{Nn}^{+\lambda }).$$
Here, $\{ \ (h^jw^lt^p)^{\ast}\ \vert \ 0\leq j\leq 2N-1,\ 0\leq l\leq n,\ p=0,1\ \} $ stands for the dual basis of the basis $\{ \ h^jw^lt^p\ \vert \ 0\leq j\leq 2N-1,\ 0\leq l\leq n,\ p=0,1\ \} $ of $A_{Nn}^{+\lambda }$. 
Setting $\xi =\alpha \beta ,\ \eta =\alpha \beta ^{-1}$, we have 
\begin{align*}
F(h^iw^k)
&=\sum\limits_{j=0}^{2N-1}\sum\limits_{l=0}^{n-1}\xi ^{2ij}\eta ^{-2kl}\bigl( (h^jw^l)^{\ast}+(-1)^{i+k}\eta ^k(h^jw^lt)^{\ast}\bigr) , \\ 
F(h^iw^kt)
&=\sum\limits_{j=0}^{2N-1}\sum\limits_{l=0}^{n-1}(-1)^{j+l}\xi ^{2ij}\eta ^{2kl-l}\bigl( (h^jw^l)^{\ast}+(-1)^{i+k}\xi ^{\frac{N^2-1}{2}}\eta ^{-k}\alpha (h^jw^lt)^{\ast}\bigr) . 
\end{align*}

In what follows, let $\xi $ be a primitive $N$-th root of unity, and $\eta $ is a primitive $n$-th root of $\lambda $. 
\par 
We consider the case of $N\geq 3$. 
Since $N$ is odd, the following claims hold for $j\in \mathbb{Z}$: 
\begin{enumerate}
\item[$\bullet$] $\xi ^{2j}=1 \ \ \Longleftrightarrow \ \ j\equiv 0\ (\textrm{mod}\ N)$, 
\item[$\bullet$] $\xi ^{2j}\not= -1$. 
\end{enumerate}
Therefore, for $j,j'\in \mathbb{Z}$, we have 
$$\sum\limits_{i=0}^{2N-1}\xi ^{2i(j-j')}
=\begin{cases}
2N &\quad \textrm{if $j\equiv j'\ (\textrm{mod}\ N$)},\\ 
0 &\quad \textrm{otherwise}, 
\end{cases} \qquad  
\sum\limits_{i=0}^{2N-1}(-1)^i\xi ^{2i(j-j')}
=0. 
$$
Thus we have 
\begin{align*}
\sum\limits_{i=0}^{2N-1}\xi ^{-2ij'}F(h^iw^k)
&=2N\sum\limits_{l=0}^{n-1}\eta ^{-2kl}\bigl( (h^{j'}w^l)^{\ast}+(h^{j'+N}w^l)^{\ast}\bigr) , \\ 
\sum\limits_{i=0}^{2N-1}\xi ^{-2ij'}F(h^iw^kt)
&=2N\sum\limits_{l=0}^{n-1}(-1)^{j'+l}\eta ^{2kl-l}\bigl( (h^{j'}w^l)^{\ast}-(h^{j'+N}w^l)^{\ast}\bigr) .
\end{align*}

Since $\eta ^{2(l-l')}=1$ if and only if $l-l' \equiv 0\ (\textrm{mod}\ n)$ under the condition $\lambda =-1$ or $(\lambda , n)=(1, \textrm{odd})$, 
it follows that for $l'\in \mathbb{Z}$
\begin{equation}\label{eq:6.6}
\sum\limits_{k=0}^{n-1}\eta ^{2k(l-l')}
=\begin{cases}
0 &\quad \textrm{otherwise}, \\ 
n &\quad \textrm{if $l\equiv l'\ (\textrm{mod}\ n$)}. 
\end{cases}
\end{equation}
This implies that 
\begin{align*}
\sum\limits_{i=0}^{2N-1}\sum\limits_{k=0}^{n-1}\eta ^{2kl'}\xi ^{-2ij'}F(h^iw^k)
&=2nN\bigl( (h^{j'}w^{l'})^{\ast}+(h^{j'+N}w^{l'})^{\ast}\bigr) , \\ 
\sum\limits_{i=0}^{2N-1}\sum\limits_{k=0}^{n-1}\eta ^{-2kl'}\xi ^{-2ij'}F(h^iw^kt)
&=2nN(-1)^{j'+l'}\eta ^{-l'}\bigl( (h^{j'}w^{l'})^{\ast}-(h^{j'+N}w^{l'})^{\ast}\bigr) . 
\end{align*}
From these equations, we have 
\begin{align*}
(h^{j'}w^{l'})^{\ast}&=\frac{1}{4nN}\Bigl( \sum\limits_{i=0}^{2N-1}\sum\limits_{k=0}^{n-1}\eta ^{2kl'}\xi ^{-2ij'}F(h^iw^k)+(-1)^{j'+l'}\eta ^{l'}\sum\limits_{i=0}^{2N-1}\sum\limits_{k=0}^{n-1}\eta ^{-2kl'}\xi ^{-2ij'}F(h^iw^kt)\Bigr) , \\ 
(h^{j'+N}w^{l'})^{\ast}&=\frac{1}{4nN}\Bigl( \sum\limits_{i=0}^{2N-1}\sum\limits_{k=0}^{n-1}\eta ^{2kl'}\xi ^{-2ij'}F(h^iw^k)-(-1)^{j'+l'}\eta ^{l'}\sum\limits_{i=0}^{2N-1}\sum\limits_{k=0}^{n-1}\eta ^{-2kl'}\xi ^{-2ij'}F(h^iw^kt)\Bigr) . 
\end{align*}

As a similar manner, we have 
\begin{align*}
\sum\limits_{i=0}^{2N-1}\sum\limits_{k=0}^{n-1}(-1)^{i+k}\eta ^{2kl'-k}\xi ^{-2ij'}F(h^iw^k)
&=2nN\bigl( (h^{j'}w^{l'}t)^{\ast}+(h^{j'+N}w^{l'}t)^{\ast}\bigr) ,\\ 
\sum\limits_{i=0}^{2N-1}\sum\limits_{k=0}^{n-1}(-1)^{i+k}\eta ^{-2kl'+k}\xi ^{-2ij'}F(h^iw^kt)
&=2nN(-1)^{j'+l'}\xi ^{\frac{N^2-1}{2}}\eta ^{-l'}\alpha \bigl( (h^{j'}w^{l'}t)^{\ast}-(h^{j'+N}w^{l'}t)^{\ast}\bigr) , 
\end{align*}
and 
\begin{align*}
(h^{j'}w^{l'}t)^{\ast}
=&\frac{1}{4nN}\Bigl( \sum\limits_{i=0}^{2N-1}\sum\limits_{k=0}^{n-1}(-1)^{i+k}\eta ^{2kl'-k}\xi ^{-2ij'}F(h^iw^k) \\ 
&\qquad +(-1)^{j'+l'}\xi ^{-\frac{N^2-1}{2}}\eta ^{l'}\alpha ^{-1}\sum\limits_{i=0}^{2N-1}\sum\limits_{k=0}^{n-1}(-1)^{i+k}\eta ^{-2kl'+k}\xi ^{-2ij'}F(h^iw^kt)\Bigr) ,\\ 
(h^{j'+N}w^{l'}t)^{\ast} =&\frac{1}{4nN}\Bigl( \sum\limits_{i=0}^{2N-1}\sum\limits_{k=0}^{n-1}(-1)^{i+k}\eta ^{2kl'-k}\xi ^{-2ij'}F(h^iw^k) \\ 
&\qquad -(-1)^{j'+l'}\xi ^{-\frac{N^2-1}{2}}\eta ^{l'}\alpha ^{-1}\sum\limits_{i=0}^{2N-1}\sum\limits_{k=0}^{n-1}(-1)^{i+k}\eta ^{-2kl'+k}\xi ^{-2ij'}F(h^iw^kt)\Bigr) .
\end{align*}

Thus $F$ is surjective, and whence $F$ is an isomorphism. 
This implies that $\sigma _{\alpha \beta }$ is non-degenerate. 
\par 
We consider the case of $N=1$. 
Then $\xi =1$, and 
\begin{align*}
F(h^iw^k)
&=\sum\limits_{j=0}^{1}\sum\limits_{l=0}^{n-1}\eta ^{-2kl}\bigl( (h^jw^l)^{\ast}+(-1)^{i+k}\eta ^k(h^jw^lt)^{\ast}\bigr) ,\\ 
F(h^iw^kt)
&=\sum\limits_{j=0}^{1}\sum\limits_{l=0}^{n-1}(-1)^{j+l}\eta ^{2kl-l}\bigl( (h^jw^l)^{\ast}+(-1)^{i+k}\eta ^{-k}\alpha (h^jw^lt)^{\ast}\bigr) .
\end{align*}
As a similar manner as above, we see that 
$(w^{l'})^{\ast}, (hw^{l'})^{\ast}, (w^{l'}t)^{\ast}, (hw^{l'}t)^{\ast} $ can be represented by linear combinations of $\{ F(h^iw^k), F(h^iw^kt)\} $, and hence  $\sigma _{\alpha \beta }$ is non-degenerate. 
\end{proof}

\par \smallskip 
\begin{cor}\label{6.7}
Suppose that $\lambda =-1$, or $(\lambda , n)=(1, \textrm{odd})$. 
Then the Hopf algebra $A_{Nn}^{+\lambda }$ is self-dual. 
\end{cor}
\begin{proof}
\par 
In general, for a finite-dimensional Hopf algebra $A$, any braiding $\sigma :A\otimes A\longrightarrow \boldsymbol{k}$ gives rise to the Hopf pairing 
$\langle \ ,\ \rangle : A^{\textrm{cop}}\otimes A\longrightarrow \boldsymbol{k}$ defined by $\langle x , y \rangle =\sigma (x, y)$ for $x,y\in A$, 
and this pairing induces a Hopf algebra map $F: A\longrightarrow (A^{\textrm{cop}})^{\ast }$ defined by $(F(a))(b)=\sigma (a,b)$ for $a,b\in A$. 
Applying this fact to the Hopf algebra $A_{Nn}^{+\lambda }$ and the braiding $\sigma _{\alpha \beta }$, we have a Hopf algebra map $F:A_{Nn}^{+\lambda }\longrightarrow ((A_{Nn}^{+\lambda })^{\textrm{cop}})^{\ast}$. 
Furthermore, an algebra isomorphism $\phi :A_{Nn}^{+\lambda }\longrightarrow A_{Nn}^{+\lambda }$ can be defined by $\phi (x_{ij})=x_{ji}\ (i,j=1,2)$, and we see that it becomes a Hopf algebra isomorphism form $A_{Nn}^{+\lambda }$ to  $(A_{Nn}^{+\lambda })^{\textrm{cop}}$ \cite{Suz}. 
So, if $\sigma _{\alpha \beta }$ is non-degenerate, then the composition ${}^t\kern-0.2em \phi \circ F:A_{Nn}^{+\lambda }\longrightarrow (A_{Nn}^{+\lambda })^{\ast}$ gives a Hopf algebra isomorphism. 
\par 
To complete the proof, by Theorem~\ref{6.6}, it suffices to show that there are $\alpha , \beta $ such that $\alpha \beta $ is a primitive $N$-th root of unity, and $\alpha \beta ^{-1}$ is a primitive $n$-th root of $\lambda$. 
Let $\omega \in \boldsymbol{k}$ be a primitive $4nN$-th root of unity. 
In the case of $\lambda =-1$, we take $\alpha =\omega ^{N+2n},\ \beta =\omega ^{2n-N}$. 
Then $\alpha \beta ^{-1}=\omega ^{2N}$ is a primitive $n$-th root of $-1$, and 
$\alpha \beta =\omega ^{4n}$ is a primitive $N$-th root of unity. 
In the case where $\lambda =1$ and $n$ is odd, we take 
$\alpha =\omega ^{2N+2n},\ \beta =\omega ^{2n-2N}$. 
Then $\alpha \beta ^{-1}=\omega ^{4N}$ and $\alpha \beta =\omega ^{4n}$
are primitive $n$-th and  primitive $N$-th roots of unity, respectively. 
\end{proof}

To show that $A_{Nn}^{++}$ is not self-dual for any even integer $n$, we compare   the groups of group-like elements of $A_{Nn}^{++}$ and $(A_{Nn}^{++})^{\ast}$. 
The  structure of $G(A_{Nn}^{\nu \lambda })$ for all $N\geq 1, n\geq 2$ and $\lambda , \nu =\pm 1$ has already determined by S. Suzuki~\cite{Suz}.  
In the case where $N$ is odd, and $\nu =+$ the group $G(A_{Nn}^{\nu \lambda })$ is given as follows. 

\par \smallskip 
\begin{lem}[{\bf S.Suzuki}]\label{6.10}
$$G(A_{Nn}^{+\lambda })\cong \begin{cases}
C_2\times C_{2N}  \quad & \textrm{$(n$ is even, or $(n, \lambda )=(\textrm{odd}, 1))$}, \\  
C_{4N} \quad &\textrm{$(n$ is odd, and $\lambda =-1)$}.
\end{cases} 
$$  
\end{lem}

\par \smallskip 
On the contrary, we have: 

\par \smallskip 
\begin{lem}\label{6.11} 
The  structure of $G((A_{Nn}^{+\lambda })^{\ast})$ is given as follows. 
\begin{align*}
G((A_{Nn}^{++})^{\ast})&\cong \begin{cases}
SA_{8N} \quad & \textrm{($n$ is even}),\\ 
C_2\times C_{2N}
 \quad & \textrm{($n$ is odd}),  \end{cases} \\ 
G((A_{Nn}^{+-})^{\ast})&\cong \begin{cases}
C_2\times C_2\times C_N \quad & \textrm{($n$ is even}),\\ 
C_{4N} \quad & \textrm{($n$ is odd}), 
 \end{cases}
\end{align*}
where $SA_{8N}$ is the finite group of order $8N$ defined by 
$$SA_{8N}=\langle \ b,\ c\ \vert \ b^2=c^{4N}=1,\ cb=bc^{2N+1}\ \rangle .$$
\end{lem}
\begin{proof} 
Let $\omega $ be a primitive $4nN$-th root of unity. 
\par 
(1) If $n$ is even, then for fixed integers $i,j,k$ there is an algebra map $\chi _{ijk}:A_{Nn}^{++}\longrightarrow \boldsymbol{k}$ such that 
$\chi _{ijk}(t)=(-1)^i,\ \chi _{ijk}(w)=(-1)^j,\ \chi_{ijk}(h)=\omega ^{2nk}$. 
Since any algebra map $A_{Nn}^{++}\longrightarrow \boldsymbol{k}$ coincides with some $\chi _{ijk}$,  
we have $G((A_{Nn}^{++})^{\ast})
=\{ \ \chi _{ijk}\ \vert \ i,j=0,1,\ k=0,1,\ldots ,2N-1\ \} $. 
Furthermore, since the product of $G((A_{Nn}^{++})^{\ast})$ is given by 
$$\chi _{ijk}\chi_{i^{\prime}j^{\prime}k^{\prime}}=\chi _{i+i^{\prime}+(j+k)k^{\prime}, j+j^{\prime}, k+k^{\prime}}\qquad (i,i^{\prime}, j,j^{\prime}, k,k^{\prime}\in \mathbb{Z}),$$
$a:=\chi _{100},\ b:=\chi_{010},\ c:=\chi_{001}$ satisfy the equations $a^2=b^2=1,\ c^{2N}=a,\ cb=bc^{2N+1}$. 
Thus we have $G((A_{Nn}^{++})^{\ast})
=\langle \ b,\ c\ \vert \ b^2=c^{4N}=1,\ cb=bc^{2N+1}\ \rangle 
=SA_{8N}$. 
\par 
If $n$ is odd, then for fixed integers 
$i,k$ there is an algebra map $\chi _{ik}:A_{Nn}^{++}\longrightarrow \boldsymbol{k}$ such that $\chi _{ik}(t)=(-1)^i,\ \chi _{ik}(w)=(-1)^k,\ \chi_{ik}(h)=\omega ^{2nk}$. 
Since any algebra map $A_{Nn}^{++}\longrightarrow \boldsymbol{k}$ coincides with some $\chi _{ik}$,  
we have $G((A_{Nn}^{++})^{\ast})
=\{ \ \chi _{ik}\ \vert \ i=0,1,\ k=0,1,\ldots ,2N-1\ \} $. 
Furthermore, since the product of $G((A_{Nn}^{++})^{\ast})$ is given by 
$$\chi _{ik}\chi _{i^{\prime}k^{\prime}}=\chi _{i+i^{\prime}, k+k^{\prime}}\qquad (i,i^{\prime}, k,k^{\prime}\in \mathbb{Z}), $$
we see that 
$G((A_{Nn}^{++})^{\ast})\cong C_2\times C_{2N}$. 
\par 
(2) If $n$ is even, then as the same manner with the proof of Part (1) we see that 
$G((A_{Nn}^{+-})^{\ast})
=\{ \ \chi _{ijk}\ \vert \ i,j=0,1,\ k=0,2,\ldots ,2N-2\ \} $, 
where $\chi _{ijk}:A_{Nn}^{+-}\longrightarrow \boldsymbol{k}$ is the algebra map such that $\chi _{ijk}(t)=(-1)^i,\ \chi _{ijk}(w)=(-1)^j,\ \chi_{ijk}(h)=\omega ^{2nk}$. 
The product of $G((A_{Nn}^{+-})^{\ast})$ is given by $\chi _{ijk}\chi_{i^{\prime}j^{\prime}k^{\prime}}=\chi _{i+i^{\prime}, j+j^{\prime}, k+k^{\prime}}$ for all $i,i^{\prime}, j,j^{\prime}, k,k^{\prime}\in \mathbb{Z}$. 
Thus the group $G((A_{Nn}^{+-})^{\ast})$ is commutative. 
Since $a:=\chi _{100},\ b:=\chi_{010},\ c:=\chi_{002}$ satisfy the equations $a^2=b^2=1,\ c^{N}=1$, we see that 
$G((A_{Nn}^{+-})^{\ast})=C_2\times C_2\times C_N$. 
\par 
If $n$ is odd, then, $G((A_{Nn}^{+-})^{\ast})
=\{ \ \chi _{ik}\ \vert \ i=0,1,\ k=0,2,\ldots ,2N-2\ \} $, where 
$\chi _{ik}:A_{Nn}^{+-}\longrightarrow \boldsymbol{k}$ is the algebra map such that $\chi _{ik}(t)=(-1)^i,\ \chi _{ik}(w)=1,\ \chi_{ik}(h)=\omega ^{2nk}$. 
Since the product of $G((A_{Nn}^{+-})^{\ast})$ is given by 
$\chi _{ik}\chi _{i^{\prime}k^{\prime}}=\chi _{i+i^{\prime}+kk^{\prime}, k+k^{\prime}}$ for all $i,i^{\prime}, k,k^{\prime}\in \mathbb{Z})$, 
we see that the group $G((A_{Nn}^{+-})^{\ast})$ is commutative. 
Furthermore, $a:=\chi _{10}$ and $b:=\chi_{01}$ satisfy the equations $a^2=1,\  b^{2N}=a$, and hence $G((A_{Nn}^{+-})^{\ast})=C_{4N}$. 
\end{proof}

\par \medskip 
If a semisimple Hopf algebra $A$ possesses a quasitriangular structure, the representation ring needs to be commutative. 
In the case where $N\geq 1$ is odd, and $n$ is even, by Lemma~\ref{6.11} 
the representation ring of the dual Hopf algebra $(A_{Nn}^{++})^{\ast}$ are not  commutative, and therefore there is no quasitriangular structure of $(A_{Nn}^{++})^{\ast}$. 
\par 
By Lemma~\ref{6.10} and Lemma~\ref{6.11} we have: 

\par \medskip 
\begin{prop}\label{6.12}
If $n$ is even, then the Hopf algebra 
$A_{Nn}^{++}$ is not self-dual. 
\end{prop}
\begin{proof}
The group $G(A_{Nn}^{++})\cong C_2\times C_{2N}$  is commutative by Lemma~\ref{6.10}, meanwhile 
$G((A_{Nn}^{++})^{\ast})\cong SA_{8N}$ is not by Lemma~\ref{6.11}. 
This implies that $G(A_{Nn}^{++})\not\cong G((A_{Nn}^{++})^{\ast})$, and $A_{Nn}^{++}\not\cong (A_{Nn}^{++})^{\ast}$. 
\end{proof} 

\par \medskip 
\begin{cor}
If $n$ is even, then all braidings of $A_{Nn}^{++}$ degenerate. 
\end{cor}
\begin{proof}
Assume that there is a non-degenerate braiding of $A_{Nn}^{++}$. 
Then there is a Hopf algebra isomorphism $F:A_{Nn}^{++}\longrightarrow ((A_{Nn}^{++})^{\textrm{cop}})^{\ast}$. 
Let us consider the Hopf algebra isomorphism $\phi :A_{Nn}^{++}\longrightarrow (A_{Nn}^{++})^{\textrm{cop}}$ defined by $\phi (x_{ij})=x_{ji} \ (i,j=1,2)$. 
Then the composition ${}^t\kern-0.2em \phi \circ F: A_{Nn}^{++}\longrightarrow (A_{Nn}^{++})^{\ast}$ is also a Hopf algebra isomorphism. This contradicts Proposition~\ref{6.12}. 
\end{proof}

\par \bigskip \noindent 
{\bf \large Acknowledgements}
\bigskip \par \ 
The author would like to thank Professor Akira Masuoka and Professor Ikuo Satake for helpful advice and discussions. 
Thanks are also due to Professor Juan Cuadra for his suggestions to improve my  preliminary version of the proof of Lemma~\ref{5.14}. 
This research is partially supported by Grant-in-Aid for Scientific Research (No. 19540025), JSPS.

\end{document}